\documentclass[aos]{imsart}
\usepackage{multirow}

\RequirePackage[OT1]{fontenc}
\RequirePackage{amsthm,amsmath}
\usepackage{mathtools}
\RequirePackage[sectionbib]{natbib}
\RequirePackage[colorlinks,citecolor=blue,urlcolor=blue]{hyperref}
\usepackage[ruled,vlined,linesnumbered,resetcount]{algorithm2e}
\usepackage{xparse}
\usepackage{booktabs}
\startlocaldefs
\numberwithin{equation}{section}
\theoremstyle{plain}
\newtheorem{theorem}{Theorem}[section]
\newtheorem{lemma}{Lemma}[section]
\endlocaldefs

\startlocaldefs
\theoremstyle{remark}
\newtheorem{remark}{Remark}

\newtheorem{assumption}{Assumption}

\endlocaldefs

\usepackage{dsfont}
\usepackage{float}
\usepackage{amsmath,amssymb}
\usepackage{graphicx}
\usepackage{enumitem}
\usepackage{chngcntr}
\usepackage{apptools}
\AtAppendix{\counterwithin{lemma}{section}}

\allowdisplaybreaks

\def\g{g^*_\pi}
\def\gn{\hat g_{n}^\pi}

\def\Var{\text{Var}}

\def\eff{\text{eff}}

\def\S{\mathcal{S}}

\def\EE{\mathbb{E}}
\def\tv{\operatorname{tv}}
\def\PP{\mathbb{P}}

\def\G{\boldsymbol{G}}
\def\F{\mathcal{F}}
\def\qn{\hat H^{\pi}_n}
\def\qpi{\tilde{H}^\pi}

\def\goes{\rightarrow}

\def\f{\boldsymbol{f}}
\def\A{\mathcal{A}}
\def\D{\mathcal{D}}

\def\R{\mathbb{R}}
\def\Pn{\mathbb{P}_n}

\def\Upin{\hat {U}^\pi_n}
\def\Upi{{U}^\pi}

\def\P{\mathcal{P}}

\def\RR{\mathcal{R}}

\def\leqconst{\lesssim}

\newcommand{\pushright}[1]{\ifmeasuring@#1\else\omit\hfill$\displaystyle#1$\fi\ignorespaces}
\newcommand{\innerprod}[2]{\langle#1, #2 \rangle}

\newcommand{\norm}[1]{\|#1\|}
\newcommand{\bignorm}[1]{\big \|#1 \big \|}

\newcommand{\abs}[1]{|#1|}
\newcommand{\dkh}[1]{\{#1\}}
\newcommand{\bigdkh}[1]{\big\{#1 \big\}}
\newcommand{\Bigdkh}[1]{\Big\{#1\Big\}}
\newcommand{\bigabs}[1]{\big|#1\big|}
\newcommand{\Bigabs}[1]{\Big|#1\Big|}
\newcommand{\dv}[1]{\frac{d}{d#1}}

\newcommand{\indicator}[1]{ \mathds{1}_{\{#1\}}}

\newcommand{\mynu}{\nu}

\def\Var{\text{Var}}
\def\EE{\mathbb{E}}

\def\f{\boldsymbol{f}}

\def\R{\mathbb{R}}
\def\Regret{\operatorname{Regret}}

\def\Rem{\operatorname{Rem}}

\def\I{\,\mathcal{I}}
\def\ones{\mathbf{1}}

\def\B{\mathcal{B}}
\def\argmin{\operatorname{argmin}}
\def\argminb{\operatorname*{argmin}}
\def\argmax{\operatorname{argmax}}

\def\wcvg{\Rightarrow}

\def\X{\mathcal{X}}

\def\jiao{\cap}

\def\C{\mathcal{C}}

\def\given{\, | \,}
\def\Given{\, \Big| \,}
\def\GG{\mathbb{G}}
\def\Gn{\mathbb{G}_n}

\def\transpose{\top}

\renewcommand{\G}{\mathcal{G}}

\def\V{\eta}
\def\Vn{\hat{\eta}_n}

\usepackage{tcolorbox}
\usepackage{mathrsfs}

\begin{document}
	
	\begin{frontmatter}
		\title{Batch Policy Learning in Average Reward Markov Decision Processes}
		\runtitle{Batch Policy Learning in Average Reward MDP}
		%\thankstext{T1}{Footnote to the title with the ``thankstext'' command.}

		\begin{aug}
%%%%%%%%%%%%%%%%%%%%%%%%%%%%%%%%%%%%%%%%%%%%%%%
%% Only one address is permitted per author. %%
%% Only division, organization and e-mail is %%
%% included in the address.                  %%
%% Additional information can be included in %%
\author[A]{\fnms{Peng}~\snm{Liao}\ead[label=e1]{pliao9122@gmail.com}}\footnote[1]{The first two authors contributed equally.},
\author[B]{\fnms{Zhengling}~\snm{Qi}\ead[label=e2]{qizhengling@gwu.edu}}
\and
\author[C]{\fnms{Runzhe}~\snm{Wan}\ead[label=e3]{runzhe.wan@gmail.com}}\footnote[2]{This work started prior to joining Amazon.}
\author[D]{\fnms{Predrag}~\snm{Klasnja}\ead[label=e4]{klasnja@umich.edu}}
\author[A]{\fnms{Susan}~\snm{A. Murphy}\ead[label=e5]{samurphy11@gmail.com}}

%%%%%%%%%%%%%%%%%%%%%%%%%%%%%%%%%%%%%%%%%%%%%%
%% Addresses                                %%
%%%%%%%%%%%%%%%%%%%%%%%%%%%%%%%%%%%%%%%%%%%%%%
\address[A]{Harvard University\printead[presep={,\ }]{e1,e5}}

\address[B]{George Washington University\printead[presep={,\ }]{e2}}
\address[C]{Amazon\printead[presep={,\ }]{e3}}
\address[D]{University of Michigan, Ann Arbor\printead[presep={,\ }]{e4}}
\end{aug}
		
		\begin{abstract}
			We consider the batch (off-line) policy learning problem in the infinite horizon Markov Decision Process. Motivated by mobile health applications, we focus on learning a policy that  maximizes the long-term average reward. We propose a doubly robust estimator for the average reward and show that it achieves semiparametric efficiency.
			Further we develop an optimization algorithm to compute the optimal policy in a parameterized stochastic policy class. The performance of the estimated policy is measured by the difference between the optimal average reward in the policy class and the average reward of the estimated policy and we establish a finite-sample regret guarantee. 
			% suggest to put in text but not in abstract: To  our knowledge, this is the first regret bound for batch policy learning in the infinite time horizon setting. 
			The performance of the method is illustrated by simulation studies and an analysis of a mobile health study promoting physical activity.
		\end{abstract}

		
	\end{frontmatter}

	\section{Introduction}

    Mobile health (mHealth) is a rapidly growing field due to the recent advances in mobile and sensing technologies.  The mHealth intervention provides a unique opportunity to promote the healthy behaviors (e.g., regular physical activity and adherence to medications) and has been successfully applied in many health fields (e.g., smoking cessation, physical activity, drug abuse and diabetes).  Just-in-time adaptive interventions (JITAI, \cite{nahum2016just}) use  a decision rule (i.e., a treatment policy or policy) that maps  real-time information about the individual's context to a particular treatment. In this work we study the problem of how to use data consisting of multiple trajectories to estimate a policy   that leads to good long-term performance.

     We model the sequential decision making process by a time-homogeneous Markov Decision Process (MDP) \citep{puterman1994markov} over infinite time horizon.
     This framework is natural for mobile health applications in which the number of decision times is often large. For example, in \textit{HeartSteps}, a physical activity mHealth study, there are five decision times per day, resulting in thousands of decision times over a year. Tremendous progress has been made in finite horizon setting; see the recent review  by \cite{kosorok2019precision} for references therein.  However when the number of time points is very large, methods that are based on the idea of backward iteration (e.g., Q-learning) or importance sampling \citep{precup2000eligibility} may suffer a large variance in problems or even be unpractical 
     \citep{voloshin2019empirical,laber2014dynamic}.    
     
     We propose to estimate  the policy that optimizes the long-term average outcomes (rewards) using data    consisting of multiple trajectories of finite length. 
     The majority of existing methods focuses on the alternative, the discounted sum of rewards \citep{sutton2018reinforcement}; see the recent works in statistics \citep{luckett2019estimating,ertefaie2018constructing,shi2020statistical,shi2021deeply}. The discounted formulation weighs immediate rewards more heavily than rewards further in the future, which is practical in some applications (e.g., finance). However, for mHealth applications, choosing an appropriate discount rate could be non-trivial. The rewards (i.e., the health outcomes) in the distant future are as important as the near-term ones, especially when considering maintenance of health behaviors as well as longer term treatment burden. This suggests using a large discount rate. However, it is well known that algorithms developed in the discounted setting can become increasingly unstable as the discount rate goes to one; see for example \cite{naik2019discounted}.  The long-term average reward framework provides a good approximation to the long-term performance of a desired treatment policy in mHealth. Indeed, it can be shown that under regularity conditions the finite average of the expected rewards converges sublinearly to the long-term average reward as time goes to infinity \citep{hernandez2012further}. Therefore, a policy that optimizes the average reward would approximately maximize the sum of the rewards over a sufficiently long time horizon.

    In this work, we present a novel algorithm that estimates the optimal policy in a prespecified, parametric policy class.  Various methods have been proposed to estimate the global optimal policy by  estimating the optimal Q-function;see for example \cite{ormoneit_kernel-based_2003,lagoudakis2003least,ernst_tree-based_2005,munos2008finite,antos_learning_2008,antos2008fitted,ertefaie2018constructing,fujimoto2019off,kumar2019stabilizing,agarwal2020optimistic}. In practice, the optimal Q-function could be highly non-smooth and  complex, thus requiring the use of a flexible function class. This usually results in a learned policy that is also complex.  If interpretability is important, this is problematic.   Furthermore, when the training data is limited, the flexible function class might overfit the data and thus the variance of the estimated value function and the corresponding policy could be high. Restricting to a pre-specified policy class was studied by \cite{zhang2012robust,zhang2013robust,zhou2017residual,zhao2015new,zhao2019efficient,athey2017efficient} in finite time horizon problems and by \cite{luckett2019estimating,murphy2016batch,liu2019off} in infinite time horizon problems.  The restriction to a simple policy class enhance the interpretability of the learned policy and reduces the variance of the learned policy, although this induces a bias when the optimal policy is not in the class (i.e., trading off the bias and variance).

		To efficiently learn an optimal policy in a prespecified policy class, the main statistical challenge is to construct an estimator for the average reward of a  policy  that is both data-efficient and performs uniformly well when optimizing over the policy class. Our first contribution of this work is a novel doubly robust estimator (see Section \ref{sec:method}); we  show that this estimator achieves the semiparametric efficiency bound under certain conditions on the estimation error of nuisance functions (see Section \ref{sec:theorem}).  Doubly robust estimators have been developed in the finite time horizon problems \citep{robins1994estimation,murphy2001marginal,dudik2014doubly,jiang2016doubly,thomas2016data} and recently in the discounted reward infinite horizon setting \citep{kallus2019double,Tang2020Doubly}.   To the best of our knowledge, our doubly robust estimator for the long-term average reward is new.  \textcolor{black}{In the literature of the average MDP in the batch setting, only the non-doubly robust estimator proposed by \cite{liao2019off} for the long-term average reward can be shown to achieve the semi-parametric efficiency, although they did not explicitly derive it. Most of the previous works on the policy optimization/evaluation under this framework are focused the online setting or under parametric models \citep[e.g.,][]{mahadevan1996average,abounadi2001learning,wan2021learning}. Theoretical studies on the average reward MDP in the batch setting are very limited, especially under non-parametric models.}

	To establish the semiparametric efficiency of the doubly robust estimator and
	the regret bound, we derive  finite-sample error bounds for two nuisance function estimators, a relative value estimator and a ratio estimator.  The obtained error bounds are shown to hold uniformly over the prespecified class of policies.
	Both the relative value and ratio estimators are both derived from the same principle (i.e., coupled estimation; see Section \ref{sec:nuisance est}).  
	In the case of the ratio estimator, we use an iterative procedure to obtain a near-optimal error bound for the ratio estimator.  To the best of our knowledge, this is the first theoretical result characterizing the ratio estimation error, which might be of independent interest.

	We learn the optimal policy by maximizing the estimated average reward over a policy class and derive a finite-sample upper bound of the regret.  We show that the our proposed method achieves $O(p^{1/2}n^{-1/2} + pn^{-\tilde \beta})$ regret, where $p$ is the number of parameters in the policy, $n$ is the number of trajectories in the training data and $\tilde \beta$ is a constant that can be chosen arbitrarily close to $1/(1+\alpha)$. Here $\alpha \in (0, 1)$ measures the complexity of nuisance function classes.
		The use of doubly robust estimation ensures the estimation error for the nuisance functions contributes only lower order terms to the regret.
     Unlike the setting in which the goal is to maximize the average reward defined over a finite horizon \citep{athey2017efficient}, when the goal is to maximize the average reward defined over an infinite horizon, 
     the regret analysis 
     requires  uniform control of the estimation error of the policy-dependent nuisance functions over the policy class.
		We believe this is the first regret bound result  for an estimator of an in-class optimal policy in the average reward MDP and  using batch data. \textcolor{black}{Recently, \cite{sharma2020approximate} proposed an approximate relative value learning algorithm for globally optimal policies under the average reward MDP with non-parametric function approximation. However, they require the sample size at least exponentially larger than the dimension of the state for the convergence of their algorithm, which seems sub-optimal compared with our result, e.g., Theorem \ref{thm: policy learning}.}

	The rest of the article is organized as follows. Section \ref{sec:avg rwrd mdp} formalizes the decision making problem and introduces the average reward MDP. Section \ref{sec:method} presents the proposed method of learning the in-class optimal policy, including the doubly robust estimator for average reward (Section \ref{subsec:doubly est}). In Section \ref{sec:nuisance est}, the coupled estimators of the policy-dependent nuisance functions are introduced. Section \ref{sec:theorem} provides a thorough theoretical analysis on the regret bound of our proposed method. In Section \ref{sec:implemenation}, we describe a practical optimization algorithm when Reproducing Kernel Hilbert Spaces (RKHSs) are used to model the nuisance functions. We further conduct several simulation studies to demonstrate the promising performance of our method in Section \ref{sec:simulation}. All the technical proofs are postponed to the supplementary material.

	\section{Problem Setup}
	\label{sec:avg rwrd mdp}

	Suppose we observe a training dataset, $\D_n = \{D_i\}_{i=1}^n$ that consists of $n$ independent, identically distributed (i.i.d.) observations of $D$: 
	$$D = \{S_1, A_1, S_2, \dots, S_T, A_T, S_{T+1}\}.$$ 
	We use $t$ to index the decision time. The length of the trajectory, $T$, is a fixed constant. 	$S_{t} \in \S $ is the state at time $t$  and $A_t \in \A$ is the action (treatment) selected at time $t$. We assume the action space, $\A$, is finite. To eliminate unnecessary technical distractions, we assume that the state space, $\S$, is finite; this assumption imposes no  practical limitations and can be extended to the general state space.

	The states evolve according to a time-homogeneous Markov process, that is, for $t \geq 1$, 
	$
	S_{t+1} \perp \{S_1, A_1, \dots, S_{t-1}, A_{t-1}\} \given \{S_t, A_t\}
	$,
	and the conditional distribution does not depend on $t$. Denote the conditional distribution by $P$, i.e., $\Pr(S_{t+1}=s'|S_t = s, A_t = a) = P(s'|s, a)$.  The reward (i.e., outcome) is denoted by $R_{t+1}$, which is assumed to be a known function of  $(S_{t}, A_t, S_{t+1})$, i.e., $R_{t+1} = \mathcal{R}(S_t, A_t, S_{t+1})$. We assume the reward is bounded, i.e., $\abs{\RR(s, a, s')} \leq R_{\max}$. 
	We use $r(s, a)$ to denote the conditional expectation of reward given state and action, i.e., $r(s, a) = \EE\left[R_{t+1}|S_t = s, A_t = a\right]$.  
	
	Let $H_t = \{S_1, A_1, \dots, S_t\}$ be the history up to time $(t-1)$ and the current state, $S_t$. Denote the conditional distribution of $A_t$ given $H_t$ by $\pi_{b,t}(a|H_t) = \Pr(A_t = a|H_t)$. Let $\pi_b=\{\pi_{b,1},\ldots,\pi_{b,T}\}$. This is often called behavior policy in the literature. 
	Throughout this paper, the expectation, $\EE$, without any subscript is assumed taken with respect to the distribution of the trajectory, $D$, with the actions selected by the behavior policy $\pi_b$.

	Consider a time-stationary, Markovian policy, $\pi$, that takes the  state as input and outputs a probability distribution on the action space, $\A$, that is, $\pi(a|s)$ is the probability of selecting action, $a$, at state, $s$.  The average reward of the policy, $\pi$, is defined as
	\begin{align}
	\V^\pi(s) :=  \operatorname*{lim}_{t^*\goes \infty} \EE_{\pi} \left(\frac{1}{t^*} \sum_{t=1}^{t^*} R_{t+1} \Given S_1
	= s\right),
	\label{def: avg rwrd}
	\end{align}
	where the expectation, $\EE_{\pi}$, is with respect to the distribution of the trajectory in which the states evolve according to $P$ and the actions are chosen by $\pi$. In the time-homogeneous MDP with finite state and bounded reward,  the limit in (\ref{def: avg rwrd}) always exists \citep{puterman1994markov}.
	The policy, $\pi$, induces a Markov chain of states with the transition as $P^\pi(s'|s) = \sum_{a} \pi(a|s) P(s'|s, a)$.  When the induced Markov chain, $P^\pi$, is irreducible,
	it can be shown (e.g., in \cite{puterman1994markov}) that the stationary distribution of $P^\pi$ exists and is unique (denoted by $d^\pi$) and the average reward, $\V^\pi(s)$ (\ref{def: avg rwrd}) is independent of initial state (denoted by $\V^{\pi}$) and equal to 
	\begin{align}
	\V^\pi(s) = \V^{\pi} = \sum_{s, a} r(s, a)  \pi(a|s)  d^\pi(s). \label{constant avg reward}
	\end{align}
	
	Throughout this paper we consider only the time-stationary, Markovian policies. In fact,  it can be shown that the maximal average reward among all possible history dependent policies can be in fact achieved by some time-stationary, Markovian policy (Theorem 8.1.2 in \cite{puterman1994markov}). Consider a pre-specified class of such policies, $\Pi$, that is parameterized by $\theta \in \Theta \subset \R^p$.   Throughout we assume that the induced Markov chain is always irreducible for any policy in the class, which is summarized below. 
	
	\begin{assumption}
		\label{assumption: irreducible}
		The induced Markov chain, $P^\pi$, is irreducible for $\pi \in \Pi$. 
	\end{assumption}

	The goal of this paper is to develop a method that can efficiently use the training data, $\D_n$, to learn a policy that maximizes the average reward over $\Pi$. 
	We propose to construct $\Vn^{\pi}$, an efficient estimator for the average reward, $\V^{\pi}$, for each 
	policy $\pi \in \Pi$ and  learn an optimal policy by solving
	\begin{align}
	\hat \pi_n \in \argmax_{\pi \in \Pi} \Vn^{\pi}. \label{est opt policy}
	\end{align}
	The performance of $\hat \pi_n$ is measured by its regret, defined as
	\begin{align}
	\Regret(\hat \pi_n) = \sup_{\pi \in \Pi} \V^{\pi} - \V^{\hat \pi_n}.  \label{regret defnition}
	\end{align}
	Note that although the average reward of the learned policy,  $\hat \pi_n$, is defined over an infinite horizon, the goal here is to characterize the regret based on  using a finite number of trajectories, $n$, hence the finite sample regret bound is in terms of $n$. Indeed while the average reward, $\eta^\pi$ is defined as $t^* \rightarrow \infty$ (\ref{def: avg rwrd}), the Markovian and stationary assumptions allow us to estimate $\eta^\pi$ using fixed length  trajectories.
	
	\section{Doubly Robust Estimator for Average Reward}
	\label{sec:method}

	In this section we present a doubly robust estimator for the average reward for a given policy. The estimator is derived from the efficient influence function (EIF). Below we  first introduce two functions that occur in the EIF of the average reward. Throughout this section we fix a time-stationary Markovian policy, $\pi$, and focus on the setting where the induced Markov chain, $P^\pi$, is irreducible (Assumption \ref{assumption: irreducible}).

	\subsection{Relative value and ratio functions}

	First, we define the relative value function by
	\begin{align}
	Q^\pi(s, a) := 
	\operatorname*{lim}_{t^* \goes \infty} \frac{1}{t^*}\sum_{t=1}^{t^*} \EE_\pi\left[\sum_{k=1}^{t} (R_{k+1} - \V^{\pi}) \Given S_1 = s, A_1 = a\right]. \label{value function. simplified}
	\end{align}
	The above limit is well-defined (\cite{puterman1994markov},  p. 338).  If we further assume the induced Markov chain is aperiodic, then the Ces\`aro limit in (\ref{value function. simplified}) can be replaced by $Q^\pi(s, a)  
	= \EE_\pi\{\sum_{t=1}^{\infty} (R_{t+1} - \V^{\pi}) \given S_1 = s, A_1 = a\}$.   $Q^\pi$ is often called relative value function in that $Q^\pi(s, a)$ represents the expected total difference between the reward and the average reward under the policy, $\pi$, when starting at state, $s$, and action, $a$. 
	
	The relative value function, $Q^\pi$, and the average reward, $\V^{\pi}$, are closely related via the Bellman equation:
	\begin{align}
	\EE_\pi[R_{t+1} + Q(S_{t+1}, A_{t+1})\given S_t = s, A_t = a] - Q(s, a) - \eta = 0.
	\label{Bellman equationL eta,V}
	\end{align} 
	Note that in the above expectation $A_{t+1} \sim \pi(\cdot \given S_{t+1})$. 
	It is known that under the irreducibility assumption, the set of solutions of (\ref{Bellman equationL eta,V}) is given by $\{(\V^{\pi}, Q): Q = Q^\pi + c \ones, c \in \R\}$ where $\ones(s, a) = 1$ for all $(s, a)$;  see \cite{puterman1994markov}, p. 343 for details. As we will see in Section \ref{subsec: value estimator}, the Bellman equation provides the foundation of estimating the relative value function.  
	
	We now introduce the ratio function. For $t=1,\dots, T$, let $d_t(s, a)$ be the probability mass of state-action pair at time $t$ in the trajectory $D$ generated by the behavior policy. Denote by $d_D(s, a) := (1/T) \sum_{t=1}^T d_t(s, a)$ the average probability mass across the $T$ decision times in $D$.   Similarly, define $d_t(s)$ as the marginal distribution of $S_t$ and $d_D(s) = (1/T) \sum_{t=1}^T d_t(s)$ as the average distribution of states in the trajectory $D$. Recall that $T$ is the fixed  length of the trajectory, $D$; $d_D$  describes the  distribution of this finite length trajectory.  
	Further recall that under Assumption \ref{assumption: irreducible}, the stationary distribution of $P^\pi$ exists and is denoted by $d^\pi(s)$.  {We assume the following conditions on the data-generating process.} 
	\begin{assumption}
		\label{assumption: data-gen}
		The data-generating process satisfies:
		\begin{enumerate}[label={(2-\arabic*)}]
			
			\item \label{a1}
			There exists some $p_{\min} > 0$, such that $\pi_{b, t} (a|H_{t}) \geq p_{\min}$ for all $a \in \A$ and $1 \leq t \leq T$ almost surely. 
			
			\item \label{a0}
			The average distribution $d_D(s) > 0$ for all $s \in \S$. 
		\end{enumerate}

	\end{assumption}
	
	Under Assumption \ref{assumption: data-gen}, it is easy to see that $d_D(s, a) \geq p_{\min}  \cdot \left(\min_{s} d_D(s)\right) > 0$ for all state-action pair, $(s, a)$. It essentially states that the data generating process ensures that every state-action pair $(s, a) \in \S \times \A$ has a positive probability of being visited, which is a standard assumption in the literature. See e.g., Theorem 7 of \cite{kallus2019efficiently} and (A2) of \cite{shi2020statistical}. In particular, Assumption \ref{a1} is often satisfied in randomized trials. See our mobile health application in Section \ref{sec:data analysis}. \textcolor{black}{In addition, the batch data of our mobile health application consist of $37$ trajectories with $210$ decision points on each trajectory. In this application, as long as every state has a positive probability of being visited in at least one of $210$ decision points, Assumption \ref{a0} is also satisfied. Assumption \ref{a0} is imposed on the data generating process. We essentially require that across an infinite number of draws from this data generating process/trajectory, every state $s \in \mathcal{S}$ will be observed.}
	Note that Assumption \ref{assumption: data-gen} does not require that the form of the behavior policy is known.  Now we can define the ratio function:
	\begin{align}
	\omega^\pi(s, a) = \frac{d^\pi(s) \pi(a|s)}{d_D(s, a)} \label{ratio}
	\end{align}
	The ratio function plays a similar role as the importance weight in finite horizon problems. While the classic importance weight only corrects the distribution of actions between behavior policy and target policy,  the ratio here also involves the correction of the states' distribution.  The ratio function is connected with the average reward by  
	\[
	\V^{\pi} = \EE\left\{  \frac{1}{T}\sum_{t=1}^T \omega^\pi(S_t, A_t) R_{t+1}\right\}
	\]
	for any fixed trajectory length, $T$. 
	An important property of $\omega^\pi$ is that for any state-action function $f(s, a)$ (not only $Q^\pi$), 
	\begin{align}
	\EE\left[
	\frac{1}{T}\sum_{t=1}^{T} \omega^\pi(S_t, A_t) \Big\{f(S_t, A_t) - \sum_{a'} \pi(a'|S_{t+1})  f(S_{t+1}, a') \Big\}
	\right] = 0. \label{orthogonal}
	\end{align}
	This orthogonality is key to develop the estimator for $\omega^\pi$ (see Section \ref{subsec: ratio estimator}).

	\subsection{Efficient influence function}
	In this subsection, we derive the EIF of $\V^{\pi}$ for a fixed policy $\pi$ under time-homogeneous Markov Decision Process described in Section \ref{sec:avg rwrd mdp}. 
	Recall that the semiparametric efficiency bound is the supremum of the Cram\`er-Rao bounds for all parametric submodels \citep{newey1990semiparametric}. EIF is defined as the influence function of a regular estimator that achieves the semiparametric efficiency bound.  
	For more details, refer to \cite{bickel1993efficient} and \cite{van2000asymptotic}. The EIF of $\V^{\pi}$ is given by the following theorem. The proof is provided in Appendix A. 
	
	\begin{theorem} \label{thm:semipara} 
		Suppose the states in the trajectory, $D$, evolve according to the time-homogeneous Markov process and Assumption \ref{assumption: data-gen} holds.  Consider a policy, $\pi$, such that Assumption \ref{assumption: irreducible} holds. Then the EIF of the average reward, $\V^{\pi}$, is
		$$
		\phi^\pi (D) = \frac{1}{T} \sum_{t=1}^T \omega^\pi(S_t, A_t) \left\{R_{t+1} +  U^\pi(S_t, A_t, S_{t+1})-  \V^{\pi} \right\}
		$$
	where 
	\begin{align}
	U^\pi(s, a, s') := \sum_{a'} \pi(a'|s') Q^\pi(s', a') - Q^\pi(s, a). \label{Upi}
	\end{align}
	\end{theorem}

	\begin{remark}
	Recall that we impose the Markovian and time-stationary assumptions on the data-generating process. Even though the parameter of interest here (i.e. the average reward $\eta^\pi$) is one-dimensional, there may exist multiple, non-efficient influence functions as a result of these assumptions on the multivariate distribution. 
	\end{remark}
	\subsection{Doubly robust estimator}
	\label{subsec:doubly est}

	It is known that EIF can be used to derive a semiparametric estimator (see, for example, Chap. 25 in \cite{van2000asymptotic}). We follow this approach.  Specifically, suppose $\hat U_n^\pi$ and $\hat \omega_n^\pi$ are estimators of $U^\pi$ and $\omega^\pi$ respectively. Then we estimate $\V^{\pi}$ by solving for $\eta$ in the plug-in estimating equation: $\Pn \dkh{(1/T)\sum_{t=1}^T \hat \omega_n^\pi(S_{t}, A_{t}) [R_{t+1} +\hat U^{\pi}_n(S_t, A_t, S_{t+1})  - \eta]} = 0$, where for any function of the trajectory, $f(D)$, the sample average is denoted as $\Pn f(D) = (1/n)\sum_{i=1}^{n} f(D_i)$. 
	The solution,  $\Vn^{\pi}$, is
	\begin{align}
	\Vn^{\pi} = \frac{\Pn \left[(1/T)\sum_{t=1}^T \hat \omega_n^\pi(S_{t}, A_{t}) \left\{R_{t+1} + \hat U^{\pi}_n(S_t, A_t, S_{t+1})\right\} \right] }{\Pn \left\{(1/T)\sum_{t=1}^T \hat \omega_n^\pi(S_{t}, A_{t}) \right\} }. \label{pol.eval}
	\end{align}
	
	We have the following doubly robustness of this estimator (the proof is given in Appendix A).
	
	\begin{theorem}\label{lemma: doubly-rob} 
		Suppose $\hat{U}_n^\pi(s, a)$ and $\hat{\omega}_n^\pi(s, a)$ converge in probability to deterministic limits $\bar {U}^\pi(s, a)$ and $\bar \omega^\pi(s, a)$ uniformly over $\S \times \A$. 
		If either $\bar U^\pi= U^\pi$ or $\bar{\omega}^\pi= \omega^\pi$, then $\Vn^{\pi}$ converges to $\V^{\pi}$ in probability. 
		
	\end{theorem}
	\begin{remark}
		The uniform convergence in probability can be relaxed to $L_2$ convergence by using uniform laws of large numbers.
		The doubly robustness can protect against potential model mis-specifications since we only require one of two models is correct. Moreover, the doubly robust structure can be used to relax the required rate for each of the nuisance function estimation to achieve the semiparametric efficiency bound, especially if we use sample-splitting techniques (see Section \ref{sec: discussion}), as discussed in \cite{chernozhukov2018double}.
	\end{remark}

	\section{Estimators for the Nuisance Functions}
	\label{sec:nuisance est}

	Recall the doubly robust estimator (\ref{pol.eval}) requires the estimation of two nuisance functions, $U^\pi$ and $\omega^\pi$.  It turns out that although these two nuisance functions are defined from different perspectives, both nuisance functions can in fact be characterized in a similar way. Both estimators can be  obtained by minimizing an objective function that involves a minimizer of another objective function (``coupled estimation''). This can be viewed as a generalization of the classical M-estimator with  a ``plug-in estimator'' in the sense that the the second objective function also involves the unknown parameters to be estimated. 
	The idea of coupled estimation was previously used by \cite{antos_learning_2008,farahmand2016regularized} to estimate the value function in the discounted reward setting and recently by \cite{liao2019off} in the average reward setting.   In what follows we provide a general coupled estimation framework and discuss the motivation for using it. We then review the coupled estimator for relative value function and ratio function in \cite{liao2019off}.

	\subsection{Review of coupled estimation}
	\label{subsec: coupled estimation}

	Consider a setting where the true parameter (or function), $\theta^*$, can be characterized as the minimizer of the following objective function:
	
	\begin{align}
	\theta^* = \argmin_{\theta} J(\theta)= \EE \left\{(l_1\circ f_{\theta})(Z)\right\} 
	\label{obj in coupled estimator}
	\end{align}
	where $l_1: \R \goes \R^{+}$ is a loss function composite with $f$ (e.g., the squared loss, $l_1(x) = x^2$ and the linear model, $f_{\theta}(Z) = Y - \theta^\transpose X$, where $Z = (X, Y)$).
	If we can directly evaluate $f_{\theta}(Z)$, 
	then we can estimate $\theta^*$ by the classical M-estimator, $\argmin_{\theta} \Pn \left\{(l_1\circ f_{\theta})(Z)\right\}$. 
	
	In our setting 
	$f_{\theta}$ is of the form $f_{\theta}(Z) = \EE[F_{\theta}(Z')|Z]$
and $f_{\theta}(Z)$ cannot be directly evaluated because we don't have an explicit formula for the conditional expectation $\EE[F_{\theta}(Z')|Z]$. 
	A natural idea to remedy this is to replace the unknown $f_{\theta}(Z)$ by $F_{\theta}(Z')$ and estimate $\theta^*$ by $\argmin_{\theta} \Pn \left\{(l_1\circ F_{\theta}(Z')\right\}$.  Unfortunately this estimator is biased in general. To see this, suppose $l_1(x) = x^2$. We note that the limit of the new objective function, $\Pn \left\{(l_1\circ F_{\theta}(Z')\right\}$, is then $\tilde J(\theta) = \EE \left\{(l_1\circ F_{\theta}(Z')\right\} = J(\theta) + \Delta (\theta)$ where $\Delta(
	\theta)  = \EE[\Var\{F_{\theta}(Z')|Z\}]$. The minimizer of $\tilde J(\theta)$ is not necessarily $\theta^*$ unless further conditions are imposed (e.g., $\Var\{F_{\theta}(Z')|Z\}$ is independent of $\theta$, which is often not the case in our setting). 
	
	The high level idea of coupled estimation is to first estimate $f_{
		\theta}$ for each $\theta$, denoted by $\hat f_{\theta}$, and then estimate $\theta^*$ by the plug-in estimator,  $\argmin_{\theta} \Pn \{(l_1 \circ \hat f_{\theta})(Z)\}$. A standard empirical risk minimization can be applied to obtain a consistent estimator for $f_{\theta}$, e.g., $\hat f_{\theta} = \argmin_{g \in \G} \Pn [l_2\{(F_{\theta}(Z'), g(Z)\}]$ for some loss function $l_2: \R \times \R \goes \R^{+}$ and a function space, $\G$ to approximate $f_{\theta}$. We call the estimator coupled because the objective function (i.e., $\Pn \{(l_1 \circ \hat f_{\theta})(Z)\}$) involves $\hat f_{\theta}$ which itself is an minimizer of another objective function (i.e., $\Pn l_2((F_{\theta}(Z'), g(Z))$) for each $\theta$.

	\subsection{Relative value function estimator}
	\label{subsec: value estimator}
	
	Recall the doubly robust estimator requires an estimator of $U^\pi$. It is enough to learn one specific version of $Q^\pi$.  More specifically,  define a shifted value function by $\tilde Q^\pi(s, a) = Q^\pi(s, a) - Q^\pi(s^*, a^*)$ for some specific state-action pair $(s^*, a^*)$.  By restricting to $Q(s^*, a^*) =0$, the solution of Bellman equations (\ref{Bellman equationL eta,V}) is unique and given by $\{\V^{\pi}, \tilde Q^\pi\}$.  Below we derive a coupled estimator for the shifted value function, $\tilde Q^\pi$, using the coupled estimation framework in Section \ref{subsec: coupled estimation}.

	Let $Z_t = (S_t, A_t, S_{t+1})$ be the transition sample at time $t$. For a given $(\eta, Q)$ pair,  let \begin{align}\delta^\pi(Z_t; \eta, Q) = R_{t+1} + \sum_{a'} \pi(a'|S_{t+1}) Q(S_{t+1}, a) - Q(S_t, A_t) - \eta \label{TDErr}
	\end{align}
	be the so-called temporal difference (TD) error. The Bellman equation then becomes $\EE[\delta^\pi(Z_t; \eta, Q)|S_t=s, A_t = a] =0$ for all state-action pair, $(s, a)$. As a result, we have
	\begin{align*}
	\{\V^{\pi}, \tilde Q^\pi\} \in \argminb_{\eta, Q} \EE\left[\frac{1}{T}\sum_{t=1}^T\left(\EE[\delta^\pi(Z_t; \eta, Q)|S_t, A_t]\right)^2\right].\
	\end{align*}
	Note that above we choose the squared loss for simplicity; a general loss function can also be applied. We see that it fits in the coupled estimation framework presented in the previous section. In particular, $\theta^* = \{\V^{\pi}, \tilde Q^\pi\}$ and $f_{\theta}$ becomes the Bellman error, i.e., $\EE[\delta^\pi(Z_t; \eta, Q)|S_t=\cdot, A_t=\cdot]$.   The above characterization involves the average reward, $\V^{\pi}$. Thus in the process of obtaining an estimator of  the relative value function, we will also obtain an estimator of the average reward. See the end of this subsection for discussion.
	
	We use $\F$ and $\G$ to denote two classes of functions of state-action. We use $\F$ to model the shifted value function, $\tilde Q^\pi$, and thus require $f (s^*, a^*) = 0$ for all $f \in \F$.  We use $\G$ to approximate   the conditional mean of the Bellman error.   In addition, $J_1: \F \goes \R^+$ and $J_2: \G \goes \R^+$ are two regularizers that measure the complexities of these two functional classes respectively. Given  tuning parameters $(\lambda_n, \mu_n)$, the coupled estimator, denoted by $(\hat{\eta}^\pi_n, \hat Q^\pi_n)$, is obtained by solving
	\begin{align}
	(\tilde \eta_n^\pi, \hat Q_n^\pi) = \argminb_{(\eta, Q) \in \R \times \F} \Pn \left[\frac{1}{T} \sum_{t=1}^{T} \gn(S_t, A_t; \eta, Q)^2\right] + \lambda_n J_1^2(Q), \label{estimator}
	\end{align}
	where $\gn(\cdot, \cdot; \eta, Q)$ is the projected Bellman error at $(\eta, Q)$:
	\begin{align}
	\gn(\cdot, \cdot; \eta, Q) = \argminb_{g \in \G} \Pn\Big[ \frac{1}{T}\sum_{t=1}^{T} \big( \delta^\pi(Z_t; \eta, Q) - g(S_t, A_t)\big)^2  \Big] + \mu_n J_2^2(g). \hspace{-3ex} \label{gn value}
	\end{align}
	Given the estimator of the  (shifted) relative value function, $\hat Q_n^\pi$, we form the estimator of $U^\pi$ by $\hat U_n^\pi (s, a, s') = \sum_{a'} \pi(a'|s') \hat Q^\pi_n(s', a') -  \hat Q_n^\pi(s, a).$
	
	Throughout this paper, we use  tuning parameters, $(\lambda_n, \mu_n)$, that do not depend on the policy. {In the setting where the policy class is highly complex and the corresponding relative value functions are very different, it could be beneficial to select the tuning parameters locally at a cost of higher computation burden. }
	
	 Recall that the goal here is to estimate relative value function, $U^\pi$, and then plug $\hat U^\pi_n$ in the doubly robust estimator (\ref{pol.eval}). The above $\tilde \eta^\pi_n$ is only used to help estimate the relative function. In fact, \cite{liao2019off} proposed using $\tilde \eta^\pi_n$ to estimate the average reward. The advantage of our doubly robust estimator (\ref{pol.eval}), $\Vn^{\pi}$, over  $\tilde \eta^\pi_n$ is that the consistency of $\Vn^{\pi}$ is guaranteed as long as at least one of the nuisance function is estimated consistently (Theorem \ref{lemma: doubly-rob}).

	\subsection{Ratio function estimator}

	\label{subsec: ratio estimator}
	Below we derive the  estimator for the ratio function, $\omega^\pi$ using the coupled estimation framework.  In particular we estimate a scaled version of the ratio function (denoted by $e^\pi$ below) and then convert this back to an estimator of $\omega^\pi$.   To estimate $e^\pi$, we first  construct a new MDP and estimate the relative value function for this new MDP (denoted by $H^\pi$) using the   
	coupled estimation framework.   The estimator of $e^\pi$ is then derived from the estimator of $H^\pi$.
	
	We start with introducing $e^\pi$: 
	\begin{align}
	e^\pi(s, a) =  \frac{\omega^\pi(s, a)}{\sum_{\tilde s, \tilde a} \omega^\pi(\tilde s, \tilde a) d^\pi(\tilde s) \pi(\tilde a|\tilde s)}. \label{epi defn}
	\end{align}
	By definition, $\sum_{s, a} e^\pi(s, a) d^\pi(s) \pi(a|s) = 1$.  If we were to replace the reward function in our MDP by  $1-e^\pi(s, a)$, then   the ``average reward'' of $\pi$ in this new MDP is constant and equal to zero under Assumption \ref{assumption: irreducible} (i.e., $\sum_{s, a} \left\{1-e^\pi(s, a)\right\} d^\pi(s) \pi(a|s) = 0$). The  ``relative value function'' of policy $\pi$ under the new MDP is,
	\begin{align}
	H^{\pi} (s, a) = 	\operatorname*{lim}_{t^* \goes \infty} \frac{1}{t^*}\sum_{t=1}^{t^*} \EE_\pi\left[\sum_{k=1}^{t} \left\{1 -e^\pi(S_k, A_k)\right\} \Given S_1 = s, A_1 = a\right]. \label{qpi}
	\end{align}
	Note that $H^\pi$ is well-defined under Assumption \ref{assumption: irreducible}. Furthermore, consider the following Bellman equation for the new MDP: 
	\begin{align}
	\EE_\pi\left\{1 - e^\pi(S_t, A_t) +  H(S_{t+1}, A_{t+1}) \given S_t=s, A_t=a\right\} =  H(s, a). \label{epi qpi}
	\end{align}
	Note that since the ``average reward'' in the new MDP is zero, the above equation only involves $H$. The set of solutions of (\ref{epi qpi}) can be shown to be $\{H: H = H^\pi+ c\ones, c \in \R\}$. 
	
	Below we construct a coupled estimator for a shifted version of $H^\pi$, i.e., $\tilde H^\pi = H^{\pi} - H^\pi(s^*, a^*)$.   Recall $Z_t = (S_t, A_t, S_{t+1}, R_{t+1})$ is the transition sample at time $t$. 	    For a given state-action function, $H$, let $\Delta^\pi(Z_t; H) = 1 - H(S_t, A_t) + \sum_{a'} \pi(a'|S_{t+1})  H(S_{t+1}, a')$.   As a result of the above Bellman-like equation and the orthogonality property (\ref{orthogonal}), we know that
	\begin{align*}
	\tilde H^{\pi} \in \argminb_{H }\EE\Bigg[\frac{1}{T}\sum_{t=1}^T\big( \EE[\Delta^\pi(Z_t; H)|S_t, A_t]\big)^2\Bigg].
	\end{align*}
	Now it can be seen that the estimation of $\tilde H^\pi$ fits into the coupled estimation framework (\ref{obj in coupled estimator}), i.e., $\theta^* = \tilde H^\pi$ and $f_{\theta}$ is $\EE[\Delta^\pi(Z_t; H)\given S_t=\cdot, A_t = \cdot]$. With a slight abuse of notation, we use $\F$ to approximate $\tilde H^\pi$ and $\G$ to form the approximation of  $\EE[\Delta^\pi(Z_t; H)|S_t=\cdot, A_t=\cdot]$. The coupled estimator,  $\qn$, is then found by solving
	\begin{align}
	\qn = \argminb_{H \in \F} \Pn\Big[\frac{1}{T}\sum_{t=1}^{T}  \gn(S_t, A_t; H)^2\Big] + \lambda'_n J_1^2(H), \label{qhat}
	\end{align}
	where for any $H \in \F$, $\gn(\cdot, \cdot; H)$ solves
	\begin{align}
	\gn(\cdot, \cdot; H) = \argminb_{g \in \G} \Pn\Big[\frac{1}{T} \sum_{t=1}^{T} \big(\Delta^\pi(Z_t; H)- g(S_t, A_t)\big)^2\Big]+  \mu'_n J_2^2(g).
	\label{minimizer of ratio}
	\end{align}
	Recall that $e^\pi$ can be written in terms of $H^{\pi}$ by (\ref{epi qpi}); that is, re-arranging terms, 
	\begin{align}
	\EE_\pi\left\{1 - H(S_t, A_t) +  H(S_{t+1}, A_{t+1}) \given S_t=s, A_t=a\right\} = e^\pi(s, a). \nonumber
	\end{align}
	Thus given the estimator, $\qn$, we estimate $e^\pi$ by $\hat e^\pi_n(s, a) = \hat g_n(s, a; \qn)$. 
	By the definition of $\omega^\pi$, we have $\EE[(1/T) \sum_{t=1}^T \omega^\pi(S_t, A_t)] = 1$. Since $e^\pi$ is a scaled version of $\omega^\pi$ up to a constant, we finally construct the estimator for ratio, $\omega^\pi$,  by scaling $\hat e^\pi$,  that is, 
	\begin{align}
	\hat \omega^\pi_n(s, a) = \hat e_n^\pi(s, a) /\Pn[(1/T)\sum_{t=1}^{T} \hat e_n^\pi(S_t,A_t)], ~ \forall (s, a) \in \S \times \A.  \label{ratio estimator}
	\end{align}

	\begin{remark}
	
		The above ratio function estimator was developed by \cite{liao2019off}.
		In this paper we, for the first time, derive a finite-sample error bound for this ratio function estimator, uniformly over the policy class (Theorem B.2 in the appendix).  This is the key element in establishing the finite-sample bound regret bound for the estimated optimal policy.
		\end{remark}
		
		\begin{remark}

		Our ratio function estimator is different from  most in the existing literature, such as \cite{liu2018breaking,uehara2019minimax,nachum2019dualdice,zhang2020gendice}, which are obtained by min-max based estimating methods. For example, \cite{liu2018breaking} aimed to estimate 
		the ratio between stationary distribution induced by a known, Markovian time-stationary behavior policy and target policy, which is then used to estimate the average reward of a given policy. This is not suitable for the setting where the behavior policy is history dependent.  
		\cite{uehara2019minimax}  estimated the ratio, $\omega^\pi(s, a)$,  based on the observation that for every state-action function $f$, 
		\begin{align*}
		\EE\left[\frac{1}{T}{\sum_{t=1}^{T}}\left(\omega^\pi(S_t, A_t)\sum_{a'} \pi(a'|S_{t+1}) f(S_{t+1}, a)-\omega^\pi(S_t, A_t)f(S_t, A_t)\right)\right] = 0,
		\end{align*}
		with the restriction that $\EE[\frac{1}{T}{\sum_{t=1}^{T}}\omega^\pi(S_t, A_t)] = 1$.
		Then they constructed their estimator by solving the empirical version of the following min-max optimization problem:
		\begin{align*}
		\min_{w \in \Delta} \max_{f \in \F'} \EE^2\left[\frac{1}{T}\sum_{t=1}^{T} \left(\omega(S_t, A_t)\sum_{a'} \pi(a'|S_{t+1}) f(S_{t+1}, a)-\omega(S_t, A_t)f(S_t, A_t)\right)\right],
		\end{align*}
		where $\Delta$ is a simplex space and $\cal F'$ is a set of discriminator functions.
		This method minimizes the upper bound of the bias of their average reward estimator if the state-action value function is contained in $\F'$. They proved consistency of their ratio and average reward estimators in the parametric  setting, that is,  where $\omega^\pi(S_t, A_t)$ can be modelled parametrically and $\F'$ is a finite dimensional space.
		Subsequently \cite{zhang2020gendice} developed a general min-max based estimator by considering variational $f$-divergence, which subsumes the case in \cite{uehara2019minimax}. Unfortunately, there are no error bounds guarantee for ratio function estimators developed in \cite{uehara2019minimax} and \cite{zhang2020gendice}. Our ratio estimator appears closely related to the estimation developed by \cite{nachum2019dualdice} as they also formulated the ratio estimator as a minimizer of a loss function. However, relying on the Fenchel's duality theorem, they still use the min-max based method to estimate the ratio. Furthermore, their method cannot be applied in average reward settings. Instead of using min-max based estimators, we use  coupled estimation.  This will facilitate the derivation of estimation error bounds as will be seen below.    We will derive the estimation error of the ratio function,  which will enable us to provide a  strong theoretical guarantee, and finally demonstrate the efficiency of our average reward estimator without imposing restrictive parametric assumptions on the nuisance function estimations, see Section \ref{sec:theorem} below.

	\end{remark}

	\section{Theoretical Results}

	\label{sec:theorem}

	\subsection{Regret bound}
	
	In this section, we provide a finite sample bound on the regret of $\hat \pi_n$ defined in (\ref{regret defnition}), i.e., the difference between the optimal average reward in the policy class, $\Pi$, and the average reward of the estimated policy, $\hat \pi_n$.

	Consider a state-action function, $f(s, a)$. Let $\I$ be the identity operator, i.e., $\I(f) = f$.  Denote the conditional expectation operator by $\P^\pi f: (s, a) \mapsto \EE_\pi[f(S_{t+1}, A_{t+1})|S_t=s, A_t =a].$ Let the expectation under stationary distribution induced by $\pi$ be $\mu^\pi(f) = \int f(s, a) \mu^\pi(ds, da)$.  Denote by $\|\cdot \|_{\tv}$  the total variation distance between two probability measures. For a function $g(s, a, s')$,  define $\norm{g}^2 = \EE\left\{(1/T) \sum_{t=1}^{T}g^2(S_t, A_t, S_{t+1})\right\}$. For a set $\X$ and $M > 0$, let $\B(X, M)$ be the class of bounded functions on $\X$ such that $\norm{f}_{\infty} \leq M$.  Denote  by $N(\epsilon, \F, \norm{\cdot})$ the $\epsilon$-covering number of a set of functions, $\F$, with respect to the norm, $\norm{\cdot}$. 
	
	We make use of the following assumption on $\Pi$. 
	\begin{assumption}
		\label{assumption: policy class}
		The policy class, $\Pi = \{\pi_{\theta}: \theta \in \Theta \subset \R^{p}\}$, satisfies: 
		
		\begin{enumerate}
			[label={(3-\arabic*)}]
			
			\item \label{c0} 
			$\Theta \subset \R^p$ is compact and let $\text{diam}(\Theta) = \sup_{\theta_1, \theta_2 \in \Theta} \norm{\theta_1 - \theta_2}_2$. 
			\item \label{c1}
			There exists $L_\Theta >0$, such that for $\theta_1, \theta_2 \in \Theta$ and for all $(s, a) \in \S \times \A$, the following holds $$\abs{\pi_{\theta_1}(a|s) - \pi_{\theta_2}(a|s)} \leq L_\Theta \norm{\theta_1 - \theta_2}_2.$$

			\item \label{a2}
			There exists  constants $C_0 > 0$ and $0 \leq \beta < 1$, such that for every $\pi \in \Pi$, the following hold for all $t \geq 1$:
			\begin{align}
			& \norm{P^\pi\left(S_t = \cdot \given S_1 =s \right)  - d^\pi(\cdot)}_{\tv}  \leq {C}_0 {\beta}^t, \label{ass2}\\
			& \norm{(\P^\pi)^t f - \mu^\pi(f)}  \leq C_0 \norm{f} \beta^t. \label{ass}
			\end{align}
			
%			\item $\sup_{\pi \in \Pi} \norm{\omega^\pi}< \infty$
			
		\end{enumerate}
		
	\end{assumption}

	\begin{remark}
		
		The Lipschitz property of the policy class \ref{c1} is used to control the complexity of nuisance function induced by $\Pi$, that is, $\{U^\pi(\cdot, \cdot, \cdot): \pi \in \Pi\}$ and $\{\omega^\pi(\cdot, \cdot): \pi \in \Pi\}$. This
		is commonly assumed in the finite-time horizon problems (e.g., \cite{zhou2017residual}). \textcolor{black}{Assumptions \ref{c0} and \ref{c1} can be easily satisfied by many policy classes such the one we used in Section \ref{sec:simulation}.} Our analysis can be extended to more general policy classes if a similar complexity property holds for these two nuisance function classes.  Intuitively the constant $\beta$ in the assumption \ref{a2} relates to the ``mixing time'' of the Markov chain induced by $\pi \in \Pi$. A similar assumption was used by \cite{van1998learning,liao2019off} in average reward setting.  \textcolor{black}{Specifically,  Equation \eqref{ass2} in Assumption \ref{a2} is used to show that two nuisance functions $U^\pi$ and $\omega^\pi$ are Lipschitz continuous with respect to the policy parameter $\theta$ so that we can quantify their estimation error uniformly over the policy class. See Lemma C.1 of Supplementary Material for more details. Equation \eqref{ass} in Assumption \ref{a2} basically requires an exponential convergence rate of the policy induced Markov chain to the stationary distribution in terms of the expectation under the $L_2$-norm with respect to the data generating process. This assumption, together with Assumption \ref{b4_2} stated below is used to guarantee the Bellman operator for $U^\pi$ based on Equation \eqref{Bellman equationL eta,V} (or a similar quantity related to the ratio function estimation defined in Lemma B.4 of Supplementary Material) is well-posed in the sense of $L_2$-norm with respect to the data generating process so as to derive their estimation errors. See Lemma B.5 of \cite{liao2019off} and Lemma B.4 of Supplementary Material for more details. }

	\end{remark}
	
	Recall that we use the same pair of function classes $(\F, \G)$ in the coupled estimation for both $U^\pi$ and $\omega^\pi$. We make  the following assumptions on $(\F, \G)$. 
	\begin{assumption} 
		\label{assumption: function classes}
		The function classes, $(\F, \G)$, satisfy the following:
		
		\begin{enumerate}[label={(4-\arabic*)}]
			\item  \label{b1} $\F \subset \B(\S \times \A, F_{\max})$ and $\G \subset \B(\S \times \A, G_{\max})$
			
			\item \label{b2} $f(s^*, a^*) = 0, f \in \F$. 
			
			\item  The regularization functionals, $J_1$ and $J_2$, are pseudo norms and induced by the inner products $J_1(\cdot, \cdot)$ and $J_2(\cdot, \cdot)$, respectively. 
			
			\item \label{b6}
			Let $\F_M = \{f \in  \F: J_1(f) \leq M\}$ and $\G_M = \{g \in \G: J_2(g) \leq M\}$. There exists $C_1$   and $\alpha \in (0, 1)$ such that for any $\epsilon, M > 0$,
			\begin{align*}
			& \max \bigdkh{\log {N} (\epsilon, \G_M, \norm{\cdot}_\infty), \log {N}(\epsilon, \F_M, \norm{\cdot}_\infty)} \leq C_1\left(\frac{M}{\epsilon}\right)^{2\alpha}
			\end{align*}
			
		\end{enumerate}

	\end{assumption}
	
	\begin{remark}
		The boundedness assumption on $\F$ and $\G$ are used to simplify the analysis and can be relaxed by truncating the estimators.  We restrict $f(s^*, a^*) = 0$ for all $f \in \F$ because $\F$ is used to model $\tilde Q^\pi$ and $\tilde H^\pi$, which by definition satisfies $\tilde Q^\pi(s^*, a^*) = 0$ and $\tilde H^\pi(s^*, a^*) = 0$. In Section \ref{sec:implemenation}, we show how to shape an arbitrary kernel function to ensure this is satisfied automatically when $\F$ is RKHS.  The complexity assumption \ref{b6} on $\F$ and $\G$ are satisfied  for common function classes, for example RKHS and Sobolev spaces \citep{steinwart2008support,gyorfi2006distribution}.  \textcolor{black}{Taking the Sobolev spaces as an example, the entropy exponent $\alpha$ will be $p/\tilde q$, where $p$ is the dimension of state-variables and $\tilde q$ is the number of continuous derivatives possessed by the functions in the corresponding space. Assumption \ref{b6}  is imposed to control the estimation error for two nuisance functions.}
	\end{remark}

	We now introduce the assumption that is used to bound the estimation error of value function uniformly over the policy class.  Define the projected Bellman error 	operator: $$ \g(\cdot, \cdot; \eta, Q) := \argminb_{g \in \G} \EE\left[\frac{1}{T} \sum_{t=1}^{T} \left\{\delta^\pi(Z_t; \eta, Q)- g(S_t, A_t)\right\}^2 \right]$$ 
	where $\delta^\pi$ is given in (\ref{TDErr}).

	\begin{assumption}
		\label{assumption: value function}
		The triplet, $(\Pi, \F, \G)$,  satisfies the following:
		
		\begin{enumerate}[label={(5-\arabic*)}]
			
			\item \label{b3_2}
			$\tilde Q^\pi(\cdot, \cdot) \in \F$ for $\pi \in \Pi$ and  $\sup_{\pi \in \Pi} J_1(\tilde Q^\pi)  < \infty$. 
			
			\item \label{b3_3} 
			$0 \in \G$. 
			
			\item \label{b4_2}
			There exits $\kappa > 0$, such that $\inf\{ \norm{\g(\cdot, \cdot; \eta, Q)}_{} :   \norm{\EE[\delta^\pi(Z_t; \eta, Q)|S_t=\cdot, A_t=\cdot]}_{ } = 1, \abs{\eta} \leq R_{\max}, Q \in \F, \pi \in \Pi \} \geq \kappa. $

			\item \label{b5_2} There exists two constants $C_2, C_3$ such that $J_2\left\{\g
			(\cdot, \cdot; \eta, Q)\right\} \leq C_2 + C_3 J_1(Q)$ holds for all $\eta \in \R$, $Q \in \F$ and $\pi \in \Pi$. 
			
		\end{enumerate}	
		
	\end{assumption}
	
	\begin{remark}
		\textcolor{black}{Assumption \ref{b3_2} basically assumes that the non-parametric function class $\mathcal{F}$ can model $\tilde{Q}^\pi$ correctly, which is mild. }Note that in the coupled estimator of $\tilde Q^\pi$, we do not require the much stronger condition that the Bellman error for every tuple of $(\eta, Q, \pi)$ is correctly modeled by $\G$. In other words, $\EE[\delta^\pi(Z_t; \eta, Q)|S_t=\cdot, A_t=\cdot]$ is not required to belong to $\G$. Instead, the combination of conditions \ref{b3_3} and \ref{b4_2} is enough to guarantee the consistency of the coupled estimator (recall that the Bellman error is zero at $\{\V^{\pi}, \tilde Q^\pi\}$).
		The last condition \ref{b5_2} essentially requires the transition matrix is sufficiently smooth so that the complexity of the projected Bellman error, $J_2\left\{\g
		(\cdot, \cdot; \eta, Q)\right\}$, can be controlled by $J_1(Q)$, the complexity of $Q$ (see \cite{farahmand2016regularized} for an example).

	\end{remark}

	A similar set of conditions are employed to bound the estimation of ratio function.   For $\pi \in \Pi$ and $H \in \F$, define the projected error: $$ \g(\cdot, \cdot; H) = \argminb_{g \in \G} \EE\left[\frac{1}{T} \sum_{t=1}^{T} \left\{\Delta^\pi(Z_t; H) - g(S_t, A_t)\right\}^2 \right]$$
    where, as before,
    $\Delta^\pi(Z_t; H) = 1 - H(S_t, A_t) + \sum_{a'} \pi(a'|S_{t+1})  H(S_{t+1}, a')$.
	
	\begin{assumption}
		\label{assumption: ratio function}
		The triplet, $(\Pi, \F, \G)$,  satisfies the following:
		
		\begin{enumerate}
			[label={(6-\arabic*)}]
			
			\item  \label{b3}
			For $\pi \in \Pi$, $\tilde H^{\pi}(\cdot, \cdot) \in \F$, and $\sup_{\pi \in \Pi} J_1(\tilde H^{\pi}) < \infty$. 
			
			\item \label{b3-2} $e^\pi(\cdot, \cdot) \in \G$, for $\pi \in \Pi$. 
			
			\item \label{b4}
			
			There exits $\kappa' > 0$, such that $\inf \{\norm{\g(\cdot, \cdot; H) - \g(\cdot, \cdot; \tilde H^{\pi}) }_{} :  \norm{(\I - \P^\pi)(H - \tilde H^{\pi})}_{ } = 1, H \in \F, \pi \in \Pi\}\geq \kappa'$.

			\item \label{b5}
			
			There exists two constants $C_2', C_3'$ such that $J_2\left\{\g
			(\cdot, \cdot; H)\right\} \leq C_2' + C_3' J_1(H)$ holds for $H \in \F$ and $\pi \in \Pi$.

		\end{enumerate}

	\end{assumption}

	\begin{remark}
		\textcolor{black}{The interpretation of Assumption \ref{assumption: ratio function} is similar to that of Assumption \ref{assumption: value function}. Specifically, Assumption \ref{b3} basically assumes the non-parametric function class $\mathcal{F}$ can model $\tilde{H}^\pi$ correctly. }
		As in the case of estimation of relative value function, we do not require the correct modelling of $\EE[\Delta^\pi(Z_t; H)|S_t=\cdot, A_t=\cdot]$ for every $(\pi, H) \in \Pi \times \F$ but instead only assume Assumptions \ref{b3-2} and \ref{b4} hold. A major difference between Assumption \ref{assumption: value function} and \ref{assumption: ratio function} is that \ref{b3_3} is now replaced by \ref{b3-2}. This is because according to the Bellman-like equation (\ref{epi qpi}), we have $\EE[\Delta^\pi(Z_t; \tilde H^\pi)|S_t=s, A_t=a] = e^\pi(s, a)$.  
	\end{remark}

	\begin{theorem}
		
		\label{thm: policy learning} 	
		Suppose Assumptions \ref{assumption: irreducible} to \ref{assumption: ratio function} hold.  Let $\hat \pi_n$ be the estimated policy (\ref{est opt policy}) in which the nuisance functions are estimated with tuning parameters $\mu_n = \lambda_n = \mu_n' = \lambda_n' = L n^{-1/(1+\alpha)}$, for some constant $L > 0$.   
		Define $\beta_{k} = \frac{1}{1+\alpha} \left\{1-(1-\alpha) 2^{-k+1}\right\}$. Fix any integer $k \geq 2$, $\delta \in (0, 1)$ and sufficiently large $n$.  
		With probability at least $1-\delta$, we have 
		\begin{align*}
		\Regret(\hat \pi_n) \leq C(\delta) \left(p^{1/2} n^{-1/2} + p n^{-\beta_k} \right),
		\end{align*}
		where $C(\delta)$ is a function of $k$, $L$, $F_{\max}$, $G_{\max}$, $L_\Theta$, $\text{diam}(\Theta)$, $\sup_{\pi \in \Pi} J_1(\tilde H^{\pi})$, $\sup_{\pi \in \Pi} J_1(\tilde Q^{\pi})$, $\sup_{\pi \in \Pi} \norm{\omega^\pi}$, $\alpha$,  constants $\{C_0, C_1, C_2, C_3, C_2', C_3'\}$, $\kappa, \kappa'$, $\beta$, $p_{\min}$ and $\norm{\frac{d_{T+1}}{d_D}}_\infty$. 
		
	\end{theorem}
	
	\begin{remark} 
		Recall that $p$ is the number of parameters in the policy, $\alpha$ is given in \ref{b6}, and $n$ is the number of trajectories in the data. 
		Theorem \ref{thm: policy learning} shows that when the tuning parameters are of the order $O(n^{-1/(1+\alpha)})$, the regret of the estimated policy is $O(p^{1/2}n^{-1/2} + pn^{- \beta_k})$. The leading term (in terms of $n$), $O(\sqrt{p/n})$, corresponds to the regret of an estimated policy as if the nuisance functions are known beforehand.  The second term is due to the estimation error of nuisance functions. In particular, we show in Theorem B.1 in Section B of the appendix  that the uniform estimation error of the relative value function is of $O(pn^{-1/(1+\alpha)})$ and  in Theorem B.2  in the same section that the uniform estimation error of ratio is of $O(pn^{-\beta_k})$ (see the remark after Theorem B.2 for why the rate depends on $k$).  Note that the error of ratio is the dominant term as $\beta_k < 1/(1+\alpha)$ and $\beta_k$ can be chosen arbitrarily close to %$n^{-\frac{1}{1+\alpha}}$ 
	   $\frac{1}{1+\alpha}$
		by choosing a sufficiently large $k$. Therefore the proposed ratio estimator can achieve the near-optimal nonparametric convergence rate. See the proof of Theorem B.2 in Section B.1 of the appendix for more details.
		To the best of our knowledge, this is the first result that characterizes the regret of the estimated optimal in-class policy in the infinite horizon setting.

	\end{remark}

    	\subsection{Asymptotic results}
	
	~{In this section, we prove that the average reward for our estimator of the optimal policy  converges to the optimal average reward at a parametric rate (i.e., $\sqrt{n}$). 
	Recall $\phi^\pi(D)$ is the efficient  influence function of $\V^{\pi}$ given in Theorem \ref{thm:semipara}}.
	\begin{theorem}\label{thm: asymptotic}
		Suppose Assumptions \ref{assumption: irreducible} to \ref{assumption: ratio function} hold.  For each $n \geq 1$, let $\Vn^{\pi}$ be the doubly robust estimator defined in (\ref{pol.eval}) and $\hat \pi_n$ be the estimated policy defined in (\ref{est opt policy}) with tuning parameters $\mu_n = \lambda_n = \mu_n' = \lambda_n' = L n^{-1/(1+\alpha)}$, for some constant $L > 0$.   Then as $n \goes \infty$, 
		
		(i) $\left\{\sqrt{n}(\Vn^{\pi} - \V^{\pi}): \pi \in \Pi \right\} \wcvg \GG(\pi)$ in $l^\infty(\Pi)$ where $\GG(\pi)$ is a zero mean Gaussian Process with covariance function $\C: \Pi \times \Pi \goes \R$, $\C(\pi_1, \pi_2) = \EE\left\{\phi^{\pi_1}(D) \phi^{\pi_2}(D)\right\} 
		$.
		
		(ii) $\sqrt{n}(\Vn^{\hat \pi_n} - \sup_{\pi \in \Pi} \V^{\pi}) \wcvg \sup_{\pi \in \Pi_{\max}} \GG(\pi)$, where $\GG(\pi)$ is the Gaussian Process defined above and $\Pi_{\max} = \argmax_{\pi \in \Pi} \V^{\pi}$ is the set of policies that maximize the average reward in $\Pi$.   
		
	\end{theorem}
	
	\begin{remark}
		
		The first result shows that the estimated average reward by the doubly robust estimator reaches the semiparametric efficiency bound when we plug in the estimator for the two nuisance functions. The double robustness structure  ensures that the estimation error of nuisance functions is only of lower order    and does not impact the asymptotic variance of the estimated average reward.
		The second result shows the asymptotic of the estimated optimal value, $\Vn^{\hat \pi_n}$, converges to the maximum of the Gaussian process at the optimal policies.  When there is a unique optimal policy $\pi^* = \argmax_{\pi \in \Pi} \V^{\pi}$,  we have $\sqrt{n} (\Vn^{\hat \pi_n} - \V^{\pi^*})$ weakly converges to a  Gaussian distribution. Estimating the limiting distribution could be challenging (especially when there exists non-unique policies) and is left for future work.  Alternatively one can consider resampling-based method to construct confidence interval for $\sup_{\pi} \eta^\pi$ (see the recent work by \cite{Wu2020} in single-stage problem).

	\end{remark}

	\section{Practical Implementation}
	\label{sec:implemenation}
	In this section, we describe an algorithm to estimate an in-class optimal policy based on our efficient average reward estimator $\Vn^{\pi}$. Without loss of generality, we consider a binary-action setting, i.e., $\mathcal{A} = \{0, 1\}$, and the following stochastic parametrized policy class $\Pi$ indexed by $\theta$:
	\[
	\Pi = \left\{\pi \; \middle\vert \; \pi(1 \,|\, s, \theta) = \frac{\exp(s^T\theta)}{1 + \exp(s^T\theta)}, \; \; \|\theta\|_{\infty} \leq c, \; \; \theta \in \mathbb{R}^p
	\right\},
	\]
	for some pre-specified constant $c>0$.  Note that  other link functions  such as the probit function might be used here instead. Here $\| \cdot \|_\infty$ refers to sup-norm in Euclidean space. We fix $c = 10$ throughout our paper.
	In addition, we set $\F$ and $\G$ in the estimation of both value and ratio functions to be Reproducing Kernel Hilbert Spaces (RKHSs) associated with Gaussian kernels because of the representer theorem and the property of universal consistency. 
	
	The constraint on $\theta$, $\|\theta\|_{\infty} \leq c$, is used to maintain sufficient stochasticity in our learned policy. The stochasticity facilitates the use of $\hat\pi_n$ as a  ``warm start" policy for use by an online algorithm with future individuals. 
	A nice side effect is that  the restriction  on $\theta$ provides  a computational stability and can avoid degenerative cases in  policy optimization similar to that when using  logistic regression in classification problems \citep{friedman2001elements}. \textcolor{black}{As discussed in the introduction, we consider the simple policy class $\Pi$ instead of nonparametric models such as neural networks or tree-based models mainly due to the concern of overfitting. In the batch setting, data are limited and often noisy. Using flexible function classes for modeling the policy may lead to overfitting and thus the variance of the resulting policy could be very large. The use of a simple policy class can reduce the variance while it may induce some possible bias. In addition, interpretability is critical in our batch policy learning problem. The interpretability of decision tree models are often not very stable, whereas neural networks are not very interpretable. Therefore we prefer using this simple policy class $\Pi$.}

	To obtain $\hat{\pi}_n \in \Pi$, we solve a multi-level  optimization problem \eqref{multi-level optimization}-\eqref{lower level4}.  Recall a multi-level optimization problem \citep{richardson1995theory} is a  optimization problems in which the  feasible set is implicitly determined by a sequence of nested optimization problems. It typically consists of an upper level optimization task that represents the objective function, and a series of (possibly nested) lower level optimization tasks that represents the feasible set.
	\vspace{1ex}

	\noindent \underbar{Upper level optimization task}:	\begin{align}
	& \max_{\pi \in \Pi} && \frac{\Pn \dkh{(1/T)\sum_{t=1}^T \hat \omega_n^\pi(S_{t}, A_{t}) [R_{t+1} + \hat U^{\pi}_n(S_t, A_t, S_{t+1})]} }{\Pn \dkh{(1/T)\sum_{t=1}^T \hat \omega_n^\pi(S_{t}, A_{t})} }\label{multi-level optimization} 
	\end{align}
	
	\noindent \underbar{Lower level optimization task 1}:\begin{align}
	& && (\hat \eta_n^\pi, \hat Q_n^\pi) = \argminb_{(\eta, Q) \in \R \times \F} \Pn \left[\frac{1}{T} \sum_{t=1}^{T} \left[\gn(S_t, A_t; \eta, Q)\right]^2\right] + \lambda_n J_1^2(Q)\label{lower level1}\\
	& &&\text{s.t.} \; \; \gn(\cdot, \cdot; \eta, Q) = \argmin_{g \in \G} \Pn\Big[ \frac{1}{T}\sum_{t=1}^{T} \big( \delta^\pi(Z_t; \eta, Q) - g(S_t, A_t)\big)^2  \Big] + \mu_n J_2^2(g) \label{lower level2}
	\end{align}
	
	\noindent \underbar{Lower level optimization task 2}:\begin{align}
	& &&\qn(\cdot,\cdot) = \argminb_{H \in \F} \Pn\Big[\frac{1}{T}\sum_{t=1}^{T}  \left[\gn(S_t, A_t; H)\right]^2\Big] + \lambda'_n J_1^2(H)\label{lower level3}\\[0.1in]
	& && \text{s.t.} \, \, \gn(\cdot, \cdot; H) = \argminb_{g \in \G} \Pn\Big[\frac{1}{T} \sum_{t=1}^{T} \big(\Delta^\pi(Z_t; H)- g(S_t, A_t)\big)^2\Big]+  \mu'_n J_2^2(g). \label{lower level4}
	\end{align}
	As a reminder, recall that in Section \ref{sec:nuisance est} we have defined
	$$\delta^\pi(Z_t; \eta, Q) = R_{t+1} + \sum_{a'} \pi(a'|S_{t+1}) Q(S_{t+1}, a) - Q(S_t, A_t) - \eta,$$ 
	$$\hat{U}_n^\pi(S_t, A_t, S_{t+1}) = \sum_{a \in \cal A} \pi(a|S_{t+1}) \hat{Q}_n^\pi(S_{t+1}, a) - \hat{Q}_n^\pi(S_t, A_t),$$
	and 
	$$\Delta^\pi(Z_t; H) = 1 - H(S_t, A_t) + \sum_{a'} \pi(a'|S_{t+1})  H(S_{t+1}, a').$$
	Also, the ratio estimator $\omega_n^\pi$ can be obtained from $\qn(\cdot,\cdot)$ by using (\ref{ratio estimator}). 
	
	The upper optimization task \eqref{multi-level optimization} is used to search for $\hat{\pi}_n$ and the two parallel lower optimization tasks \eqref{lower level1}-\eqref{lower level2} and \eqref{lower level3}-\eqref{lower level4} are used to compute two nuisance function estimators for a given $\pi \in \Pi$, i.e., the feasible set, respectively. Note that each nuisance function estimation is itself a nested optimization sub-problem. 
	Multi-level optimization problems in general cannot be computed by iteratively updating solutions to lower problems \eqref{lower level1}-\eqref{lower level2} and \eqref{lower level3}-\eqref{lower level4}, and solutions to the upper problem \eqref{multi-level optimization}, in a similar manner to  coordinate descent. Hence, in order to solve this problem, one common approach is to replace the inner optimization problems \eqref{lower level1}-\eqref{lower level2} and \eqref{lower level3}-\eqref{lower level4}  by their corresponding Karush-Kuhn-Tucker (KKT) conditions so that the overall problem can be equivalently formulated as a nonlinear constraint optimization problem. However, this approach can be computationally expensive and may not be suitable for large scale settings. Instead we overcome this computational obstacle by using the representer theorem and obtain the closed-form solutions for our inner optimization problems \eqref{lower level1}-\eqref{lower level2} and \eqref{lower level3}-\eqref{lower level4} respectively. After plugging these closed-form solutions into \eqref{multi-level optimization}, we can use a gradient-based method to find $\hat{\pi}_n$.

	\subsection{RKHS reformulation} In the following subsection, we briefly discuss how to simplify our multi-level optimization problem \eqref{multi-level optimization} using the representer theorem. The details of computation can be found in Appendix E. 
	For the ease of illustration, we rewrite the training data $\D_n$ into tuples $Z_h = \{S_h, A_h, R_h, S_{h}'\}$ where $h = 1, \dots, N = nT$ indexes the tuple of the transition sample in the training set $\D_n$,   $S_h$ and $S_h'$ are the current and next states and $R_h$ is the associated reward. Let $W_h = (S_h, A_h)$  be the state-action pair,  and $W_h' = (S_h, A_h, S_{h}')$.  
	Suppose the kernel function for the state is denoted by $k_0(s_1, s_2)$, where $s_1, s_2 \in \S$. In order to incorporate the action space, we can define $k((s_1, a_1), (s_2, a_2)) = \indicator{a_1 = a_2} k_0(s_1, s_2)$. Basically, we model each $Q(\cdot, a)$ separately for each arm in the RKHS with the same kernel $k_0$.  Recall that we have to restrict the function space $\F$ such that $Q(s^*, a^*) = 0$ for all $Q \in \F$ so as to avoid the identification issue.   Thus for any given kernel function $k$ defined on $\S \times \A$, we make the following transformation by defining $k(W_h, W_{j}) = k_0(W_h, W_{j}) - k_0(( s^*, a^*), W_{h}) - k_0(( s^*, a^*), W_{j}) + k_0((s^*, a^*), (s^*, a^*))$ for any $1 \leq h, j \leq N$. 
	One can check that the induced RKHS by $k(\cdot, \cdot)$ satisfies the constraint in $\F$ automatically. 
	
	We denote kernel functions for $\F$ and $\G$ by $k(\cdot, \cdot), l(\cdot, \cdot)$ respectively. The corresponding inner products are defined as $\innerprod{\cdot}{\cdot}_\F$ and $\innerprod{\cdot}{\cdot}_\G$. We first discuss the inner minimization problem \eqref{lower level1}-\eqref{lower level2}. Note that this is indeed a nested kernel ridge regression problem, different from the standard ridge regression. The closed form solution can be obtained as $\gn(\cdot, \cdot; \eta, Q) =  \sum_{h=1}^{N} l(W_h, \cdot) \hat{\gamma}(\eta, Q)$. In particular, $\hat{\gamma}(\eta, Q) = (L + \mu I_N)^{-1} \delta_N^{{\pi}}(\eta, Q)$, where $Q \in \F$ and $L$ is the kernel matrix induced by $l$, $\mu = \mu_n N$, and $\delta^\pi_N(\eta, Q) = (\delta^\pi(Z_h; \eta, Q))_{h=1}^N$ is a vector of TD error.  Each TD error can be further written as $\delta^\pi(Z_h; \eta, Q) = R - \eta- \innerprod{Q}{
f_{W'_h}}_\F$ where 
	\[
	f_{W'_h}(\cdot) = k(W_h, \cdot) - \sum_{a'} \pi(a'|S'_h) k((S'_h, a'), \cdot) \in \F(\S \times \A)
	\]
	It can be shown that $\hat Q^\pi_n $ in \eqref{lower level1} must be in the linear span $\dkh{\sum_{h=1}^N \alpha_h f_{W_h'}(\cdot):  \alpha_h \in \R, h = 1, \dots, N}$ by using the representer property.
	
	Then we can solve the optimization problem \eqref{lower level1}-\eqref{lower level2}. The solutions for $\{\hat{U}_n^\pi(W'_h)\}_{h = 1}^N$ can be found as $-\tilde{F}(\pi)  \hat \alpha(\pi)$  where $\tilde F(\pi) = ( \innerprod{f_{W_h'}}{f_{W_{j}'}}_\F)_{j, h = 1}^N$ is a $N$ by $N$ matrix and $\hat \alpha(\pi)$ is the vector of coefficients with  a closed-form expression (see Appendix E for details). 
	Similarly, we can compute the closed-form solutions $\{ \gn(W_h, \qn )\}_{h=1}^N$ to the problem \eqref{lower level3}-\eqref{lower level4} as $L\hat \mynu(\pi)$. Here $\hat \mynu(\pi)$ is the corresponding estimated coefficients associated with the kernel matrix $L$. The details can be found in appendix E. Note that all of these intermediate terms except for $L$ depends on the policy $\pi$.
	
	Summarizing together and plugging all the intermediate results into  \eqref{multi-level optimization}, the multi-level optimization problem can be simplified as:
	\begin{align}
	& \max_{\pi \in \Pi} ~~ \frac{\left(\hat \mynu (\pi)\right)^\transpose L \left(R_N - \tilde F (\pi)\hat \alpha(\pi)\right)}{\hat \mynu (\pi)^\transpose L 1_N}\label{final optimization},
	\end{align}
	where $1_N$ is a length-$N$ vector of all ones.

	\subsection{Optimization}
	Note that problem \eqref{final optimization} becomes a smooth nonlinear optimization with box constraints. We  use limited-memory Broyden-Fletcher-Goldfarb-Shanno  algorithm with box constraints (L-BFGS-B) to compute the solution $\hat{\theta}$ \citep{liu1989limited}. The gradient computing is provided in  appendix. \textcolor{black}{The computational complexity/operations of our algorithm is of order $\Upsilon N^3p$, where $\Upsilon$ is the number of iterations in our optimization algorithm. The memory requirement is of order $N^2p$. One may implement some sub-sampling methods such as stochastic gradient decent to further improve both computation and memory complexity  of our algorithm. We will leave it for future work. }  Although the overall optimization problem is non-convex and, thus an optimal solution may not be  achievable, the performance of our numerical experiments in the following section are quite stable and promising. Recently, there is a  growing interest in studying statistical properties of algorithm-type of nonconvex M-estimators, e.g., \citep{mei2018landscape,loh2017statistical}. For many practical applications, gradient decent methods with a random initialization have been demonstrated   to converge to local minima (or even global minima) that are statistically good. While this is not the focus of our paper, it will be interesting to pursue toward this direction for future research such as studying the landscape of $\V^{\pi}$ and its related properties.

	\subsection{Tuning parameters selection} 
	In this subsection, we discuss the choice of tuning parameters in our method. 
	The bandwidths in the Gaussian kernels are selected using median heuristic, e.g., median of pairwise distance \citep{fukumizu2009kernel}. The tuning parameters $(\lambda_n, \mu_n)$ and $(\lambda'_n, \mu'_n)$ are selected based on 3-fold cross-validation. Given assumptions in Theorems B.1 and B.2 of the appendix that these tuning parameters are independent of the policy $\pi$, we can select them for the ratio and value functions separately. Specifically, for the tuning parameters $(\lambda_n, \mu_n)$ in the estimation of value function, we focus on \eqref{lower level1}-\eqref{lower level2}. For the tuning parameters $(\lambda'_n, \mu'_n)$ in the estimation of ratio function, we focus on \eqref{lower level3}-\eqref{lower level4}. At the first glance, one may think the selection of tuning parameters will be the same as those in the standard supervised learning. However, this actually requires an additional step as we cannot observe responses when estimating these two coupled estimators (recalled that we need to first compute projected bellman errors), in contrast to the standard kernel regression setting. In the following, we discuss our selection procedure of $(\lambda_n, \mu_n)$ and $(\lambda'_n, \mu'_n)$ with more details.
	
		\begin{algorithm}
		\caption{Tuning parameters selection via cross-validation}
		\label{alg:cross-validation}
		\textbf{Input:} Data $ \{Z_h \}_{h=1}^{N} $, a set of $M$ policies $\left\{\pi_1, \cdots, \pi_M \right\} \subset \Pi$, a set of $J$ candidate tuning parameters $\{(\mu_j, \lambda_j)\}_{j=1}^J$ in the value function estimation, and a set of $J$ candidate tuning parameters $\{(\mu'_j, \lambda'_j)\}_{j=1}^J$ in the ratio function estimation.
		
		Randomly split Data into $K$ subsets: $\{Z_h \}_{h=1}^{N} = \left\{D_k\right\}_{k = 1}^K$
		
		Denote $e^{(1)}(m, j)$ and $e^{(2)}(m, j)$ as the total validation error for $m$-th policy and $j$-th pair of tuning parameters in value and ratio function estimation respectively, for $m = 1, \cdots M$ and $j = 1, \cdots, J$. Set their initial values as $0$.
		
		Repeat for $ m = 1, \cdots, M$, 
		
		\Indp Repeat for $ k = 1,\cdots, K$, 
		
		\Indp Repeat for $ j = 1,\cdots, J$
		
		\Indp Use $\{Z_h \}_{h=1}^{N} \backslash D_k $ to compute $(\hat \eta_n^{\pi_{m}}, \hat \alpha(\pi_m))$ and $\hat{\mynu}(\pi_{m})$ by \eqref{lower level1}-\eqref{lower level2} and \eqref{lower level3}-\eqref{lower level4} using tuning parameters $(\mu_j, \lambda_j)$ and $(\mu'_j, \lambda'_j)$ respectively;
		
		Compute $\delta^{\pi_m}(\cdot; \hat \eta(\pi_m), \hat{Q}_n^{\pi_m})$ and $\varepsilon^{\pi_m}(\cdot; \hat{H}_n^{\pi_m})$ and their corresponding squared Bellman errors $mse^{(1)}$ and $mse^{(2)}$ on the dataset	 $D_k$ by Gaussian kernel regression;
		
		Assign $e^{(1)}(m, j) = e^{(1)}(m, j) + mse^{(1)}$ and $e^{(2)}(m, j) = e^{(2)}(m, j) + mse^{(2)}$;
		
		\Indm \Indm \Indm
		
		Compute $j^{(1)\ast} \in \argmin_{j} \max_{m} e^{(1)}(m, j)$ and $j^{(2)\ast} \in \argmin_{j} \max_{m} e^{(2)}(m, j)$
		
		\textbf{Output:} $(\mu^{(1)}_{j^{(1)\ast}}, \lambda^{(1)}_{j^{(1)\ast}})$ and $(\mu'_{j^{(2)\ast}}, \lambda'_{j^{(2)\ast}})$.
	\end{algorithm}

	We first randomly choose a set of candidate policies used to gauge our tuning parameters. 
	For each candidate policy, $\pi$, in this set, we can firstly estimate $(\hat \eta_n^\pi, \hat \alpha(\pi))$ by the proposed method using two folds of data. Then for the value function estimation, we calculate temporal difference errors $\delta^\pi(\cdot;\hat \eta_n^\pi, \hat \alpha(\pi))$ for each transition sample in the validation set. Since we cannot observe/calculate the true bellman error, following the idea in \citep{farahmand2011model}, we estimate the Bellman error by projecting these temporal differences on the space of $\S \times A$ in the validation set using the standard Gaussian kernel regression. Thus for each policy $\pi$ and each pair of tuning parameters, we output the squared estimated Bellman error in the validation set as a criterion to evaluate the performance of our value function estimation. Since tuning parameters are assumed independent of policies, we then select the tuning parameters that minimize the worst case of estimated Bellman errors among the set of all candidate policies.  We use the same strategy to select the tuning parameters for our ratio estimation. The details are given in the Algorithm \ref{alg:cross-validation}. Without the independent assumptions of tuning parameters from the policies in  $\Pi$, one may alternatively choose these tuning parameters jointly by maximizing $\Vn^{\pi}$ on the validation set, which requires large computational costs and we omit here.
	But it would be very interesting to study the theoretical properties of these two cross-validation procedures, or more generally, the selection of tuning parameters in the framework of couple estimation, which we leave it as  future work.

	\section{Simulation Studies}
	\label{sec:simulation}
	
	In this section, we consider two scenarios to evaluate the proposed algorithm. For both scenarios, we consider $S_t = (S_{t, 1}, S_{t, 2}, S_{t, 3})$ as a three-dimensional state at each decision point $t$, and the action space is binary, i.e., $\mathcal{A} = \{0, 1\}$. The behavior policy used to generate actions follows Bernoulli distribution with equal probabilities. In addition, the initial state $S_1$ is sampled from standard multi-variate normal distribution, i.e., $S_1 \sim MVN(0, I_4)$
	
	The first scenario we consider is a standard MDP setting. Let $\xi_t$ follows a standard multi-variate normal distribution. Then we generate the transition of states and reward functions via following models:
	\begin{align*}
	    & S_{t+1, 1} = 0.5 S_{t,1} + 2\xi_{t, 1}, \\
	    & S_{t+1, 2} = 0.25 S_{t,2} + 0.125 A_{t} + 2\xi_{t, 2}, \\
	    & S_{t+1, 3} = 0.9 S_{t, 3} + 0.05 S_{t, 3} A_t + 0.5 A_{t} + \xi_{t, 3}, \\
	    & R_{t+1} = 10 - 0.4 S_{t, 3} + 0.25 S_{t, 1} A_{t} \times (0.04+ 0.02S_{t,1} + 0.02S_{t, 2}) + 0.16 \xi_{t,4},
	    \end{align*}
	    for $t = 1, \cdots, T$. Here $S_{t, 3}$ can be interpreted as the treatment burden or fatigue.
	    
	 The second scenario we consider is a non-stationary environment. In particular, we consider the same transition models as above, but let the reward function to be time-dependent. More specifically, we consider
	 \begin{align*}
	        R_{t+1} = 10 - \tau_t S_{t, 3} + \beta_t S_{t, 1} A_{t} ( 0.04 + 0.02S_{t,1} + 0.02S_{t, 2}),
	    \end{align*}
	    where the time-varying parameters $\beta_t = 0.25 \times \exp(-0.05(t-1))$ and $\tau_t = 0.4 \times \exp(-0.05 (t-1))$. This generative model represents the scenario in which as the study progresses, the overall impact of intervention is decreasing. Note that since the reward function is non-stationary, we do not have a guarantee for our proposed algorithm to find an optimal policy. 
	    
	    \textcolor{black}{We compare with four baseline methods, which were proposed in the setting of the discounted sum of rewards. The first two are recently proposed deep off-policy RL algorithms \citep{fujimoto2019off,kumar2019stabilizing} denoted by BCQ and BEAR respectively. The underlying idea behind these two state-of-art algorithms is to conservatively estimate the optimal $Q$-function on the less explored state-action pair and restrict the resulting policy close to the behavior one.
	    The third method is the celebrated fitted-Q iteration (FQI) method proposed by \cite{ernst_tree-based_2005}. At each iteration, relying on the optimal Bellman equation, FQI algorithm updates the estimation of the optimal $Q$-function via solving a supervised learning problem.  The last method is V-learning proposed by \citep{luckett2019estimating}, which also aims to learn an optimal in-class policy. Since our goal is to maximize the long-term average reward, we set the discount factor $\gamma$ in these four methods as $0.99$ to approximate the average reward for comparison. In addition, to draw a relatively fair comparison, we implement these four methods using the same policy class as ours. Specifically, the first three methods will output an estimation of the optimal $Q$-function (defined in the discounted setting), after which we implement a weighted logistic regression to estimate the optimal in-class policy. For V-learning, we keep the default setup 
	   and use the same policy class as ours. Finally, for BEAR and BCQ, we use two-hidden layers neural networks with 32 nodes for each and ReLU activation functions to model the optimal $Q$-function. The other hyper-parameters are either tuned for their best performance, or recommended in the official implementation as robust choices. For FQI, we implement a kernel ridge regression at each iteration with tuning parameters selected similar to our nuisance parameter estimation. }
	    
	 To demonstrate the performance of our algorithm compared with other methods, we consider different combinations of the number of trajectories $n$ and the length of each trajectory $T$. Specifically, we consider $(n, T) = (40, 50), (40, 100), \text{and} \, (80, 50)$. Once all estimated policies are obtained, we generate another $100$ test samples with the length of trajectories $1000$ using all learned policies and compute the corresponding empirical average of observed rewards. In order to compare results with the best in-class stationary policy, we combine the gradient-type optimization algorithm with Monte Carlo method to estimate the best in-class policy that can maximize
	$\EE^\pi\left[\frac{1}{1000}\sum_{t=1}^{1000}R_{t+1}\right].$
	Specifically, for each policy with parameter $\theta$, we generate a sample of $n=100$ and $T=1000$ to approximate $\EE^\pi\left[\frac{1}{1000}\sum_{t=1}^{1000}R_{t+1}\right]$ by the empirical average of rewards. Then we apply L-BFGS algorithm and require $\theta$ to be between $-10$ to $10$ to search for the best in-class stationary policy, which is treated as the oracle policy.

	\textcolor{black}{Results of the above two scenarios can be found in Table \ref{Tab: new sim setting}. As we can see, our algorithm performs well in finding optimal in-class stationary policies, compared with the other four baseline methods. Compared with the oracle one with the best average reward about $10$, the regret of our algorithm is almost the smallest among all these methods, which is expected as we aim to maximize the average reward while the other four methods are for maximizing the discounted sum of rewards. For BEAR and BCQ with neural network models, due to the relatively small sample size and a large discount factor $\gamma$, the performances seem unstable. FQI and V-learning methods overall show competitive performances. But it can be seen that V-learning may suffer some large variance.} In addition, one possible reason for the high quality performance of our method in Scenario 2 is that the time-dependent effect in reward function is exponentially decaying by our design. Therefore we expect the performance of our algorithm may not be affected severely by the non-stationarity.  In addition, it can be seen that as the sample size $n$ or the length of each trajectory $T$ increases, the average rewards of our estimated policies are also improved, demonstrating the appealing performance of the proposed method. \textcolor{black}{Finally, we remark that the maximum running time of our method for one replication in our simulation studies is less than 40 minutes. }
	
\begin{table}[h!]
\centering
\caption{Monte Carlo estimation of the average rewards of the learned policy from proposed algorithm and three baseline offline RL algorithms over $T = 1000$ with $100$ replications. Numbers in parentheses are corresponding standard deviations. The oracle in-class optimal average rewards for both scenario are about 10.002. }
\begin{tabular}{c|cc|ccccc}
\hline
% &            & \multicolumn{5}{c|}{Factors} \\ \hline
 & $n$            & $T$ &  Our method & BEAR & BCQ & FQI & V-learning  \\ \toprule
\multirow{3}{*}{Scenario 1} 
 &  40 & 50 & $9.215 \, (0.133)$ & 7.513 (0.044) & 8.187 (0.067) & 9.728 (0.007) & 9.246 (0.470) \\ 

 &  40 & 100 & $9.913 \, (0.050)$ & 6.949 (0.053) & 7.362 (0.068) & 9.820 (0.028) & 9.345 (0.476) \\ 

 &  80 & 50 & $9.834 \, (0.052)$ & 7.487 (0.033) & 7.992 (0.059) & 9.764 (0.005) & 9.461 (0.457) \\ 

\toprule \multirow{3}{*}{Scenario 2}
 & 40 & 50 & $9.243 \, (0.133)$ & 9.128 (0.009) & 9.426 (0.020) & 9.579 (0.028) & 9.834 (0.097) \\ 
 &  40 & 100 &  $9.905 \, (0.006)$ & 9.508 (0.006) & 9.692 (0.012) & 9.858 (0.011) & 9.840 (0.108) \\ 
 & 80 & 50 & $9.919 \, (0.005)$ & 9.141 (0.011) & 9.384 (0.020) & 9.652 (0.025)& 9.873 (0.097) \\ 
  \toprule
\end{tabular}
\label{Tab: new sim setting}
\end{table}

    	\section{Application to mobile health}
	\label{sec:data analysis}

    We apply the proposed method to HeartSteps. HeartSteps is mobile health application focusing on physical activity. Three studies were conducted to develop the intervention. In this work, we apply the proposed method to the data collected from the first study, which we will refer to as HS1 in the throughout. HS1 is a 42-day micro-randomized trial (\cite{klasnja2015microrandomized,Liaoetal2015}). Each participant was provided with a Jawbone wrist tracker to collect step count data and  specified five decision times, roughly 2.5 hours apart during each day, that would be good times to potentially receive  contextually tailored activity suggestion message. In HS1, the activity message was sent with a fixed probability 0.6 at each of the five decision times.
    Our goal is to use HS1 data to learn a treatment policy that determines at each decision time whether to send the activity message (i.e., binary action). 
    
    We construct the state variable using the previous step count (the 30-min step count prior to the decision time and from yesterday), location, temperature and past notifications. We set the reward to be the log transformation of the  step count in 30-min window after each decision time.  In this analysis, we include 37 participants' data and exclude the decision times when participants were traveling
abroad or experiencing technical issues or when the reward (i.e.,
post 30-min step count) was considered as missing \citep{klasnja2018efficacy}.

    Next, we construct the policy class. In this analysis, we include two state variables in the policy. The first variable is the location (home/work vs. other locations). Location is important because people, in a more structured environment (i.e. at home or work), may respond better to an activity suggestion as compared to when they are at other locations. As a proxy for participant burden, the second variable included in the policy is ``dosage'',  a discounted sum of the number of past activity messages sent  with the discount rate chosen as 0.95. The rationale for using this variable is that receiving too many notifications in the recent past is likely to decrease the effectiveness of sending the activity message due to over-burdening participants. We consider the policy class of the form $\pi_{\theta}(1|s) = \operatorname{expit}(\theta^\transpose \phi(s)), \theta \in \Theta$, 
	where the feature vector $\phi(s) = (1, \text{dosage}, \text{location})$ and $\Theta$ is the box constraint within -10 and 10. Here dosage is standardized to be within 0 and 1.  

    We apply the proposed method with the tuning parameter selected by cross-validation in Algorithm \ref{alg:cross-validation}. The estimated coefficients are $[10,  -10,  -4.788]$. Figure \ref{fig:data analysis} shows the estimated policy at different combination of dosage and location. As one would expect, the learned policy tends to send fewer suggestions if the participant received many suggestions in the recent past. Also, the policy indicates that it is more effective to send the message when the user is at home/work location.  The estimated average reward of this policy is 3.301. As a comparison, the estimated average reward of the simple location-based policy (i.e., send only when the user is at home/work) is 3.15 and the send-nothing policy is 2.96. Transforming to the scale of the raw step count as that in \citep{klasnja2018efficacy}, the learned policy can result in $16\%$ (i.e., $\exp(3.301-3.15)-1=0.16$) improvement, which is equivalent to $40$ more steps (the mean step count across all decision times in the data is $248$) compared with the simple location-based policy, and $40\%$ (i.e., $\exp(3.301-2.96)-1=0.40$) improvement, or equivalently $101$ steps more, compared with the send-nothing policy.  \textcolor{black}{Lastly, we remark that the running time of our real data analysis is about $2$ hours, which is acceptable in the batch setting.  This is ultimately different from online RL domains where the policy is usually updated upon the arrival of each observation. }
    
		\begin{figure}[H]
		\centering
%		\vspace{-4ex}
		\includegraphics[width=0.8\linewidth]{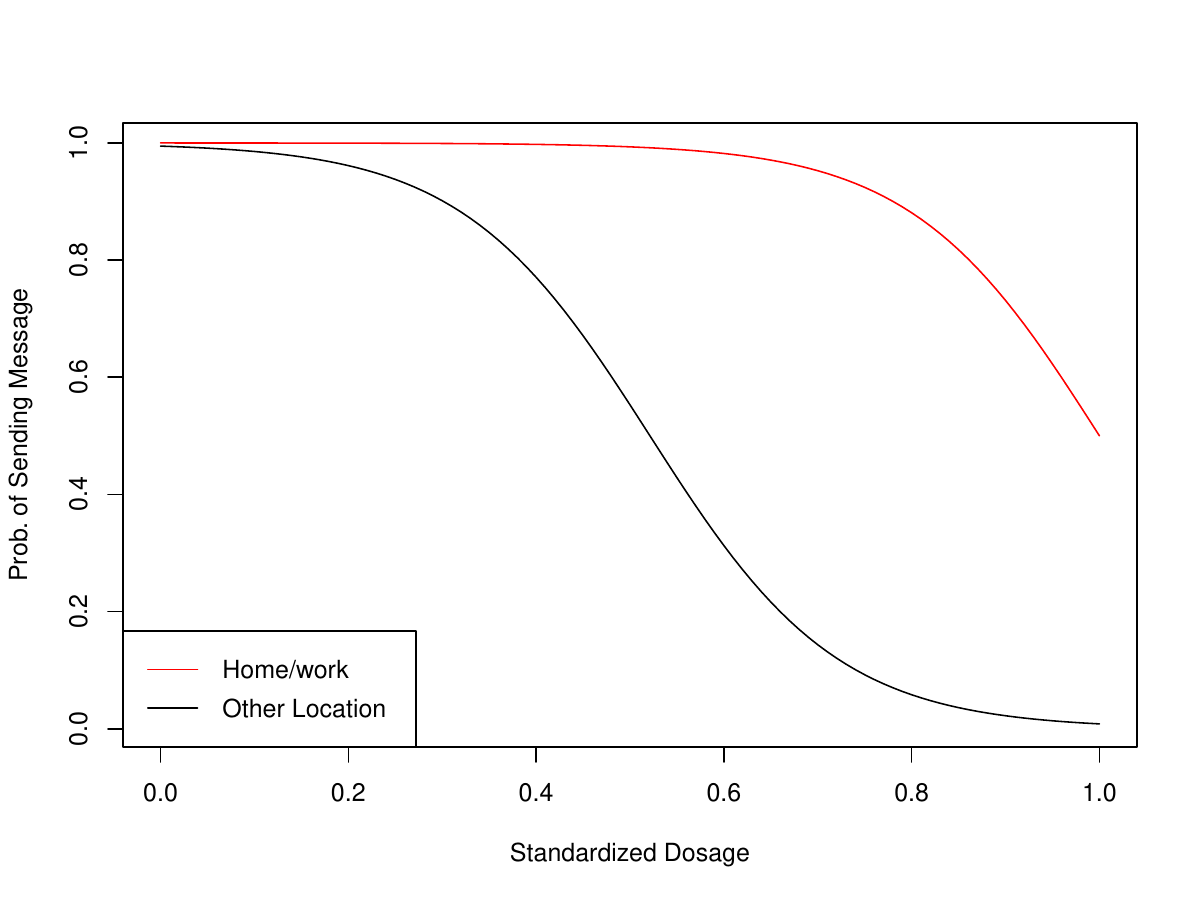}
		\caption{The estimated policy in HeartSteps data. }
		\label{fig:data analysis}
	\end{figure}

    \section{Discussion}
    \label{sec: discussion}
    
    \textit{Double/Debiased machine learning} An alternative way to construct the estimator for the average reward is based on the idea of double/debiased machine learning (a.k.a. cross-fitting, \cite{bickel1993efficient} and \cite{chernozhukov2018double}). There is growing interest in using double machine learning in causal inference and in the policy learning literature \citep{zhao2019efficient} in order to relax  assumptions on the convergence rates of nuisance parameters. The basic idea is to split the data into $K$ folds. For each  of the $K$ folds, construct the estimating equation by plugging in the estimated nuisance functions that are obtained using the remaining $(K-1)$ folds. The final estimator is obtained by solving the aggregated estimation equations. 
	While cross-fitting requires weaker conditions on the nuisance function estimations, it indeed incurs additional computational cost, especially in our setting where nuisance functions are policy-dependent and we aim to search for the in-class optimal policy. Further, this sample splitting procedure may not be stable when the sample size is relatively small, e.g., in a typical  mHealth clinical trial. A more efficient way of data splitting under the framework of MDP is needed, which we leave as future work       
    
    \textit{Computation and optimization} Our current algorithm requires relatively large computation and
memory because of the non-parametric estimation and the policy-dependent structure of nuisance functions. It is therefore desirable to develop a more efficient algorithm.
One possible remedy is to consider a zero-order optimization method such as Bayesian optimization \citep{snoek2012practical}, which is suitable when the dimension of state variables is small. \textcolor{black}{Another possible way to improve the computational efficiency is to first apply some simple algorithm to estimate a sub-optimal policy, based on which we can implement our method to estimate two nuisance parameters. Then one can develop the performance difference lemma in terms of the average reward MDP, similar to that in the discounted setting \citep{kakade2002approximately}, to construct a lower bound for $\mathcal{V}(\pi)$ using two estimated nuisance parameters. The last step is to optimize this lower bound for obtaining a better policy. This method may require less computational cost.}

\textit{Tuning parameters/Model selection} In our proposed algorithm, we assume tuning parameters are independent of policies, based on which we develop a min-max cross-validation procedure for the selection of tuning parameters. Model selection in the offline RL setting, which is necessary for improving generalization of RL techniques, is often considered as a challenging task as there is no ground truth available for performance demonstration, in contrast to the online setting with simulated environment. Therefore, it will be interesting to systematically investigate how to perform model selection in offline RL and to provide  theoretical guarantees.

\section{Acknowledgements} Peng Liao was supported by NIH grants P50DA039838, R01AA023187, and U01 CA229437. Susan Murphy was supported by NIH grants P50DA039838, R01AA023187, P50DA054039, P41EB028242, U01 CA229437, UG3DE028723, and UH3DE028723. The authors would also like to thank two reviewers, the Associate Editor and the Editor for helpful comments and suggestions that led to substantial improvement in the presentation.

	\newpage
	\appendix
\section{Semi-parametric efficiency bound and doubly robustness}
\label{appendix: semi efficiency}

In this section, we calculate the semi-parametric efficient influence function and proves the doubly robustness. 
Denote by $L(D; \zeta)$ the likelihood of a parametric sub-model of the data collected over $T$ decision times: 
$$
L(D; \zeta) = p_\zeta(S_1) \pi_{\zeta, 1} (A_1|S_1) P_\zeta (S_2|S_1, A_1) \dots \pi_{\zeta, T} (A_T|H_T) P_{\zeta}(S_{T+1}|S_{T}, A_{T}).
$$
Then the score function of the above parametric sub-model is given by
\begin{align*}
S_\zeta(D) & = \frac{d \log L(D; \zeta)}{d \zeta}= S_\zeta(S_1) + \sum_{t=1}^T  S_\zeta(S_{t+1}|S_t, A_t)  + \sum_{t=1}^T S_\zeta(A_{t} | H_t),
\end{align*}
where $S_\zeta(S_1)$ is $\frac{d \log  p_\zeta(S_1)}{d \zeta}$, $S_\zeta(S_{t+1}|S_t, A_t) = \frac{d \log S_\zeta(S_{t+1}|S_t, A_t)}{d \zeta}$, and $S_\zeta(A_{t} | H_t) = \frac{d \log S_\zeta(A_{t} | H_t)}{d \zeta}$ for $t = 1, \cdots, T$.
\begin{lemma}\label{lemma: is influence}
	If $\phi_\eff$ is an influence function, then  for any parametric submodel that contains the true parameter $\zeta_0$, 
	\begin{align*}
	\frac{d \V_\zeta(\pi)}{ d\zeta} \Big|_{\zeta = \zeta_0} = \EE[\phi^\pi (D) S_{\zeta_0} (D) ]
	\end{align*}
	%\peng{which specific theorem do we use here to show it is influence function?}
\end{lemma}

\begin{proof}[Proof of Lemma \ref{lemma: is influence}]
	
	Plugging into RHS and using the definition of score function gives
	\begin{align*}
	\EE[\phi^\pi (D) S_{\zeta_0} (D) ] &= \EE[(1/T)\sum_{t=1}^T \omega^\pi(S_t, A_t) \delta^\pi(S_t, A_t, S_{t+1}) S_{\zeta_0}(S_{t+1}|S_t, A_t)]\\
	& = \EE_\pi[\delta^\pi(S, A, S') S_{\zeta_0}(S'|S, A)] \quad (S,A,S') \sim d^\pi(s, a) P_{\zeta_0}(s'|s, a)\\
	& = \int \delta^\pi(s, a, s') \frac{d}{d\zeta} P_\zeta(s'|s, a) d^\pi(s, a) d\mu(s, a, s') \big|_{\zeta = \zeta_0}
	\end{align*}
	where $(S,A,S') \sim d^\pi(s, a) P_{\zeta_0}(s'|s, a)$ follows the stationary distribution under the true model.

	For the LHS, we start with the Bellman equation: for any $\zeta$, we have 
	\begin{align*}
	0 &= \EE_{\zeta} [\RR(S, A, S') + U^\pi_\zeta- \V_\zeta(\pi) |S=s, A= a ]\\
	& = \int \left(\RR(s, a, s') + U^\pi_\zeta- \V_\zeta(\pi) \right) P_\zeta(s'|s, a) d \mu(s'),
	\end{align*}
	where we write $U^\pi$ and  $\V^{\pi}$ as $U^\pi_\zeta$ and  $\V_\zeta(\pi)$ to explicitly indicate its dependency on $\zeta$.
	Taking the derivative implies
	\begin{align*}
	\frac{d\V_\zeta(\pi)}{d\zeta} & = \dv{\zeta} \int \left(\RR(s, a, s') + U^\pi_\zeta(s, a, s')\right) P_\zeta(s'|s, a) d \mu(s') \\
	& =  \int \dv{\zeta} \left(\RR(s, a, s') + U^\pi_\zeta(s, a, s')\right) P_\zeta(s'|s, a) d \mu(s')   \\
	& \qquad + \int \left(\RR(s, a, s') + U^\pi_\zeta(s, a, s')\right) \dv{\zeta}  P_\zeta(s'|s, a) d \mu(s')\\
	& =  \int  \left(\dv{\zeta} U^\pi_\zeta(s, a, s')\right) P_\zeta(s'|s, a) d \mu(s')   \\
	& \qquad + \int \left(\RR(s, a, s') + U^\pi_\zeta(s, a, s')\right) \dv{\zeta}  P_\zeta(s'|s, a) d \mu(s')\\
	& =  \int \left(\dv{\zeta} U^\pi_\zeta(s, a, s')\right) P_\zeta(s'|s, a) d \mu(s') + \int \delta_\zeta^\pi(s, a, s') \dv{\zeta}  P_\zeta(s'|s, a) d \mu(s')
	\end{align*}
	where in the second last line we use $\int \dv{\zeta}  \P_\zeta(s'|s, a) d\mu(s') = 0$. Now averaging over the stationary distribution $d^\pi_\zeta(s, a) = d_\zeta^\pi(s) \pi(a|s)$ of the state-action pair gives
	\begin{align*}
	\frac{d\V_\zeta(\pi)}{d\zeta}  & = \int  \left(\dv{\zeta} U^\pi_\zeta(s, a, s')\right) P_\zeta(s'|s, a) d^\pi_\zeta(s, a) d \mu(s, a, s') \\
	& \qquad + \int \delta_\zeta^\pi(s, a, s')) \dv{\zeta}  P_\zeta(s'|s, a) d^\pi_\zeta(s, a) d \mu(s, a, s')
	\end{align*}
	For the first term, using the definition of stationary distribution we have
	\begin{align*}
	& \int  \left(\dv{\zeta} U^\pi_\zeta(s, a, s')\right) P_\zeta(s'|s, a) d^\pi_\zeta(s, a) d \mu(s, a, s')
	\\
	& = \int  \frac{d Q^\pi_\zeta(s', a')}{ d\zeta} \pi(a'|s')  P_\zeta(s'|s, a)  d^\pi_\zeta(s, a) d \mu(s, a, s', a') \\
	& \qquad -  \int   \frac{d Q^\pi_\zeta(s, a)}{ d\zeta} P_\zeta(s'|s, a) d^\pi_\zeta(s, a) d \mu(s, a, s')\\
	& = \int  \frac{d Q^\pi_\zeta(s', a')}{ d\zeta} \pi(a'|s')  P_\zeta(s'|s, a)  d^\pi_\zeta(s, a) d \mu(s, a, s', a')  \\
	& \qquad -  \int   \frac{d Q^\pi_\zeta(s, a)}{ d\zeta}  d^\pi_\zeta(s, a) d \mu(s, a) = 0
	\end{align*}
	Thus we have
	\begin{align*}
	\frac{d\V_\zeta(\pi)}{d\zeta} \Big|_{\zeta = \zeta_0} & = \dv{\zeta}  \int \delta^\pi(s, a, s'))  P_\zeta(s'|s, a) d^\pi(s, a) d \mu(s, a, s') \Big|_{\zeta = \zeta_0} 
	= \EE[\phi^\pi (D) S_{\zeta_0} (D) ]
	\end{align*}
\end{proof}

\begin{lemma}\label{lemma: efficient influence}
	(i) 		The tangent space $\mathscr{T}$ is given by
	\begin{align*}
	\mathscr{T} = \left(\bigoplus_{1\leq t \leq T} (\F_t \bigoplus \G_t)\right) \bigoplus \F_{T+1}
	\end{align*}
	where
	\begin{align*}
	&\F_1 = \dkh{ q(s_1): \EE[q(S_1)] = 0, q \in L^2}\\
	&\F_t = \dkh{q(s_t, s_{t-1}, a_{t-1}): \EE[q(S_t, S_{t-1}, A_{t-1}) | S_{t-1}, A_{t-1}] = 0, q \in L^2 }, ~ 2 \leq t \leq T+1, ~ \\
	& \G_t = \{q(a_t, h_t): \EE[q(A_t, H_t) | H_t] = 0, q \in L^2 \}, ~ 1 \leq t \leq T
	\end{align*}
	(ii) The orthogonal complement of the tangent space $\mathscr{T}$ is 
	\begin{align*}
	\mathscr{T}^\perp 
	% = \left(\bigoplus_{1\leq t \leq T} (\F_t' \bigoplus \G_t')\right) \bigoplus \F_{T+1}' 
	= \bigoplus_{2\leq t \leq T+1} \F_t'
	\end{align*}
	where $\F_t' = \dkh{q(h_t) - \EE[q(H_t)|S_t, S_{t-1}, a_{t-1}]: \EE[q(H_t) | H_{t-1}, A_{t-1}] = 0, q \in L^2 }$

\end{lemma}

\begin{proof}[Proof of Lemma \ref{lemma: efficient influence}]
	Given the expression of score function $S_\zeta$, we can obtain statement (i). In particular, $\F_1$ is induced by $S_{\zeta_0}(S_1)$, $\F_t$ is induced by $S_{\zeta_0}(S_{t+1} | S_t, A_t)$ for $t = 2, \cdots, T+1$, and $\G_t$ is induced by $S_{\zeta_0}(A_t | H_t)$ for $t = 1, \cdots, T$. See Theorem 1 of \citep{kallus2020double}. For any $2 \leq t_1, t_2 \leq T+1$ and $t_1 \neq t_2$, we can show that $\F_{t_1}$ is orthogonal to $\F_{t_2}$. Without loss of generality, suppose $t_1 < t_2$. Then for any $q_{t_1} \in \F_{t_1}$ and $q_{t_2} \in \F_{t_2}$,
	\begin{align*}
	&\EE[q(S_{t_1}, S_{t_1 - 1}, A_{t_1 - 1}) q(S_{t_2}, S_{t_2 - 1}, A_{t_2 - 1})] \\[0.1in]
	=& \EE[\EE[q(S_{t_1}, S_{t_1 - 1}, A_{t_1 - 1}) q(S_{t_2}, S_{t_2 - 1}, A_{t_2 - 1}) | S_{t_1-1}, A_{t_1-1}]] \\[0.1in]
	=& \EE[\EE[q(S_{t_1}, S_{t_1 - 1}, A_{t_1 - 1}) | S_{t_1-1}, A_{t_1-1}] \times \EE[q(S_{t_2}, S_{t_2 - 1}, A_{t_2 - 1}) | S_{t_1-1}, A_{t_1-1}]]\\[0.1in]
	= & 0,
	\end{align*}
	where the second equality is by Markov property. By the similar argument, we can also show that $\G_{t_1}$ is orthogonal to $\G_{t_2}$ for $1 \leq t_1, t_2, \leq T$ and $t_1 \neq t_2$. In addition, for $2 \leq t \leq T$, we can show $\G_t$ is orthogonal to $\F_t$ by again similar argument.

	In order to derive the orthogonal complement of tangent space $\mathscr{T}$, we first note that 
	$$
	\left(\bigoplus_{2 \leq t \leq {T+1}} \F''_t\right) \bigoplus \left(\F_1 \bigoplus \G_1\right)
	$$	
	is the space of all random functions with mean zero and finite variance, where 
	$$
	\F''_t := \left\{ q(h_t) : \EE[q(H_t)|H_{t-1}, A_{t-1}] = 0, q \in L^2\right\},
	$$
	for $t = 2, \cdots, (T+1)$, are orthogonal to each other.
	Then it is enough to project each elements in $\F''_t$ onto the orthogonal complement of $\left(\F_t \bigoplus \G_t\right)$ for $2 \leq t \leq T$ and $\F_{T+1}$ respectively.

	First of all, we can see that $\F''_t$ is orthogonal to $\G_t$ for $2 \leq t \leq T$ by the definition of $\G_t$. Secondly, it is straightforward to show that $\F''_t$ is orthogonal to
	$$
	\F'''_t = \left\{ q(s_{t-1}, a_{t-1}) : \EE[q(S_{t-1}, A_{t-1}) ] = 0, q  \in L^2
	\right\},
	$$
	which is indeed also orthogonal to $\F_t$ for $2 \leq t \leq (T+1)$.
	Then projecting each element in $\F''_t$ onto the orthogonal complement of $\F_t$ is equivalent to projecting onto the orthogonal space of
	$$
	\F'''_t \cup \F_t = \left\{ q(s_t, s_{t-1}, a_{t-1}) : \EE[q(S_t, S_{t-1}, A_{t-1}) ] = 0, q  \in L^2
	\right\},
	$$
	which gives us exactly $\F'_t$. This concludes statement (ii).

\end{proof}

\begin{proof}[Proof of Theorem 3.1]
	
	Choose an arbitrary function $q(H_k)$  such that $\EE[q(H_k) | H_{k-1}, A_{k-1}] = 0$ and define $f(H_k) = q(H_k) - \EE[q(H_k)|S_k, S_{k-1}, A_{k-1}] \in \F_k'$. We have
	\begin{align*}
	& \EE[\phi^\pi (D) f(H_k)] = \frac{1}{T} \sum_{t=1}^T \EE[ \omega^\pi(S_t, A_t) \delta^\pi(S_t, A_t, S_{t+1}) f(H_k)]\\
	& = \frac{1}{T} \sum_{t=1}^{k-1} \EE[ \omega^\pi(S_t, A_t) \delta^\pi(S_t, A_t, S_{t+1}) f(H_k)] \\
	& = \frac{1}{T} \Bigdkh{ \EE[ \omega^\pi(S_{k-1}, A_{k-1}) \delta^\pi(S_{k-1}, A_{k-1}, S_{k}) f(H_k)] \\
		& \qquad +\sum_{t=1}^{k-2} \EE[ \omega^\pi(S_t, A_t) \delta^\pi(S_t, A_t, S_{t+1}) f(H_{k-1})]}
	\end{align*}
	The first term is zero since $\omega^\pi(s_{k-1}, a_{k-1}) \delta^\pi(s_{k-1}, a_{k-1}, s_k) \in \F_k$ by definition of TD error and thus orthogonal to $f(h_k) \in \F_k^\perp$.  
	For the second term,  for any $1 \leq t \leq k-2$, we have $(S_t, A_t, S_{t+1}) \in \sigma(H_{k-1})$
	\begin{align*}
	& \EE[ \omega^\pi(S_t, A_t) \delta^\pi(S_t, A_t, S_{t+1}) f(H_k)]  \\
	& = \EE[ \EE[\omega^\pi(S_t, A_t) \delta^\pi(S_t, A_t, S_{t+1}) f(H_k)|H_{k-1}, A_{k-1}]]\\
	& = \EE[ \omega^\pi(S_t, A_t) \delta^\pi(S_t, A_t, S_{t+1}) \EE[ f(H_k)|H_{k-1}, A_{k-1}]] = 0
	\end{align*}
	where the last equality follows from $\EE[ f(H_k)|H_{k-1}, A_{k-1}] = 0$ by noting
	\begin{align*}
	\EE[ f(H_k)|H_{k-1}, A_{k-1}] & = \EE[ q(H_k) - \EE[q(H_k)|S_k, S_{k-1}, A_{k-1}]|H_{k-1}, A_{k-1}] \\
	& = 0 - \EE[ \EE[q(H_k)|S_k, S_{k-1}, A_{k-1}]|H_{k-1}, A_{k-1}]\\
	& = - \EE[ \EE[q(H_k)|S_k, S_{k-1}, A_{k-1}]|S_{k-1}, A_{k-1}]\\
	& = -\EE[q(H_k)|S_{k-1}, A_{k-1}]\\
	& = - \EE[\EE[q(H_k)|H_{k-1}, A_{k-1}] |S_{k-1}, A_{k-1}] = 0
	\end{align*}
	where the second and last equality follow from the definition of $q$ (i.e., $\EE[q(H_k) | H_{k-1}, A_{k-1}] = 0$) and the third equality follows from the Markov property. As a result, we have $\EE[\phi^\pi (D) f(H_k)] = 0$. Lemma \ref{lemma: efficient influence} then implies that $\phi_\eff$ is orthogonal to $\F^\perp$ and thus $\phi_\eff$ is the efficient influence function (see, for example, Theorem 4.3 \cite{tsiatis2007semiparametric}).
	
\end{proof}	

\begin{proof}[Proof of Theorem 3.2]
	Note that the ratio estimator satisfies $$\PP_n (1/T) \sum_{t=1}^T \hat \omega^\pi_n(S_t, A_t) = 1.$$ Define 
	$$
	\tilde{\V}_n^{\pi} = \Pn \dkh{(1/T)\sum_{t=1}^T \bar w^{\pi}(S_{t}, A_{t}) [R_{t+1} + \bar U^{\pi}(S_t, A_t, S_{t+1})]}.
	$$
	We can first show that
	\begin{align*}
	&|\hat{\V}_n(\pi) - \tilde{\V}_n^{\pi}| \\[0.1in]
	=& | \Pn \dkh{(1/T)\sum_{t=1}^T \hat \omega_n^\pi(S_{t}, A_{t}) [R_{t+1} + \hat U^{\pi}_n(S_t, A_t, S_{t+1})]} \\[0.1in]
	-& \Pn \dkh{(1/T)\sum_{t=1}^T \bar w^{\pi}(S_{t}, A_{t}) [R_{t+1} + \bar U^{\pi}(S_t, A_t, S_{t+1})]}  |\\[0.1in]
	\leq & \left| \Pn \dkh{(1/T)\sum_{t=1}^T \hat \omega_n^\pi(S_{t}, A_{t}) [\hat U^{\pi}_n(S_t, A_t, S_{t+1}) - \bar U^{\pi}(S_t, A_t, S_{t+1})]}  \right|	\\[0.1in]
	+ & \left| \Pn \dkh{(1/T)\sum_{t=1}^T \left(\hat \omega_n^\pi(S_{t}, A_{t}) - \bar w^{\pi}(S_{t}, A_{t})\right)\left[R_{t+1}+\bar U^{\pi}(S_t, A_t, S_{t+1})\right]} \right|\\[0.1in]
	\leqconst & \sup_{s \in S, a \in \A, s' \in S} \left|\hat U_n^\pi(s, a, s') - \bar U^\pi(s, a, s')\right| + \sup_{s \in S, a \in \A}\left|\hat \omega_n^\pi(s, a) - \bar \omega^\pi(s, a)\right|,
	\end{align*}
	which converges to 0 in probability by assumptions in the lemma. In addition,	by the law of large numbers, we know $\tilde{\V}_n^{\pi}$ converges to $\tilde{\V}(\pi)$ in probability, where
	$$
	\tilde{\V}(\pi) = \EE[(1/T)\sum_{t=1}^T \bar w^{\pi}(S_{t}, A_{t}) [R_{t+1} + \bar U^{\pi}(S_t, A_t, S_{t+1})]].
	$$	
	If $\bar w^{\pi}= w^{\pi}$, then by the definition of stationary distribution, we have
	\begin{align*}
	\tilde{\V}(\pi) = \int \left(r(s, a) + \bar{U}^\pi(s, a, s')\right) d^\pi(s, a) ds da = \int r(s, a)d^\pi(s, a) ds da = \V^{\pi}.
	\end{align*}
	If $Q^\pi = \bar{Q}^\pi$, which implies $U^\pi = \bar{U}^\pi$, then
	\begin{align*}
	\tilde{\V}(\pi) &= \EE[(1/T)\sum_{t=1}^T \bar w^{\pi}(S_{t}, A_{t}) \EE[R_{t+1} +  U^{\pi}(S_t, A_t, S_{t+1}) | S_t, A_t ]]\\[0.1in]
	&= \EE[(1/T)\sum_{t=1}^T \bar w^{\pi}(S_{t}, A_{t}) \V^{\pi}] =\V^{\pi}.
	\end{align*}
	Thus if either $\bar \omega^\pi$ or $\bar U^\pi$ is correct, then $\tilde{\V}(\pi) = \V^{\pi}$. This gives that $\tilde{\V}_n^{\pi}$ converges to $\V^{\pi}$ with probability $1$. Summarizing above, we can show that
	\begin{align*}
	\left|\hat{\V}_n(\pi) - \V^{\pi} \right| &\leq \left|\hat{\V}_n(\pi) - \tilde{\V}_n^{\pi} \right| + \left|\bar{\V}_n(\pi) - \V^{\pi} \right|,
	\end{align*}
	converges to $0$ in probability. This concludes our statement.
\end{proof}

\section{Theoretical Results on Nuisance Function Estimation}\label{appendix: nuisance}

In this section, we present  two finite sample upper bounds for the estimation error of the nuisance functions that holds uniformly over the policy class. These results are needed to prove Theorem 5.1 and 5.2. 

We first present the uniform bound for the estimation error of $U^\pi$ over $\pi \in \Pi$. This is a generalization of Theorem 1 in \cite{liao2019off} in which they focused only on a single policy.

\begin{theorem}
	\label{thm: value}
	Suppose the tuning parameters $\mu_n = \lambda_n = L n^{-1/(1+\alpha)}$ and Assumptions   1-5 hold.  
	Fix some $\delta > 0$. There exists some constant $C(\delta)$ that depends on $R_{\max}, F_{\max}, G_{\max}, L_\Theta, \text{diam}(\Theta)$, $\sup_{\pi \in \Pi} J_1(\tilde Q^\pi)$, $\sup_{\pi \in \Pi} \norm{\omega^\pi}$, $\{C_0, C_1, C_2, C_3\}$ and $\alpha$, such that the following holds with probability {$1-\delta$:}
	\begin{align*}
	\sup_{\pi \in \Pi} \norm{\Upin- \Upi}^2 \leq C(\delta)   \iota pn^{-1/(1+\alpha)}
	\end{align*}
	where $\iota = 4(\kappa)^{-2}(1+p_{\min}^{-1}(1+(1/T) \norm{\frac{d_{T+1}}{d_D}}_\infty))(1+C_0\beta/(1-\beta))^2$
\end{theorem}
Since the proof is similar to that in \citep{liao2019off} with additional efforts on controlling complexity of the policy class $\Pi$, we omit here. Next we present the uniform finite sample bound for the ratio estimator. 

\begin{theorem}
	\label{thm: ratio}
	Suppose assumptions  1-4 and  6 hold.  Let $\hat \omega^\pi_n$ be the estimated ratio function with tuning parameter $\mu_n'= \lambda_n' = L n^{-1/(1+\alpha)}$ defined in (4.10). 
	Fix some $k \geq 2$ and  $\delta > 0$. There exists some constant $C^{[k]}(\delta)$ that depends on $F_{\max}, G_{\max}, L_\Theta, \text{diam}(\Theta)$, $\sup_{\pi \in \Pi} J_1(\tilde H^{\pi})$, $\{C_0, C_1, C_2', C_3'\}$, $\sup_{\pi \in \Pi} \norm{\omega^\pi}$ and $\alpha$, such that for sufficiently large $n$ the followings hold with probability {$1-(5+k)\delta$:}
	\begin{align*}
	\sup_{\pi \in \Pi} \norm{\hat \omega^\pi_n- \omega^\pi}^2 \leq C^{[k]}(\delta) (\iota')^{\omega_{k}} pn^{-\beta_{k}},
	\end{align*}
	where  $\omega_{k} = 1 - 2^{-k+1}$ and $\iota' = 4(\kappa')^{-2}(1+p_{\min}^{-1}(1+(1/T) \norm{\frac{d_{T+1}}{d_D}}_\infty))(1+C_0\beta/(1-\beta))^2$. 
	
\end{theorem}

\begin{remark}

	Recall that the optimal convergence rate for the classical nonparametric regression problem under the entropy condition (6-4) is $n^{-1/(1+\alpha)}$. Theorem \ref{thm: ratio} shows that the the ratio estimator achieves the near-optimal convergence rate. As we have seen in Theorem 5.2, the achieved error rate is enough to guarantee the asymptotic efficiency of the doubly robust estimator; in fact we only need to ensure the error decays faster than $n^{-1/4}$.

	Although the estimator for $\omega^\pi$ (or $\tilde H^\pi$ more specifically)  is similar to the estimator for $U^\pi$ (or $\tilde Q^\pi$) in that both of them use the coupled estimation framework. It turns out that the Bellman error at $\{\V^{\pi}, \tilde Q(\pi)\}$ equal to zero greatly simplifies the analysis for relative value function. In the case of ratio estimator, the analog of Bellman error is not zero at $\tilde H^\pi$, that is, $\EE[\Delta^\pi(Z_t; \tilde H^{\pi})|S_t=s, A_t = a] \neq 0$.  Our proof of Theorem \ref{thm: ratio} is based on iteratively controlling a remainder term to make the error rate arbitrarily close to the desired rate, $n^{-1/(1+\alpha)}$.  To the best of our knowledge, this is the first result that directly characterizes the error of the ratio estimator. 
	How to obtain the exact optimal rate for the ratio estimator is left for future work.  In the following section, we give the proof of Theorem \ref{thm: ratio}. 
	
\end{remark}

\subsection{Uniform error bound of ratio estimator}
\label{appendix: ratio bound}
%	\peng{Give an overview of the proofs. }
%	Recall that we need to show the estimation error for the ratio.  
%	\begin{enumerate}
%		\item $\norm{\g(\qn) - \g(\qpi)}^2$ -- this will give the bound on $\qn - \qpi$, $J_1(\qn)$ and $J_2(\gn(\qn))$. 
%		\item $\norm{\hat e_n^\pi - e^\pi}^2 = \norm{\gn(\qn) - \g(\qpi)}^2$
%		\item
%	\end{enumerate}

\begin{theorem} \label{thm: pre-ratio}
	Suppose $\mu_n = \lambda_n = L n^{-1/(1+\alpha)}$. Fix some $k \in \mathds{N}^+$. Under Assumptions  (2-1), (2-2), (3-3),  (4-1), (4-2), (6-1), (6-2), (6-3), (6-4) and (4-4), the followings hold with probability $1-(3+k)\delta$: for all $\pi \in \Pi$:
	\begin{align*}
	& \norm{\g(\qn) - \g(\qpi)}^2 \leq C^{[k]}(\delta) \iota^{\omega_{k}} pn^{-\beta_{k}}
	\\& J_1^2(\qn) \leq C^{[k]}(\delta) \iota^{\omega_{k}} pn^{1/(1+\alpha) - \beta_{k}}
	\end{align*}
	where $\beta_{k} = \frac{1}{1+\alpha} (1-(1-\alpha) 2^{-k+1})$,  $\omega_{k} = 1 - 2^{-k+1}$, $C^{[k]}(\delta)$ is  a constant that depends on $\dkh{k, \log(1/\delta), F_{\max}, G_{\max}, L_\Theta, \text{diam}(\Theta), \{C_i\}_{i=1}^3, \sup_{\pi \in \Pi} J_1(\tilde H^{\pi}), \alpha}$ and 
	\[
	\iota = \Bigg[\frac{2 + 2 C_0 \beta/(1-\beta) }{\kappa'} \Bigg]^2 \Bigg[\frac{2 + (1/T) \norm{\frac{d_{T+1}}{d_D}}_\infty}{p_{\min}} \Bigg]
	\]
\end{theorem}
\begin{proof}[Proof of Themrem \ref{thm: pre-ratio}]
	%	Link the lemmas below to show the convergence rate of the ratio estimator. 
	
	We start with 
	\begin{align}
	\norm{\g(\qn) - \g(\qpi)}^2 
	\leq 2 \norm{\g(\qn) - \gn(\qn)}^2 + 2\norm{\gn(\qn) - \g(\qpi)}^2  \label{target}
	\end{align}
	Denote the leading constant in Lemma \ref{lemma: uniform l2 bound} by ${K_0}$. For the choice of tuning parameters $(\lambda_n, \mu_n)$,  Lemma \ref{lemma: uniform l2 bound} and Assumption (6-4) imply that w.p. $1-\delta$,  for all $\pi \in \Pi$,   the first term in (\ref{target}) can be bounded by
	%		\[
	%		\norm{\g(\qn) - \gn(\qn)}_2^2 \leq C_1(\delta)pn^{-1/(1+\alpha)} + {K_1} n^{-1/(1+\alpha)} J_1^2(\qn)
	%		\]
	\begin{align*}
	& \norm{\g(\qn) - \gn(\qn)}_2^2 \\
	& \leq {K_0}\left[\mu_n (1+J_1^2(\qn) + J_2^2(\g(\qn)) + \frac{p}{n \mu_n^{\alpha}} + \frac{1}{n} +  \frac{\log(1/\delta)}{n}\right]\\
	& \leq  {K_0}\left[L n^{-1/(1+\alpha)} (1+J_1^2(\qn) + 2C_1^2 + 2C_2^2 J_1^2(\qn)) + \frac{p}{n (Ln^{-1/(1+\alpha)})^{\alpha}} + \frac{1}{n} +  \frac{\log(1/\delta)}{n}\right]\\
	%		& \leqconst \mu_n (1+J_1^2(\qn)) + \frac{p}{n \mu_n^{\alpha}} + \frac{1+\log(1/\delta)}{n}  \\
	& \leq C_1(\delta)pn^{-1/(1+\alpha)} + {K_1} n^{-1/(1+\alpha)} J_1^2(\qn)
	%\\
	%& = {K_1}(\mu_n + \frac{p}{n \mu_n^{\alpha}} + \frac{1+\log(1/\delta)}{n} ) + {K_1} \mu_n J_1^2(\qn) 
	%\\
	%& = \gamma_1(\delta, n, p, \mu_n) + {K_1} \mu_n J_1^2(\qn)
	\end{align*}
	where ${K_1} = {K_0} L (1+ 2C_2^2)$ and $C_1(\delta) =  {K_0} \left(L (1+2C_1^2) + L^{-\alpha} + 1 + \log(1/\delta)\right)$.  
	%		In addition, from Lemma \ref{lemma: upper level initial bound} we have
	%		\begin{align}
	%		& J_2(\gn(\qn))  \leq {K_0}(1 + J_1(\qn) +  J_2(\g (\qn)) + \sqrt{\frac{p}{n \mu_n^{\alpha+1}} }+ \sqrt{\frac{1}{n \mu_n}} +  \sqrt{\frac{\log(1/\delta)}{n\mu_n}}) \notag \\
	%		& \leq {K_0}(1 + C_2 + (1+C_2) J_1(\qn) + \sqrt{pL^{-\alpha-1}}+  L^{-\frac{1}{2}}(1+\sqrt{\log(1/\delta)})n^{-\frac{\alpha}{2(1+\alpha)}} ) \label{J2 bound}
	%		\end{align}

	Now consider the second term. Denote the leading constant in Lemma \ref{lemma: upper level initial bound} by ${K_2}$. Applying the Decomposition Lemma \ref{lemma: second term} implies that w.p. $1-2\delta$, for all $\pi \in \Pi$
	\begin{align*}
	& \norm{\gn(\qn) - \g(\qpi)}^2 + \lambda_n J_1^2(\qn) 
	\\
	& \leq 	\gamma_2(\delta, n, p, \mu_n, \lambda_n)  + {K_2}n^{-1/(1+\alpha)} J_1(\qn) +  \Rem(\pi) \\
	%	& \leq C_2(\delta) p n^{-1/(1+\alpha)} + {K_2}n^{-1/(1+\alpha)} J_1(\qn) +  \Rem(\pi) 
	\end{align*}
	With the choice of $(\lambda_n, \mu_n)$, $\gamma_2(\delta, n, p, \mu_n, \lambda_n)$ can be bounded by 
	%	\peng{where the constant, $C_2(\delta)$, can be found by
	%		\begin{align*}
	%		\gamma_2(\delta, n, p, \mu_n, \lambda_n) 
	%		%	 =  {K_1}\Big[\mu_n + \lambda_n + \frac{p}{n \mu_n^{\alpha}} + \frac{1+\log(1/\delta)}{n}  \Big]  \\ 
	%		%	& 
	%		%	+ {K_2}\Big[ \frac{1}{n^{1/(1+\alpha)}} + \frac{1}{n\lambda_n^\alpha} + \frac{p^{\frac{\alpha}{1+\alpha}}}{n \mu_n^{\alpha}} + \frac{1+\log^{\frac{\alpha}{1+\alpha}}(1/\delta)}{n\mu_n^{\alpha/(1+\alpha)}} + \frac{1+\log(1/\delta)}{n} \Big] \\
	%		%	& + {K_3} \mu_n \Big(1 + \sqrt{\frac{p}{n \mu_n^{\alpha+1}} }+ \sqrt{\frac{1}{n \mu_n}} +  \sqrt{\frac{\log(1/\delta)}{n\mu_n}}\Big) \\
	%		\leq C_2(\delta) p n^{-1/(1+\alpha)}
	%		\end{align*}
	\begin{align*}
	\gamma_2 & (\delta, n, p, \mu_n, \lambda_n) 
	%	 =  {K_1}\Big[\mu_n + \lambda_n + \frac{p}{n \mu_n^{\alpha}} + \frac{1+\log(1/\delta)}{n}  \Big]  \\ 
	%	& 
	%	+ {K_2}\Big[ \frac{1}{n^{1/(1+\alpha)}} + \frac{1}{n\lambda_n^\alpha} + \frac{p^{\frac{\alpha}{1+\alpha}}}{n \mu_n^{\alpha}} + \frac{1+\log^{\frac{\alpha}{1+\alpha}}(1/\delta)}{n\mu_n^{\alpha/(1+\alpha)}} + \frac{1+\log(1/\delta)}{n} \Big] \\
	%	& + {K_3} \mu_n \Big(1 + \sqrt{\frac{p}{n \mu_n^{\alpha+1}} }+ \sqrt{\frac{1}{n \mu_n}} +  \sqrt{\frac{\log(1/\delta)}{n\mu_n}}\Big) \\
	= {K_2} \Big[\mu_n + \lambda_n + \frac{p}{n \mu_n^{\alpha}} + \frac{1}{n^{1/(1+\alpha)}} + \frac{1}{n\lambda_n^\alpha} \\ 
	&   +  \frac{1+\log^{\frac{\alpha}{1+\alpha}}(1/\delta)}{n\mu_n^{\alpha/(1+\alpha)}} + \frac{1+\log(1/\delta)}{n} + \sqrt{\frac{p}{n \mu_n^{\alpha-1}} }+ \sqrt{\frac{\mu_n}{n}} (1+\sqrt{\log(1/\delta)}) \Big] \\
	&   = {K_2} \Big[ 2Ln^{-1/(1+\alpha)}  + L^{-\alpha} pn^{-1/(1+\alpha)}  + \frac{1}{n^{1/(1+\alpha)}} + L^{-\alpha} n^{-1/(1+\alpha)}  \\ 
	&   +  \frac{2(1+\log(1/\delta))}{n^{1-\alpha/(1+\alpha)^2} L^{\alpha/(1+\alpha)}}  + L^{(1-\alpha)/2}\sqrt{p} n^{-1/(1+\alpha)}  + L^{1/2} n^{-(1+\alpha/2)/(1+\alpha)} (1+\sqrt{\log(1/\delta)}) \Big] \\
	& \leq  {K_2} \left[  2L + 2L^{-\alpha}  + 1 +  2L^{-\alpha/(1+\alpha)} (1+\log(1/\delta)) + L^{1/2} (1+\sqrt{\log(1/\delta)})\right] pn^{-1/(1+\alpha)} \\
	& := C_2(\delta) pn^{-1/(1+\alpha)}
	\end{align*}

	As a result,  we obtain that with probability at least $1-3\delta$, for all $\pi \in \Pi$, 
	\begin{align}
	\lambda_n J_1^2(\qn) \leq C_2(\delta) p n^{-1/(1+\alpha)} +  {K_2} n^{-1/(1+\alpha)} J_1(\qn) +  \Rem(\pi) \label{J iter}
	\end{align}
	%and 
	\begin{align}
	\norm{\g(\qn) - \g(\qpi)}^2 & \leq 2(C_1(\delta) + C_2(\delta))pn^{-1/(1+\alpha)} + 2{K_1} n^{-1/(1+\alpha)} J_1^2(\qn)  \notag \\
	& \qquad + 2{K_2}n^{-1/(1+\alpha)} J_1(\qn) + 2 \Rem(\pi) \label{err iter}
	\end{align}

	\textbf{Initial Rate} \quad
	We derive an initial rate by bounding $\Rem(\pi)$ uniformly over $\pi \in \Pi$.  Let $$f(H_1, H_2) (D) = \frac{4}{T} \sum_{t=1}^T \g(S_t; A_t; H_2) [\Delta^\pi(S_t, A_t, S_{t+1}; H_1) - \Delta^\pi(S_t, A_t, S_{t+1}; H_2)].$$ We thus have $\Rem(\pi) = \Pn f(\qn, \qpi)$. Note that under Assumption (6-2), $\g(\qpi) = e^\pi$. The orthogonality property (3.4) then implies that  $P f(q, \qpi) = 0$ for any $H \in \F$. We can bound $\Rem(\pi)$ by 
	\begin{align*}
	\Rem (\pi) = \abs{\Pn f(\qn, \qpi)} = \frac{J_1(\qn-\qpi)}{\sqrt{n}} \frac{\abs{\Gn f(\qn, \qpi) }}{J_1(\qn - \qpi)}   \leq  \frac{J_1(\qn - \qpi)}{\sqrt{n}}  \sup_{f \in \F_0}  \abs{\Gn f}
	\end{align*}
	where the function class $\F_0$ is given by
	\begin{align*}
	%	\F_0 & =  \{(Y_{q} - Y_{\qpi}) \g(\qpi)/J_1(q-H^{\pi}): H \in \F, \pi \in \Pi \} \\
	\F_0 & =  \Bigdkh{D \mapsto \frac{4}{T}\sum_{t=1}^T\Big[\frac{(H^{\pi} - q)(S_t, A_t)}{J_1(H^{\pi} - q)} 
		- \sum_{a'} \pi(a'|S')   \frac{(H^{\pi} - q)(S_{t+1}, a')}{J_1(H^{\pi} - q)} \Big]  \g(S_t, A_t; \qpi): H \in \F, \pi \in \Pi }\\
	&=  \Bigdkh{D \mapsto \frac{4}{T}\sum_{t=1}^T\Big[H(S_t, A_t) 
		- \sum_{a'} \pi(a'|S')   H(S_{t+1}, a')\Big] \g(S_t, A_t; \qpi): H \in \F, J_1(q) = 1, \pi \in \Pi }
	\end{align*}
	and thus we have $\F_0 \subset \F_1$, where 
	\begin{align*}
	\F_1 = & \Bigdkh{D \mapsto \frac{4}{T}\sum_{t=1}^T\Big[H(S_t, A_t)   - \sum_{a'} \pi(a'|S') H(S_{t+1}, a') \Big] g(S_t, A_t):\\
		& \qquad\qquad\qquad  H \in \F, J_1(q) \leq 1, g \in \G, J_2(g) \leq \sup_{\pi \in \Pi} J_2(e^\pi), \pi \in \Pi}
	\end{align*}
	%	\begin{align*}
	%	\F = & \Bigdkh{ (S, A, S') \mapsto \big[q(S, A) - \sum_{a'} \pi(a'|S')    q (S', a')\big] g(S, A): \\
	%		& \hspace{30ex}H \in \F, J_1(q) \leq 1, \pi \in \Pi, J_2(g) \leq \sup_{\pi \in \Pi} J_2(e^\pi)}
	%	\end{align*}
	{Applying  Lemma \ref{lemma: donsker class prob bound} with  $M=1$ and $\sigma = f_{\max} := 8G_{\max}F_{\max}$ implies that the following holds with probability at least $1-\delta$:}
	\begin{align*}
	\sup_{f \in \F_1}  \abs{\Gn f}  & \leq {\tilde C(\delta)} \Big(2\sqrt{p}f_{\max}^{1-\alpha} + 2pf_{\max}^{-2\alpha}n^{-1/2} + f_{\max} + n^{-1/2} \\
	& \qquad + 2p^{1/4} n^{-1/4} f_{\max}^{(1-\alpha)/2} + 2\sqrt{p} f_{\max}^{-\alpha}n^{-1/2} \Big)   \\
	& \leq {\tilde C(\delta)} (2f_{\max}^{1-\alpha} + 2f_{\max}^{-2\alpha}+ f_{\max} + 1 + 2f_{\max}^{(1-\alpha)/2} +  2f_{\max}^{-\alpha})  \\
	& \qquad \times \left( \sqrt{p} + p n^{-1/2} + 1 + n^{-1/2} +  p^{-1/4}n^{-1/4} + \sqrt{p} n^{-1/2} \right)  \\
	& \leq 6 {\tilde C(\delta)} (2f_{\max}^{1-\alpha} + 2f_{\max}^{-2\alpha}+ f_{\max} + 1 + 2f_{\max}^{(1-\alpha)/2} +  2f_{\max}^{-\alpha}) \cdot \sqrt{p} := C_3(\delta) \sqrt{p} 
	\end{align*}
	where ${\tilde C(\delta)}$ is the constant specified in Lemma \ref{lemma: donsker class prob bound}.  As a result, combing with (\ref{J iter}), w.p. $1-4\delta$,  the following holds for all $\pi$:
	\begin{align*}
	& \lambda_n J_1^2(\qn)   \leq C_2(\delta) p n^{-1/(1+\alpha)} +  {K_2} n^{-1/(1+\alpha)} J_1(\qn) +  \Rem(\pi)   \\
	%& \leq  C_2(\delta) p n^{-1/(1+\alpha)} +  {K_2} n^{-1/(1+\alpha)} J_1(\qn) +  \Rem(\pi) \\
	& \leq   C_2(\delta) p n^{-1/(1+\alpha)} +  {K_2} n^{-1/(1+\alpha)} J_1(\qn)  +  n^{-1/2}  C_3(\delta) \sqrt{p}   J_1(\qn - \qpi) \\
	& \leq  C_2(\delta) p n^{-1/(1+\alpha)} +  \sup_{\pi \in \Pi} J_1(\tilde H^\pi) C_3(\delta) \sqrt{p} n^{-1/2} + ({K_2} n^{-1/(1+\alpha)}  + n^{-1/2} \sqrt{p} C_3(\delta))J_1(\qn)
	\end{align*}
	Dividing $\lambda_n$ on both sides gives 
	%		& = U_1(\delta, n, p, \mu_n, \lambda_n) + U_2(\delta, n, p, \mu_n )
	\begin{align*}
	J_1^2(\qn) & \leq \frac{C_2(\delta) p n^{-1/(1+\alpha)} +  \sup_{\pi \in \Pi} J_1(\tilde H^\pi) C_3(\delta) \sqrt{p} n^{-1/2}}{\lambda_n} \\
	& \qquad + \frac{{K_2} n^{-1/(1+\alpha)}  + n^{-1/2} \sqrt{p} C_3(\delta)}{\lambda_n} J_1(\qn)
	\end{align*}
	Let $x = J_1(\qn)$ and the above inequality becomes $x^2 \leq a + b x$ for some $a, b > 0$. When $a \leq bx$, we have $x^2 \leq 2bx$, or $x^2 \leq 4b^2$. When $a > bx$, we have $x^2 \leq a+bx \leq 2a$. Thus  $x^2 \leq \max(4b^2, 2a) \leq 2a + 4b^2$.  Now we have
	\begin{align*}
	& J_1^2(\qn) \leq \frac{2C_2(\delta) p n^{-1/(1+\alpha)} +  2\sup_{\pi \in \Pi} J_1(\tilde H^\pi) C_3(\delta) \sqrt{p} n^{-1/2}}{\lambda_n} \\
	& \qquad \qquad + 4\left(\frac{{K_2} n^{-1/(1+\alpha)}  + n^{-1/2} \sqrt{p} C_3(\delta)}{\lambda_n}\right)^2 \\
	& \leq 2L^{-1} C_2(\delta) p +2L^{-1} \sup_{\pi \in \Pi} J_1(\tilde H^\pi) C_3(\delta) \sqrt{p} n^{\frac{1-\alpha}{2(1+\alpha)}} + 8 {K_2^2} L^{-2} + 8pC_3(\delta)^2 n^{\frac{1-\alpha}{1+\alpha}} \\
	& \leq (2L^{-1} C_2(\delta) + 2L^{-1} \sup_{\pi \in \Pi} J_1(\tilde H^\pi) C_3(\delta)  +  8 {K_2^2} L^{-2} + 8C_3(\delta)^2) p n^{\frac{1-\alpha}{1+\alpha}}  := C_4(\delta) p n^{\frac{1-\alpha}{1+\alpha}}  
	\end{align*}
	Now using (\ref{err iter}), w.p. $1-4\delta$ for all $\pi \in \Pi$:
	\begin{align*}
	& \norm{\g(\qn) - \g(\qpi)}^2 \\
	& \leq 2(C_1(\delta) + C_2(\delta))pn^{-1/(1+\alpha)} + 2{K_1} n^{-1/(1+\alpha)} J_1^2(\qn)  + 2{K_2}n^{-1/(1+\alpha)} J_1(\qn) + 2 \Rem(\pi) \\
	& \leq 2(C_1(\delta) + C_2(\delta))pn^{-1/(1+\alpha)} + 2{K_1} n^{-1/(1+\alpha)} J_1^2(\qn)  + 2{K_2}n^{-1/(1+\alpha)} J_1(\qn) \\
	& \qquad + 2 \sup_{\pi \in \Pi} J_1(\tilde H^\pi) C_3(\delta) \sqrt{p} n^{-1/2} + 2 n^{-1/2} C_3(\delta) \sqrt{p} J_1(\qn) \\
	& \leq 2(C_1(\delta) + C_2(\delta))pn^{-1/(1+\alpha)} + 2{K_1} n^{-1/(1+\alpha)}  C_4(\delta) p n^{\frac{1-\alpha}{1+\alpha}} + 2{K_2} n^{-1/(1+\alpha)} \sqrt{C_4(\delta)} \sqrt{p} n^{\frac{1-\alpha}{2(1+\alpha)}} \\
	& \qquad + 2 \sup_{\pi \in \Pi} J_1(\tilde H^\pi) C_3(\delta) \sqrt{p} n^{-1/2}+2 n^{-1/2} C_3(\delta) \sqrt{p} \sqrt{C_4(\delta) p n^{\frac{1-\alpha}{1+\alpha}}}\\
	& = 2(C_1(\delta) + C_2(\delta))pn^{-1/(1+\alpha)} + 2{K_1}C_4(\delta) p n^{\frac{-\alpha}{1+\alpha}} +  2{K_2} \sqrt{C_4(\delta)} \sqrt{p} n^{-1/2} \\& \qquad + 2 \sup_{\pi \in \Pi} J_1(\tilde H^\pi) C_3(\delta) \sqrt{p} n^{-1/2}+2  \sqrt{C_4(\delta)}C_3(\delta) {p} n^{-\frac{\alpha}{1+\alpha}} \\
	& \leq \left(2(C_1(\delta) + C_2(\delta)) + 2{K_1}C_4(\delta) + 2{K_2}\sqrt{C_4(\delta)} + 2 \sup_{\pi \in \Pi} J_1(\tilde H^\pi) C_3(\delta) + 2  \sqrt{C_4(\delta)}C_3(\delta)\right) pn^{-\frac{\alpha}{1+\alpha}} \\
	& := C_5(\delta)pn^{-\frac{\alpha}{1+\alpha}}
	\end{align*}
	Let $C^{[1]}(\delta) = \max(C_4(\delta) +C_5(\delta), 1)$, $\beta_1 = \frac{\alpha}{1+\alpha}$ and $\omega_1 = 0$.  We have shown that with probability at least $1-4\delta$,  the inequalities (\ref{J iter}), (\ref{err iter}) and the followings hold:
	\begin{align*}
	& \norm{\g(\qn) - \g(\qpi)}^2 \leq C^{[1]}(\delta) pn^{-\frac{\alpha}{1+\alpha}} = C^{[1]}(\delta) \iota^{\omega_1} pn^{-\beta_1}\\
	& J_1^2(\qn) \leq C^{[1]}(\delta) pn^{ \frac{1-\alpha}{1+\alpha}} = C^{[1]}(\delta) \iota^{\omega_1} pn^{ \frac{1}{1+\alpha} - \beta_1} 
	\end{align*}

	%	and obtain the improved rates at $\beta_{2} = \frac{\beta_1 + 1/(1+\alpha)}{2}$ and $\omega_{2} = (1+\omega_1)/2$.  That is, for $k = 2$, with probability $1-(3+k)\delta$ for all $\pi \in \Pi$:
	%	\begin{align*}
	%	& \norm{\g(\qn) - \g(\qpi)}^2 \leq C^{[k]}(\delta) \iota^{\omega_{k}} pn^{-\beta_{k}}\\
	%	& J_1^2(\qn) \leq C^{[k]}(\delta) \iota^{\omega_{k}} pn^{1/(1+\alpha) - \beta_{k}}
	%	\end{align*}
	%	\peng{where $C^{[k]}(\delta)$ is a constant depending on bla bla } 

	%	\peng{rewrite below}
	%	
	%		The proof is based on improving the bound of the remainder term, $\Rem(\pi)$.  Denote the event specified in the lemma by $E_n$ thus we have $\Pr(E_n) \geq 1- c\delta$.  

	\textbf{Rate Improvement} \quad
	Let $c = 4, \beta = \beta_1, \omega = \omega_1$ and $C(\delta) = C^{[1]}(\delta)$. Denote by $E_n$ the event that the inequalities (\ref{err iter}), (\ref{J iter}) holds and 
	\begin{align*}
	& \norm{\g(\qn) - \g(\qpi)}^2 \leq C(\delta) \iota^{\omega} pn^{-\beta}\\
	& J_1^2(\qn) \leq C(\delta) \iota^{\omega} pn^{1/(1+\alpha) - \beta}.
	\end{align*}
	We have shown that $\Pr(E_n) \geq 1-c\delta$.  Below we improve the rate by refining the bound of the remainder term, $\Rem(\pi)$.   First, we note that {for the constant, $\iota$, specified in the condition}, under the event $E_n$, 
	\begin{align*}
	P f^2(\qn, \qpi) & \leq  16G_{\max}^2 \EE\Big[\Big(\frac{1}{T}\sum_{t=1}^T [\Delta^\pi(S_t, A_t, S_{t+1}; \qn) - \Delta^\pi(S_t, A_t, S_{t+1}; \qpi)]\Big)^2\Big]\\
	& \leq 16G_{\max}^2 \cdot \iota \norm{\g(\qn) - \g(\qpi)}^2 \\
	& \leq 16 G_{\max}^2 \iota  \cdot C(\delta) \iota^{\omega} pn^{-\beta}  =\iota^{1+\omega}\tilde{C}_1(\delta) p n^{-\beta} \notag;
	\end{align*}
	see Lemma \ref{lemma: key ineq} for the deviation of the second inequality and 
	%	see Derivation of Inequality (B.9) in JASA)
	similarly,
	\begin{align*}
	J_1(\qn - \qpi) & \leq J_1(\qn) + J_1(\tilde H^\pi) \leq \sqrt{C(\delta)} \iota^{\omega/2}\sqrt{p} n^{\frac{1/(1+\alpha) - \beta}{2}} + \sup_{\pi \in \Pi} J_1(\tilde H^\pi) \\
	& \leq  \tilde{C}_2(\delta) \iota^{\omega/2} \sqrt{p} n^{\frac{1/(1+\alpha) - \beta}{2}}
	\end{align*}
	where we define $\tilde C_1(\delta) = G_{\max}^2 C(\delta)$ and $\tilde C_2(\delta) = \sqrt{C(\delta)} + \sup_{\pi \in \Pi} J_1(\tilde H^\pi)$.  Thus under the event $E_n$, we have 
	\begin{align*}
	\Rem (\pi) = \abs{\Pn f(\qn, \qpi)} \leq  n^{-1/2} \sup_{f \in \F_0}\abs{\Gn f}
	\end{align*}
	where $\F_0$ is given by 
	\begin{align*}
	\F_0 = & \Bigdkh{ f: D \mapsto \frac{4}{T}\sum_{t=1}^T \big[H(S_t, A_t) - \sum_{a'} \pi(a'|S_{t+1})    q (S_{t+1}, a')\big] g(S_t, A_t): \pi \in \Pi, g \in \G, H \in \F, \\
		& \hspace{10ex} J_1(q) \leq \tilde{C}_2(\delta)\iota^{\omega/2}\sqrt{p}n^{\frac{1/(1+\alpha) - \beta}{2}},  J_2(g) \leq \sup_{\pi \in \Pi} J_2(e^\pi), P f^2 \leq \tilde{C}_1(\delta)\iota^{1+\omega}pn^{-\beta}}
	\end{align*}
	Apply Lemma \ref{lemma: donsker class prob bound} with $\sigma^2 = \tilde{C}_1(\delta) \iota^{1+\omega} pn^{-\beta}$ and $M = \tilde{C}_2(\delta) \iota^{\omega/2}\sqrt{p}n^{\frac{1/(1+\alpha) - \beta}{2}}$, with probability $1-\delta$, $$\sup_{f \in \F_0}\abs{\Gn f}  \leq {\tilde C(\delta)} \gamma(n, p, M, \sigma),$$
	where ${\tilde C(\delta)}$ is the leading constant in Lemma \ref{lemma: donsker class prob bound}. Note that $\gamma(n, p, M, \sigma)$ can be bounded by
	\begin{align*}
	& \gamma(n, p, M, \sigma) 
	%& \leq [(1 + \tilde{C}_2(\delta)^{\alpha}) \tilde{C}_1(\delta)^{\frac{1-\alpha}{2}}]\iota^{\frac{1-\alpha}{2}}p n^{-\frac{\beta_1(1-\alpha)}{2} +  \frac{\alpha/(1+\alpha) - \alpha\beta_1}{2}} + \frac{1+\tilde{C}_2(\delta)^{2\alpha}}{\tilde{C}_1(\delta)^{\alpha}} p n^{\alpha(1/(1+\alpha) - \beta_1 + \beta_1) -1/2} \\
	%& \qquad + \sqrt{\tilde{C}_1(\delta)} \sqrt{{K}}\sqrt{p} n^{-\beta_1/2} + n^{-1/2} + (1+\tilde{C}_2(\delta)^{\alpha/2}) \sqrt{p} n^{-\frac{1 + \beta_1-\alpha(\beta_1+ (1/(1+\alpha) - \beta_1))}{4}} \\
	%& \qquad + (1+ \tilde{C}_2(\delta)^{\alpha}) \sqrt{p} n^{-\frac{1-\alpha\beta_1 }{2} + \frac{\alpha/(1+\alpha) - \alpha\beta_1}{2}} \\
	\leq [(1 + \tilde{C}_2(\delta)^{\alpha}) \tilde{C}_1(\delta)^{\frac{1-\alpha}{2}}]\iota^{\frac{1-\alpha + \omega}{2}}p n^{-\frac{\beta - \alpha/(1+\alpha)}{2}} + \frac{1+\tilde{C}_2(\delta)^{2\alpha}}{\tilde{C}_1(\delta)^{\alpha}}\iota^{-\alpha}p n^{\alpha/(1+\alpha) -1/2} \\
	& \qquad + \sqrt{\tilde{C}_1(\delta)}\iota^{\frac{1+\omega}{2}}\sqrt{p} n^{-\beta/2} + n^{-1/2} + (1+\tilde{C}_2(\delta)^{\alpha/2})\iota^{\frac{\alpha \omega}{4}}\sqrt{p} n^{-\frac{1 + \beta- \alpha/(1+\alpha)}{4}} \\
	& \qquad + (1+ \tilde{C}_2(\delta)^{\alpha})\iota^{\frac{\alpha \omega}{2}}\sqrt{p} n^{-\frac{1 -\alpha/(1+\alpha) }{2}} \\
	& \leq [(1 + \tilde{C}_2(\delta)^{\alpha}) \tilde{C}_1(\delta)^{\frac{1-\alpha}{2}} + \frac{1+\tilde{C}_2(\delta)^{2\alpha}}{\tilde{C}_1(\delta)^{\alpha}} + \sqrt{\tilde{C}_1(\delta)} + 2+\tilde{C}_2(\delta)^{\alpha/2} + \tilde{C}_2(\delta)^\alpha]\iota^{\frac{1+\omega}{2}} p n^{-\frac{\beta - \alpha/(1+\alpha)}{2}} \\
	& \leq \left[\left(2G_{\max}^{1-\alpha} + J_{\max}^{\alpha} G_{\max}^{1-\alpha} + G_{\max} + 2\right) \sqrt{C(\delta)} + \left(3 + 2G_{\max}^{-2\alpha} + J_{\max}^{2\alpha} + J_{\max}^{\alpha/2} + J_{\max}^\alpha\right)\right] \\
	& \qquad \times\iota^{\frac{1+\omega}{2}} p n^{-\frac{\beta - \alpha/(1+\alpha)}{2}}\\
	& \leq {K_3} (1+\sqrt{C(\delta)})\iota^{\frac{1+\omega}{2}} p n^{-\frac{\beta - \alpha/(1+\alpha)}{2}}
	%\\%\leq \left[(1 + \tilde{C}_2(\delta)^{\alpha}) \tilde{C}_1(\delta)^{\frac{1-\alpha}{2}} + \frac{1+\tilde{C}_2(\delta)^{2\alpha}}{\tilde{C}_1(\delta)^{\alpha}} + \sqrt{\tilde{C}_1(\delta)} + 2+\tilde{C}_2(\delta)^{\alpha/2} + \tilde{C}_2(\delta)^\alpha\right] p n^{-\frac{\beta_1 - \alpha(2\gamma_1 + \beta_1)}{2}}
	\end{align*}
	%Note that $2\gamma_1 + \beta_1 = \frac{1}{1+\alpha}$ and $\beta_1 < 1/(1+\alpha)$.  We then have 
	%\begin{enumerate}
	%	\item $-\frac{\beta_1(1-\alpha)}{2} + \alpha \gamma_1 - (\alpha(2\gamma_1 + \beta_1) -1/2) = \frac{1/(1+\alpha) - \beta_1}{2} > 0$
	%	\item $-\frac{\beta_1(1-\alpha)}{2} + \alpha \gamma_1 + \beta_1/2 = \frac{\alpha}{2(1+\alpha)} > 0$
	%	\item $-\frac{\beta_1(1-\alpha)}{2} + \alpha \gamma_1  + 1/2 = \frac{1+\alpha/(1+\alpha) - \beta_1}{2} > 0$
	%	\item $ -\frac{\beta_1(1-\alpha)}{2} + \alpha \gamma_1 +\frac{1 + \beta_1-\alpha(\beta_1+ 2\gamma_1)}{4} = \frac{1+\alpha/(1+\alpha) - \beta_1}{4} > 0$
	%	\item $ -\frac{\beta_1(1-\alpha)}{2} + \alpha \gamma_1 -(-\frac{1-\alpha\beta_1 }{2} + \alpha \gamma_1) = \frac{1-\beta_1}{2} > 0$
	%\end{enumerate}
	{where the constant ${K_3}$ depends on $J_{\max} := C_2 + C_3 \sup_{\pi \in \Pi} J_1(\tilde H^{\pi}) , F_{\max}, G_{\max}$}. We then have w.p. $1-(c+1)\delta$ for all $\pi$
	\begin{align*}
	\Rem (\pi) \leq \left({K_3} {\tilde C(\delta)} (1+\sqrt{C(\delta)})\right)\iota^{\frac{1+\omega}{2}}p n^{-\frac{\beta - \alpha/(1+\alpha)+1}{2}} = \tilde C_3(\delta)\iota^{\frac{1+\omega}{2}} p n^{-\frac{\beta + 1/(1+\alpha)}{2}}
	\end{align*}
	Combing with (\ref{J iter}), which holds under the event $E_n$, we have 
	\begin{align*}
	\lambda_n J_1^2(\qn) & \leq \bar {C}(\delta) p n^{-1/(1+\alpha)} + \tilde C_3(\delta) \iota^{\frac{1+\omega}{2}}p n^{-\frac{\beta+ 1/(1+\alpha)}{2}} +  {K_2} n^{-1/(1+\alpha)} J_1(\qn) \\
	&  \leq (\bar {C}(\delta) + \tilde C_3(\delta)) \iota^{\frac{1+\omega}{2}} p n^{-\frac{\beta + 1/(1+\alpha)}{2}} +  {K_2} n^{-1/(1+\alpha)} J_1(\qn) 
	\end{align*}
	Thus using the same argument as in the proof of Lemma \ref{lemma: upper level initial bound} gives 
	\begin{align*}
	J_1^2(\qn) & \leq 2 (\bar {C}(\delta) + \tilde C_3(\delta)) L^{-1} \iota^{\frac{1+\omega}{2}} p n^{-\frac{\beta+ 1/(1+\alpha)}{2} + \frac{1}{1+\alpha}} + 4 {K_2}^2 L^{-1} n^{-1/(1+\alpha)}\\
	& \leq \tilde C_4(\delta) \iota^{\frac{1+\omega}{2}} p n^{\frac{ 1/(1+\alpha) - \beta}{2}}
	\end{align*}
	where $	\tilde C_4(\delta) = (2 (\bar {C}(\delta)+ \tilde C_3(\delta)) + 4{K_2}^2) L^{-1}$.  Now using (\ref{err iter}) (again this inequality holds under the event $E_n$) gives
	\begin{align*}
	& \norm{\g(\qn) - \g(\qpi)}^2 \\
	& \leq 4\bar {C}(\delta)pn^{-1/(1+\alpha)} +2{K_1} n^{-1/(1+\alpha)} J_1^2(\qn) + 2{K_2} n^{-1/(1+\alpha)} J_1(\qn)  + 2 \Rem(\pi) \\
	& \leq  4\bar {C}(\delta) pn^{-1/(1+\alpha)}  +2{K_1} n^{-1/(1+\alpha)}  \tilde C_4(\delta)\iota^{\frac{1+\omega}{2}}  p n^{\frac{ 1/(1+\alpha) - \beta}{2}}  \\
	& \qquad  + 2{K_2} n^{-1/(1+\alpha)} \sqrt{\tilde C_4(\delta)}\iota^{\frac{1+\omega}{4}} \sqrt{p} n^{\frac{ 1/(1+\alpha) - \beta}{4}} +2 \tilde C_3(\delta)\iota^{\frac{1+\omega}{2}}  p n^{-\frac{\beta + 1/(1+\alpha)}{2}} \\
	& \leq \tilde C_5(\delta)\iota^{\frac{1+\omega}{2}}  p n^{-\frac{\beta + 1/(1+\alpha)}{2}}
	\end{align*}
	where $\tilde C_5(\delta) = 2[2\bar {C}(\delta)+ 2{K_1}\tilde C_4(\delta) + \frac{{K_2}^2}{{K_1}}+ \tilde C_3(\delta)]$. 
	Note that the two constants $\tilde C_4(\delta), \tilde C_5(\delta)$ can be simply replaced by $C_{new}(\delta) = {K_4} {\tilde C(\delta)} (1 + \sqrt{C(\delta)})$ {for some constant ${K_4}$} by noting 
	%{K_3}\bar C(\delta) (1+\sqrt{C(\delta)})
	\begin{align*}
	& \tilde C_4(\delta) + \tilde C_5(\delta) \\
	& = 4{\tilde C(\delta)} + \frac{2{K_2}^2}{{K_1}}+ 2\tilde C_3(\delta)  + (4{K_1} + 1)\tilde C_4(\delta) \\
	& = 4{\tilde C(\delta)}+ \frac{2{K_2}^2}{{K_1}}+ 2\tilde C_3(\delta) + (4{K_1} + 1) L^{-1}  (2 ({\tilde C(\delta)}+ \tilde C_3(\delta)) + 4{K_2}^2)  \\
	& = (4 + 2(4{K_1} + 1) L^{-1} ){\tilde C(\delta)}+ 4(4{K_1} + 1) L^{-1}{K_2}^2   + \frac{2{K_2}^2}{{K_1}}+ (2 + 2(4{K_1} + 1) L^{-1} )\tilde C_3(\delta) \\
	& = (4 + 2(4{K_1} + 1) L^{-1} ){\tilde C(\delta)}+ 4(4{K_1} + 1) L^{-1}{K_2}^2   + \frac{2{K_2}^2}{{K_1}} \\
	& \qquad + (2 + 2(4{K_1} + 1) L^{-1} ) \left({K_3}{\tilde C(\delta)} (1+\sqrt{C(\delta)})\right) \\
	& = (4 + 2(4{K_1} + 1) L^{-1} ){\tilde C(\delta)}+ 4(4{K_1} + 1) L^{-1}{K_2}^2   + \frac{2{K_2}^2}{{K_1}} \\
	& \qquad + (2 + 2(4{K_1} + 1) L^{-1} ) \left({K_3}{\tilde C(\delta)} (1+\sqrt{C(\delta)})\right) \\
	& \leq {K_4}{\tilde C(\delta)} (1 + \sqrt{C(\delta)}) 
	\end{align*}
	We now obtain that the following holds w.p. $1-(c+1)\delta$, for all $\pi \in \Pi$
	\begin{align*}
	& \norm{\g(\qn) - \g(\qpi)}_2^2 \leq C_{new} (\delta) \iota^{\frac{1+\omega}{2}} p n^{-\frac{\beta + 1/(1+\alpha)}{2}}
	\\
	& J_1^2(\qn) \leq  C_{new}  (\delta) \iota^{\frac{1+\omega}{2}} p n^{\frac{ 1/(1+\alpha) - \beta}{2}} 
	\end{align*}
	Thus the convergence rate is improved to $\beta_{2} = \frac{\beta_1 + 1/(1+\alpha)}{2}$ and $\omega_{2} = (1+\omega_1)/2$. The same procedure can be applied $k$ times. It is easy to verify that  for any $k \geq 1$,  
	$\beta_{k+1} = \frac{\beta_k + 1/(1+\alpha)}{2} = \frac{1}{1+\alpha} -  \frac{(1-\alpha) 2^{-k}}{1+\alpha} $ and  $\omega_{k+1} = (1+\omega)/2 = 1 - 2^{-k}$ and thus the desired result. 
\end{proof}

\begin{proof}[Proof of Theorem \ref{thm: ratio}]
	
	Recall the ratio estimator, $\hat \omega^\pi_n$ in (4.10).  From Theorem \ref{thm: pre-ratio} and Lemma \ref{lemma: uniform l2 bound}, w.p. $1-(3+k)\delta$ for all $\pi \in \Pi$ we have
	\begin{align*}
	& \norm{\hat e^\pi_n - e^\pi}^2  = \norm{\gn(\qn) - \g(\qn) + \g(\qn) -  \g(\qpi)}^2 \\
	& \leq 2\norm{\gn(\qn) - \g(\qn)}^2 + 2\norm{\g(\qn) -  \g(\qpi)}^2 \\
	& \leq C_1(\delta)pn^{-1/(1+\alpha)} + {K_1} n^{-1/(1+\alpha)} J_1^2(\qn) + C^{[k]}(\delta) \iota^{\omega_{k}} pn^{-\beta_{k}} \\
	& \leq C_1(\delta)pn^{-1/(1+\alpha)} + {K_1} n^{-1/(1+\alpha)} C^{[k]}(\delta) \iota^{\omega_{k}} pn^{1/(1+\alpha) - \beta_{k}} + C^{[k]}(\delta) \iota^{\omega_{k}} pn^{-\beta_{k}} \\
	& = C_1(\delta)pn^{-1/(1+\alpha)} + ({K_1} + 1)C^{[k]}(\delta) \iota^{\omega_{k}} pn^{-\beta_{k}},
	\end{align*}
	where, as in the beginning of the proof of Theorem \ref{thm: pre-ratio}, ${K_1} = {K_0} L (1+ 2C_2^2)$ and ${K_0}$ is the leading constant in Lemma \ref{lemma: uniform l2 bound}. 
	As such, we have $$\norm{\hat e^\pi_n - e^\pi} \leq C_1^{[k]}(\delta)  \iota^{\omega_{k}/2} \sqrt{p}n^{-\beta_{k}/2} $$ 
	For simplicity, let $f(g)(D) = (1/T)\sum_{t=1}^T g(S_t, A_t)$. On the other hand,	
	\begin{align*}
	\abs{\Pn f(\gn(\qn)) - P f(e^\pi)} & = \abs{\Pn f( \gn(\qn)) - P f(\gn(\qn)) + P f(\gn(\qn)) - P f(e^\pi)} \\
	& \leq  \abs{(\Pn -P) f(\gn(\qn))} + P \abs{ f(\gn(\qn)) - f(e^\pi)}  \\
	& \leq \abs{(\Pn -P) f(\gn(\qn))} + \norm{\gn(\qn) - e^\pi}\\
	& =  \abs{(\Pn -P) f(\gn(\qn))} + \norm{\hat e^\pi_n - e^\pi}.
	\end{align*}
	%			We use the empirical process, Lemma \ref{lemma: donsker class prob bound}, to bound the first term.   \peng{At the minimum we should get the same rate as the second term. }
	Consider the first term: \begin{align*}
	\abs{(\Pn -P) f(\gn(\qn))}& \leq n^{-1/2} J_2(\gn(\qn)) \frac{\Gn f(\gn(\qn))}{J_2(\gn(\qn))} \\&
	\leq n^{-1/2} J_2(\gn(\qn))  \sup_{g \in\G_1}  \abs{\Gn f(g)},
	\end{align*}
	where $\G_1 = \{ f(g): g \in \G,  J_2(g) \leq 1 \}$.  Using Lemma \ref{lemma: uniform l2 bound}, w.p $1-\delta$ for all $\pi$, 
	\begin{align*}
	& J_2(\gn(\qn))  \leq {K_0}\left(1 + J_1(\qn) +  J_2(\g (\qn)) + \sqrt{\frac{p}{n \mu_n^{\alpha+1}} }+ \sqrt{\frac{1}{n \mu_n}} +  \sqrt{\frac{\log(1/\delta)}{n\mu_n}}\right)\\
	& \leq {K_0}\left(1 + J_1(\qn) +  C_2 + C_3 J_1(\qn) + \sqrt{pL^{-\alpha-1}}+  L^{-1/2}(1+\sqrt{\log(1/\delta)})n^{-\frac{\alpha}{2(1+\alpha)}} \right)
	\end{align*}
	where ${K_0}$ is the leading constant in Lemma \ref{lemma: uniform l2 bound}. 
	Combining with the bound on $J_1(\qn)$ in Theorem \ref{thm: pre-ratio} gives that $$J_2^2(\gn(\qn)) \leq C^{[k]}_2(\delta) \iota^{\omega_{k}} pn^{1/(1+\alpha) - \beta_{k}}.$$ 
	
	{Similar to Lemma \ref{lemma: donsker class prob bound}, we can apply Talagrand's inequality to show that $\sup_{g \in\G_1}  \abs{\Gn f(g)} \leqconst C(\delta)$. As a result, we have
		\begin{align*}
		\abs{(\Pn -P) f(\gn(\qn))}  \leq C_3^{[k]}(\delta) \iota^{\omega_k/2}\sqrt{p}n^{-(\alpha/(1+\alpha)+ \beta_k)/2}.
		\end{align*}
		Together with the bound on $\norm{\hat e^\pi_n - e^\pi}$, we have
		\begin{align*}
		\abs{\Pn f(\gn(\qn)) - P f(e^\pi)} \leq C_4^{[k]}(\delta) \iota^{\omega_k/2}\sqrt{p}n^{-\beta_{k}/2}
		\end{align*}
	}
	
	Suppose $n$ is large enough such that the RHS  satisfies $C_4^{[k]}(\delta) \iota^{\omega_k/2}\sqrt{p}n^{-\beta_{k}/2} < (1/2) P f(e^\pi)$, then $\abs{\Pn f(\gn(\qn)) - P f(e^\pi)} \leq (1/2) Pf(e^\pi)$. Thus we have $\Pn f(\gn(\qn)) \geq (1/2) P f(e^\pi) > 0$ and 
	\begin{align*}
	\norm{\hat \omega^\pi_n - \omega^\pi} & \leq \frac{\norm{\hat e^\pi_n - e^\pi}}{\abs{\Pn f(\hat e^\pi_n)}} + \norm{e^\pi} \times \Bigabs{\frac{1}{\Pn f(\hat e^\pi_n)} - \frac{1}{P f(e^\pi)}} \\
	& \leq \frac{2}{P f(e^\pi)} \norm{\hat e^\pi_n - e^\pi} +\frac{2 \norm{e^\pi} }{(P f(e^\pi))^2} \times   \abs{\Pn f(\gn(\qn)) - P f(e^\pi)} 
	\end{align*}
	Recall that $\norm{\omega^\pi}^2 = \int \omega^\pi(s, a) d^\pi(s, a) > 1$. As a result, $P f(e^\pi)  = \frac{1}{\norm{\omega^\pi}^2} = \norm{e^\pi}^2$. Finally we have 
	\begin{align*}
	\norm{\hat \omega^\pi_n - \omega^\pi}  
	& \leq 2 \norm{\omega^\pi}^2 \cdot \norm{\hat e^\pi_n - e^\pi} + 2 \norm{\omega^\pi}^3 \cdot   \abs{\Pn f(\gn(\qn)) - P f(e^\pi)} \\ 
	& \leq 2 \norm{\omega^\pi}^3  \left(\norm{\hat e^\pi_n - e^\pi} + \abs{\Pn f(\gn(\qn)) - P f(e^\pi)}\right) \\
	& \leq 4 \sup_{\pi \in \Pi} \{\norm{\omega^\pi}^3\} (C_1^{[k]}(\delta) +  C_4^{[k]}(\delta)) \iota^{\omega_{k}/2} \sqrt{p} n^{-\beta_{k}/2} 		
	\end{align*}

\end{proof}

\begin{lemma}[Lower level]
	\label{lemma: uniform l2 bound}
	Let $\norm{g}_{n}^2 = \Pn [1/T\sum_{t=1}^{T} g^2(S_t, A_t)]$.  Suppose Assumptions (4-1), (4-2), (6-1), (6-2), (6-4), (4-4) and (3-2) hold. Then, $\Pr(E_n) \geq 1-\delta$, where the event $E_n$ is that  for all $(q, \pi) \in \F \times \Pi$, the followings hold
	%		with probability at least $1-\delta$, the following event,  $E_n$, hold  and $\alpha$ : 
	\begin{align*}
	& \norm{\gn (q) - \g (q)}_{}^2 \leqconst \mu_n (1+J_1^2(q) + J_2^2(\g (q))) + \frac{p}{n \mu_n^{\alpha}} + \frac{1}{n} +  \frac{\log(1/\delta)}{n}\\
	&  J_2(\gn(q))\leqconst  1 + J_1(q) +  J_2(\g (q)) + \sqrt{\frac{p}{n \mu_n^{\alpha+1}} }+ \sqrt{\frac{1}{n \mu_n}} +  \sqrt{\frac{\log(1/\delta)}{n\mu_n}}\\
	& \norm{\gn(q) - \g(q)}_{n}^2 \leqconst \mu_n (1+J_1^2(q) + J_2^2(\g (q))) + \frac{p}{n \mu_n^{\alpha}} + \frac{1}{n} +  \frac{\log(1/\delta)}{n}
	\end{align*}
	where the leading constant only depends on $F_{\max}, G_{\max}, L_\Theta, \text{diam}(\Theta), C_3, \alpha$
	
\end{lemma}

\begin{proof}[Proof of Lemma \ref{lemma: uniform l2 bound}]
	Recall that for any $(s, a, s')$ and $H \in \F$, $\Delta^\pi(s, a, s'; H) = 1 - q(s, a) + \sum_{a'} \pi(a'|s') q(s', a')$. For simplicity,  let $\epsilon_t^\pi(q) = \Delta^\pi(S_t, A_t, S_{t+1}; H)$. 	
	%	We start with decomposing the error $\norm{\gn(q) - \g(H) }_{}^2$: 
	%		\begin{align*}
	%		& \norm{\gn(q) - \g(H) }_{}^2  
	%		= (1/T) \sum_{t=1}^{T}\EE\left[(\gn (S_t, A_t; H) - \g(S_t, A_t; H))^2\right] \\
	%		& =  (1/T) \sum_{t=1}^{T}\EE\left[ (\gn (S_t, A_t; H) - \epsilon_t^\pi(q) + \epsilon_t^\pi(q) -  \g(S_t, A_t; H))^2\right]\\
	%		& =(1/T) \sum_{t=1}^{T} \EE\left[ (\epsilon_t^\pi(q) - \gn (S_t, A_t; H))^2\right] + (1/T) \sum_{t=1}^{T} \EE\left[ (\epsilon_t^\pi(q) -  \g(S_t, A_t; H))^2\right] \\
	%		& \qquad \quad  +  (2/T) \sum_{t=1}^{T}\EE\left[ \gn(S_t, A_t; H) - \epsilon_t^\pi(q)) (\epsilon_t^\pi(q) -  \g(S_t, A_t; H))\right]
	%		\end{align*}
	First we note that  for all $g \in \G$, $ \sum_{t=1}^{T}\EE\left[(\Delta^\pi(S_t, A_t; H) -  \g(S_t, A_t; H)) g(S_t, A_t)\right] = 0$ because of the optimizing property of $\g$. It can then be shown that 
	%		\begin{align*}
	%		&  \sum_{t=1}^{T} \EE\left[ \gn(S_t, A_t; H) - \g(S_t, A_t; H) + \g(S_t, A_t; H)- \epsilon_t^\pi(q)) (\epsilon_t^\pi(q) -  \g(S_t, A_t; H))\right]\\
	%		& = \sum_{t=1}^{T} \EE\left[ (\g(S_t, A_t;\eta, Q)- \epsilon_t^\pi(q)) (\epsilon_t^\pi(q) -  \g(S_t, A_t; H))\right]\\
	%		& = - 2 \sum_{t=1}^{T}\EE\left[ (\epsilon_t^\pi(q) -  \g(S_t, A_t; H))^2\right]
	%		\end{align*}
	%		Thus we have 
	\begin{align*}
	\norm{\gn(q) - \g(H) }_{}^2 =  \EE\left[ \frac{1}{T} \sum_{t=1}^{T} (\epsilon_t^\pi(q) - \gn(S_t, A_t; H))^2 -  (\epsilon_t^\pi(q) -  \g(S_t, A_t; H))^2\right]
	\end{align*}
	For $g_1, g_2 \in \G, q \in  \F, \pi \in \Pi$,  we introduce
	\begin{align*}
	& f_1(g_1, g_2, q, \pi): D \mapsto \frac{1}{T} \sum_{t=1}^{T} (\epsilon_t^\pi(q) - g_1(S_t, A_t))^2 -  (\epsilon_t^\pi(q) -  g_2(S_t, A_t))^2\\
	& f_2(g_1, g_2, q, \pi): D \mapsto \frac{1}{T} \sum_{t=1}^{T} (\epsilon_t^\pi(q) - g_2(S_t, A_t))(g_1(S_t, A_t) - g_2(S_t, A_t)) \\
	& J^2(g_1, g_2, q) = J_2^2(g_1) + (2/3)J_2^2(g_2) + (2/3)J_1^2(H)
	\end{align*}
	We then have 
	\begin{align*}
	& \norm{\gn(q) - \g(H) }_{}^2 = \EE f_1(\gn(q), \g(H), q, \pi)  \\
	& \norm{\gn(q) - \g(H) }_{n}^2 = \Pn f_1(\gn(q), \g(H), q, \pi)  + 2 \Pn f_2(\gn(q), \g(H), q, \pi)
	\end{align*}
	Thus we have $\norm{\gn(q) - \g(H) }_{}^2 + \norm{\gn(q) - \g(H) }_{n}^2 + \mu_n J_2^2(\gn(q))  = I_1(q, \pi) + I_2(q, \pi)$, where 
	\begin{align*}
	& I_1(q, \pi) = 3 \big(\Pn f_1(\gn(q), \g(H), q, \pi) + \mu_n J^2(\gn(q), \g(H), q)\big)\\
	& I_2(q, \pi) =  (\Pn + P) f_1(\gn(q), \g(H), q, \pi) + \mu_n J_2^2(\gn(q)) \\
	& \qquad \qquad \qquad + 2 \Pn f_2(\gn(q), \g(H), q, \pi)- I_1(q,  \pi)
	\end{align*}
	For the first term, the optimizing property of $\gn(q)$ implies that 
	\begin{align*}
	& (1/3)I_1(q, \pi)  = \Pn f_1(\gn(q), \g(H), q, \pi) + \mu_n J^2(\gn(q), \g(H), q)\\
	& =  \Pn \big[ (1/T) \sum_{t=1}^{T} (\epsilon_t^\pi(q) - \gn(S_t, A_t; H))^2 -  (\epsilon_t^\pi(q) -  \g(S_t, A_t; H))^2\big] \\
	& \qquad + \mu_n J_2^2(\gn(q)) + (2/3)\mu_n J_2^2 (\g(H)) +(2/3)\mu_n J_1^2(H)\\
	& = \Big[\Pn \big[ (1/T) \sum_{t=1}^{T} (\epsilon_t^\pi(q) - \gn(S_t, A_t; H))^2\big]+ \mu_n J^2(\gn(q)) \Big] \\
	&  \qquad - \Pn \big[ (1/T) \sum_{t=1}^{T}  (\epsilon_t^\pi(q) -  \g(S_t, A_t; H))^2\big] + (2/3) \mu_n J_2^2 (\g(H))+ (2/3)  \mu_n J_1^2(H)  \\
	& \leq (5/3) \mu_n J_2^2 (\g(H))+ (2/3) \mu_n J_1^2(H)
	\end{align*}
	Thus, $I_1(q, \pi) \leq 5 \mu_n J_2^2 (\g(H))+ 2\mu_n J_1^2(H)$ holds for all $(q, \pi)$.

	Next,  we derive the uniform bound of $I_2(q, \pi)$ over all $(q, \pi) \in \F \times \Pi$ using the peeling device and the exponential inequality of the relative deviation of the empirical process (\cite{gyorfi2006distribution}, thm. 19.3). Note that $ \EE f_2(\gn(q), \g(H), q, \pi) = 0$. We can then rewrite $I_2(q, \pi)$ as 
	\begin{align*}
	I_2 (q, \pi)   & =  2 (P - \Pn) f(\gn(q), \g(H), q, \pi) - P f(\gn(q), \g(H), q, \pi) \\
	& - 2\mu_n (J_2^2(\gn(q)) + J_2^2(\g(H)) + J_1^2(H))
	\end{align*}
	where we introduce $f = f_1 - f_2$. 
	Fix some $t > 0$. 
	\begin{align*}
	& \Pr\left(\exists (q, \pi) \in \F \times \Pi,  I_2(q, \pi) > t\right)\\
	& = \sum_{l=0}^{\infty} \Pr\Big(\exists (q, \pi) \in  \F\times \Pi,  ~ 2\mu_n \left[J_2^2(\gn(q)) + J_2^2(\g(H)) +  J_1^2(H)\right] \in ~ [2^l t \indicator{l\neq 0}, 2^{l+1}t),   \\
	& \hspace{14ex} 2(P-\Pn) f(\gn(q), \g(H), q, \pi) > P f(\gn(q), \g(H), q, \pi) \\
	& \hspace{20ex}+ 2\mu_n \left[J_2^2(\gn(q)) + J_2^2(\g(H)) +  J_1^2(H)\right] + t\Big)\\
	& \leq \sum_{l=0}^{\infty} \Pr\Big(\exists (q, \pi) \in \F\times \Pi,  ~ 2\mu_n \left[J_2^2(\gn(q)) + J_2^2(\g(H)) +  J_1^2(H)\right] \leq 2^{l+1}t,   \\
	& \hspace{14ex} 2(P-\Pn) f(\gn(q), \g(H), q, \pi) > P f(\gn(q), \g(H), q, \pi) + 2^l t\Big)\\
	& \leq \sum_{l=0}^\infty \Pr\left( \sup_{f \in \F_l}  \frac{(P-\Pn) f(D)}{P f(D) + 2^l t} > \frac{1}{2} \right)
	\end{align*}
	where $$\F_l = \{ f(g, \g(H), q, \pi): J_2^2(g) \leq \frac{2^{l} t}{\mu_n}, J_2^2(\g(H)) \leq \frac{2^{l} t}{\mu_n},  J_1^2(H) \leq \frac{2^{l} t}{\mu_n}, H \in \F, \pi \in \Pi\}$$ Next we verify the conditions (A1 - A4) in Theorem 19.3 in \citep{gyorfi2006distribution} with $\F = \F_l$, $\epsilon = 1/2$ and $\eta = 2^{l} t$ to get an exponential inequality for each term in the summation, similar to the proof of Lemma \ref{lemma: second term} below.   
	
	The (A1) and (A2) conditions are easy to verify using Assumptions (4-1), (6-1) and (6-2). For (A1), it is easy to see that 
	\begin{align*}
	\norm{f(g, \g(H), q, \pi)}_\infty \leq  6G_{\max}(G_{\max} + 1+2F_{\max}), \forall f \in \F
	\end{align*}
	and thus $K_1 = 6G_{\max}(G_{\max} + 1+2F_{\max})$. For (A2), note that 
	\begin{align*}
	\EE[f(g, \g(H), q, \pi)^2] \leq 2\EE [f_1(g, \g(H), q, \pi)(D)^2] + 2\EE [f_2(g, \g(H), q, \pi)(D)^2]
	\end{align*}
	For the first term:
	\begin{align*}
	& \EE [f_1(g, \g(H), q, \pi)(D)^2]  \\
	& = \EE[(\frac{1}{T} \sum_{t=1}^{T} (\epsilon_t^\pi(q) - g(S_t, A_t))^2 -  (\epsilon_t^\pi(q) - \g(S_t, A_t; H))^2)^2] \\
	& \leq \EE[\frac{1}{T} \sum_{t=1}^{T} \left((\epsilon_t^\pi(q) - g(S_t, A_t))^2 -  (\epsilon_t^\pi(q) - \g(S_t, A_t; H))^2\right)^2] \\
	& = \EE[\frac{1}{T} \sum_{t=1}^{T} \left(2\epsilon_t^\pi(q) - g(S_t, A_t)  - \g(S_t, A_t; H)\right)^2 \left(\g(S_t, A_t; H) - g(S_t, A_t)\right)^2] \\
	& \leq \left(2(1 + 2 F_{\max}) + 2G_{\max}\right)^2   \EE[\frac{1}{T} \sum_{t=1}^{T} \left(\g(S_t, A_t; H) - g(S_t, A_t)\right)^2] \\
	& = 4\left(1 + 2 F_{\max} + G_{\max}\right)^2 \EE \left[f_1(g, \g(H), q, \pi)(D)\right]
	\end{align*}
	and the second term:
	\begin{align*}
	&  \EE [f_2(g, \g(H), q, \pi)(D)^2] \\
	& = \EE[(\frac{1}{T} \sum_{t=1}^{T} (\epsilon_t^\pi(q) - \g(S_t, A_t; H))(g(S_t, A_t) - \g(S_t, A_t; H)))^2] \\
	& \leq  \EE[\frac{1}{T} \sum_{t=1}^{T} (\epsilon_t^\pi(q) - \g(S_t, A_t; H))^2(g(S_t, A_t) - \g(S_t, A_t; H))^2] \\
	& \leq \left(1 + 2 F_{\max} + G_{\max}\right)^2 \EE[\frac{1}{T} \sum_{t=1}^{T} \left(\g(S_t, A_t; H) - g(S_t, A_t)\right)^2] \\
	&  = \left(1 + 2 F_{\max} + G_{\max}\right)^2 \EE [f_1(g, \g(H), q, \pi)(D)]
	\end{align*}
	Recall that $\EE [f_2(g, \g(H), q, \pi)(D)] = 0$. Putting together implies that the condition (A2) is satisfied  with $K_2 = 10\left(1 + 2 F_{\max} + G_{\max}\right)^2$, that is, $$\EE[f(g, \g(H), q, \pi)^2] \leq  K_2\EE [f(g, \g(H), q, \pi)(D)].$$ 
	
	To ensure the condition (A3) holds for every $l$, i.e., $\sqrt{n} \epsilon\sqrt{1-\epsilon} \sqrt{\eta} \geq 288 \max(K_1, \sqrt{2K_2})$ (recall $\epsilon = 1/2$ and $\eta = 2^{l} t$), 
	we just need to ensure  the inequality holds for $l = 0$, i.e.
	\[
	\sqrt{n} (1/2)^{3/2} \sqrt{t} \geq 288 \max(K_1, \sqrt{2K_2})
	\]
	That is, $t \geq c_1 n^{-1}$ where $c_1 = 8(288 \max(K_1, \sqrt{2K_2}))^2$. 
	
	Next we verify the condition (A4). 
	It is straightforward to see that with $M = \sqrt{\frac{2^l t}{\mu_n}}$, we have 
	\begin{align*}
	& N(\epsilon(6 + 12 F_{\max} + 24 G_{\max}), \F, \norm{\cdot}_n) \\
	& \leq N(\epsilon, \F_{M}, \norm{\cdot}_\infty) N^2(\epsilon, \G_{M}, \norm{\cdot}_\infty) N(\epsilon, \V(\Pi, \F_M), \norm{\cdot}_\infty)
	\end{align*}
	where $\V(\Pi, \F_M) = \{ s\mapsto \sum_{a} \pi(a|s) q(s, a): H \in \F_M, \pi \in \Pi \}$ is a class of state-only function depending on the policy class and the function class $\F_{M}$.  
	Let $c_2 = 6 + 12 F_{\max} + 24 G_{\max}$. As a result of the entropy condition in Assumption (4-4), we have 
	\begin{align*}
	\log N(\epsilon, \F, \norm{\cdot}_n) 
	& \leq \log N(\epsilon/c_2, \F_{M}, \norm{\cdot}_\infty) + 2 \log N(\epsilon/c_2, \G_{M}, \norm{\cdot}_\infty) \\
	& \qquad + \log N(\epsilon/c_2, \V(\Pi, \F_M), \norm{\cdot}_\infty)\\
	& \leq 3C_1  (c_2M/\epsilon)^{2\alpha} +  \log N(\epsilon/c_2, \V(\Pi, \F_M), \norm{\cdot}_\infty)
	\end{align*}
	where $C_1$   is the constant in Assumption (4-4). Note that when $t \geq \mu_n$, it can be seen that for every $l \geq 0$, $M = \sqrt{\frac{2^l t}{\mu_n}} \geq 1$.  Under the metric entropy assumption (4-4) and the  assumptions (3.1) and (3-2), it can be shown that for $\tilde C_1  = C_1  + p (\text{diam}(\Theta)L_\Theta F_{\max})^{2\alpha}/(2\alpha)$ (the constant $C_1$  is specified in Assumption (4-4)), $\log N(\epsilon, \V(\Pi, \F_M), \norm{\cdot}_\infty) \leq \tilde C_1  (M/\epsilon)^{2\alpha}$. As a result,  for a constant $\tilde C_3$ depending on the policy class $\Pi$, $F_{\max}$ and $\alpha$, we have
	\begin{align*}
	& \log N(\epsilon, \F, \norm{\cdot}_n)  \leq 3C_1  (c_2M/\epsilon)^{2\alpha} +  \tilde C_1  (c_2M/\epsilon)^{2\alpha} \\ & = c_2^{2\alpha} (3C_1  + \tilde C_3)  \left(\frac{2^l t}{\mu_n}\right)^{\alpha} \epsilon^{-2\alpha} = c_3  \left(\frac{2^l t}{\mu_n}\right)^{\alpha} \epsilon^{-2\alpha}
	\end{align*}
	where we define $c_3 = c_2^{2\alpha } (3C_1  + \tilde C_3)$. 
	Now we see the condition (A4) is true  if the following inequality holds for all $x \geq 2^lt/8$, 
	\begin{align*}
	\frac{\sqrt{n} (1/2)^2 x}{96\sqrt{2} \max(K_1, 2K_2)} 
	\geq \int_{0}^{\sqrt{x}} \sqrt{c_3} \left(\frac{2^l t}{\mu_n}\right)^{\alpha/2} u^{-\alpha} du = x^{\frac{1-\alpha}{2}} \sqrt{c_3} \left(\frac{2^l t}{\mu_n}\right)^{\alpha/2}.
	\end{align*}
	Or equivalently $	x^{\frac{1+\alpha}{2}} \geq 4 \cdot 96\sqrt{2} \max(K_1, 2K_2) \sqrt{c_3}  \left(\frac{2^l t}{\mu_n}\right)^{\alpha/2} n^{-1/2}$. 
	Clearly we only need to ensure  the inequality holds when $x$ is at the minimum. That is,  below is sufficient for the condition (A4) to hold:
	\begin{align*}
	(2^l t/8)^{\frac{1+\alpha}{2}} \geq 4 \cdot 96\sqrt{2} \max(K_1, 2K_2) \sqrt{c_3}  \left(\frac{2^l t}{\mu_n}\right)^{\alpha/2} n^{-1/2} \\
	\iff (2^l t)^{1/2} \geq 8^{\frac{1+\alpha}{2}} \cdot 4 \cdot 96\sqrt{2} \max(K_1, 2K_2) \sqrt{c_3}  (\mu_n^\alpha n)^{-1/2}
	\end{align*}
	To ensure the above holds for all $l \geq 0$, we require $t$ to satisfy
	\begin{align*}
	t \geq \left(8^{\frac{1+\alpha}{2}} \cdot 4 \cdot 96\sqrt{2} \max(K_1, 2K_2) \sqrt{c_3}\right)^2  (\mu_n^\alpha n)^{-1} 
	\end{align*}
	Or, simply requiring $t \geq c_4  (\mu_n^\alpha n)^{-1}$ where $c_4 = c_3 18 (32)^4  \max(K_1^2, 4K_2^2)$.
	
	To summarize, the conditions (A1-A4) would be satisfied for every $l$ as long as $t \geq c_1n^{-1}, t \geq \mu_n$ and $t \geq c_4  (\mu_n^\alpha n)^{-1}$. Applying Theorem 19.3 in \citep{gyorfi2006distribution} for each term implies that 
	\begin{align*}
	& \Pr\left(\exists (q, \pi) \in \F \times \Pi,  I_2(q, \pi) > t\right) \\
	& \leq \sum_{l=0}^\infty \Pr\left( \sup_{f \in \F_l}  \frac{(P-\Pn) f(D)}{P f(D) + 2^l t} > \frac{1}{2} \right) \\
	& \leq \sum_{l=0}^\infty 60 \exp(-\frac{n2^l t (1/2)^{3}}{128\cdot 2304\max(K_1^2, K_2)}) \\& = \sum_{l=0}^\infty 60 \exp(-\frac{n2^l t}{c_5}) \leq \sum_{k=1}^\infty 60 \exp(-\frac{nk t}{c_5}) \leq  \frac{60\exp(-\frac{nt}{c_5}) }{1-\exp(-\frac{n t}{c_5}) }
	\end{align*}
	where $c_5 = 8 \cdot 128\cdot 2304\max(K_1^2, K_2)$. For any $\delta > 0$, when $t \geq \log(120/\delta) c_5n^{-1}$, we have both $\exp(-\frac{n t}{c_5})  \leq 1/2$ and $120 \exp(-nt/c_5) \leq \delta$ and as a result
	\begin{align*}
	\Pr\left(\exists (q, \pi) \in \F \times \Pi,  I_2(q, \pi) > t\right)  \leq \frac{60\exp(-\frac{nt}{c_5}) }{1-\exp(-\frac{n t}{c_5}) }\leq 120 \exp(-nt/c_5) \leq \delta
	\end{align*}
	Collecting all the conditions on $t$ and combing with the bound of $I_1(q, \pi)$, we have shown that w.p. at least $1-\delta$, the following holds for all $q, \pi$:
	\begin{align*}
	& \norm{\gn(q) - \g(H) }_{}^2 + \norm{\gn(q) - \g(H) }_{n}^2 + \mu_n J_2^2(\gn(q))  \\
	& \leq 5 \mu_n J_2^2 (\g(H))+ 2\mu_n J_1^2(H) + c_1n^{-1} + \mu_n + c_4  (\mu_n^\alpha n)^{-1} + \log(120/\delta) c_5 n^{-1} \\
	& = (1 + 5 J_2^2 (\g(H)) + 2J_1^2(H) ) \mu_n + (c_1 + \log(120)c_5 + \log(1/\delta) c_5)n^{-1} + c_4 (\mu_n^\alpha n)^{-1} \\
	& \leq K \left((1 + J_2^2 (\g(H)) + J_1^2(H) ) \mu_n + \frac{1+  \log(1/\delta)}{n} +\frac{p}{n\mu_n^\alpha}\right)
	\end{align*}
	where the leading constant can be chosen by $K=5 + 8(288 \max(K_1, \sqrt{2K_2}))^2 + 6 \cdot 8 \cdot 128\cdot 2304\max(K_1^2, K_2) + 2(18 (32)^4  \max(K_1^2, 4K_2^2) ) (6+ 12 F_{\max} + 24 G_{\max})^{2\alpha} \cdot (4C_1   +  (\text{diam}(\Theta)L_\Theta F_{\max})^{2\alpha}/(2\alpha))$.
	
\end{proof}

\begin{lemma}[Decomposition] \label{lemma: second term} \label{lemma: upper level initial bound}
	Suppose Assumptions (4-1), (4-2), (6-1), (6-2), (6-4), (4-4) and (3-2) hold. Then, the following hold with probability at least $1-{2\delta}$: for all policy $\pi \in \Pi$:
	\begin{align*}
	\norm{\gn(\qn) - \g(\qpi)}^2 + \lambda_n J_1^2(\qn) 
	\leq  \gamma_2(\delta, n, p, \mu_n, \lambda_n) + {K}\mu_n J_1(\qn) +  \Rem(\pi)
	\end{align*}
	where ${K}$ depends only on {$F_{\max}, G_{\max}, L_\Theta, \text{diam}(\Theta), \{C_i\}_{i=1}^3, \sup_{\pi \in \Pi} J_1(\tilde H^{\pi})$ and $\alpha$},  $\gamma_2(\delta, n, p, \mu_n, \lambda_n)$ is 		
	\begin{align*}
	\gamma_2 & (\delta, n, p, \mu_n, \lambda_n) 
	= {K} \Big[\mu_n + \lambda_n + \frac{p}{n \mu_n^{\alpha}} + \frac{1}{n^{1/(1+\alpha)}} + \frac{1}{n\lambda_n^\alpha} \\ 
	&   +  \frac{1+\log^{\frac{\alpha}{1+\alpha}}(1/\delta)}{n\mu_n^{\alpha/(1+\alpha)}} + \frac{1+\log(1/\delta)}{n} + \sqrt{\frac{p}{n \mu_n^{\alpha-1}} }+ \sqrt{\frac{\mu_n}{n}} (1+\sqrt{\log(1/\delta)}) \Big],
	\end{align*}
	and the remainder term, $\Rem(\pi)$ is given by $$\Rem(\pi) = 4\abs{\Pn (1/T) \sum_{t=1}^T \g(S_t; A_t; \qpi) [\Delta^\pi(S_t, A_t, S_{t+1}; \qn) - \Delta^\pi(S_t, A_t, S_{t+1}; \qpi)]}$$

\end{lemma}
\begin{proof}[Proof of Lemma \ref{lemma: upper level initial bound}]
	For $g_1, g_2  \in \G$,  define the functionals $\f_1, \f_2$, 
	\begin{align*}
	& \f_1(g_1) (D) = \frac{1}{T}\sum_{t=1}^T g_1(S_t, A_t)^2\\
	& \f_2(g_1, g_2) (D) = \frac{2}{T} \sum_{t=1}^T g_1(S_t, A_t)g_2(S_t, A_t) 
	\end{align*}
	With this notation, we have
	\begin{align*}
	& \norm{\gn(\qn) - \g(\qpi)}^2 + \lambda_n J_1^2(\qn) 
	= P\f_1(\gn(\qn)-\g(\qpi)) + \lambda_n J_1^2(\qn) \\
	&= 2 \times \big[\Pn  \f_1(\gn(\qn)) + \lambda_nJ_1^2(\qn)\big] + P\f_1(\gn(\qn)-\g(\qpi))  \\
	& \qquad \qquad + \lambda_n J_1^2(\qn) - 2 \times \big[\Pn  \f_1(\gn(\qn)) + \lambda_nJ_1^2(\qn)\big] 
	\end{align*}	
	Using the optimizing property of $\qn$ in (4.8), the term in the first parentheses can be bounded 
	\begin{align*}
	& \Pn  \f_1(\gn(\qn))+ \lambda_nJ_1^2(\qn) \leq \Pn  \f_1(\gn(\qpi)) + \lambda_nJ_1^2(\qpi)\\
	& = \Pn  \f_1 (\gn(\qpi) - \g(\qpi)) +  \Pn \f_1(\g(\qpi))+  \Pn \f_2(\gn(\qpi) - \g(\qpi), \g(\qpi))  + \lambda_nJ_1^2(\qpi) \\
	& = \Pn  \f_1(\gn(\qpi) - \g(\qpi)) +  \lambda_nJ_1^2(\qpi) +  (1/2) \Pn \f_2(2\gn(\qpi) - \g(\qpi), \g(\qpi))
	\end{align*}
	so that
	\begin{align*}
	&  \norm{\gn(\qn) - \g(\qpi)}^2 + \lambda_n J_1^2(\qn) \\
	&  \leq 2 \left[\Pn  \f_1(\gn(\qpi) - \g(\qpi)) +  \lambda_nJ_1^2(\qpi) + (1/2) \Pn \f_2(2\gn(\qpi) - \g(\qpi), \g(\qpi))  \right]  \\
	& \qquad + P\f_1(\gn(\qn)-\g(\qpi)) + \lambda_n J_1^2(\qn) - 2 (\Pn  \f_1(\gn(\qn))+ \lambda_nJ_1^2(\qn)) \\
	& = 2 \left[\Pn  \f_1(\gn(\qpi) - \g(\qpi)) +  \lambda_nJ_1^2(\qpi) \right]   + P\f_1(\gn(\qn)-\g(\qpi)) - \lambda_n J_1^2(\qn) \\
	& \qquad - 2 \Pn [ \f_1(\gn(\qn)) +  (1/2)  \f_2(\g(\qpi) - 2\gn(\qpi), \g(\qpi))] \\
	& = 2 \left[\Pn  \f_1(\gn(\qpi) - \g(\qpi)) +  \lambda_nJ_1^2(\qpi) \right]  + P\f_1(\gn(\qn)-\g(\qpi)) - \lambda_n J_1^2(\qn) \\
	&	\qquad - 2 \Pn [ \f_1(\gn(\qn)-\g(\qpi)) + \f_2 (\gn(\qn)- \gn(\qpi) , \g(\qpi))]\\
	& \leq 2 \left[\Pn  \f_1(\gn(\qpi) - \g(\qpi)) +  \lambda_nJ_1^2(\qpi) \right]  \\
	& \qquad + P\f_1(\gn(\qn)-\g(\qpi)) - \lambda_n J_1^2(\qn) - 2 \Pn  \f_1(\gn(\qn)-\g(\qpi))\\
	&	 \qquad + 2\abs{ \Pn \f_2 (\gn(\qn)- \gn(\qpi) , \g(\qpi))}
	\end{align*}
	where in the last equality we use the fact that 
	$g_1^2 + (g_2 - 2g_3) g_2 = g_1^2 + g_2^2 - 2g_2 g_3 = (g_1 - g_2)^2 + 2 g_1g_2 - 2 g_2 g_3 = (g_1 - g_2)^2 + 2 g_2(g_1 - g_3)$. In summary, we have 
	\begin{align*}
	\norm{\gn(\qn) - \g(\qpi)}^2 + \lambda_n J_1^2(\qn) \leq I_{1}(\pi) +  I_2(\pi) + I_3(\pi)
	\end{align*}
	where
	\begin{align*}
	& I_1 (\pi) = 2 \left[\Pn  \f_1(\gn(\qpi) - \g(\qpi)) +  \lambda_nJ_1^2(\qpi) \right] \\
	& I_2(\pi) = P\f_1(\gn(\qn)-\g(\qpi)) - \lambda_n J_1^2(\qn) - 2 \Pn  \f_1(\gn(\qn)-\g(\qpi))\\
	& I_3(\pi) = 2\abs{ \Pn \f_2 (\gn(\qn)- \gn(\qpi) , \g(\qpi))}
	\end{align*}
	Below we provide the upper bound for each of the three terms.  Recall that by Lemma \ref{lemma: uniform l2 bound}, the event, $E_n$ holds with probability at least $1-\delta$. Let the leading constant specified in Lemma \ref{lemma: uniform l2 bound} be ${K_0}$. 

	\textbf{Step I: bounding $I_1(\pi)$}
	Under the event $E_n$, we have 
	\begin{align*}
	I_1(\pi) & =  2\norm{\gn (\qpi) - \g (\qpi)}_{}^2 + \lambda_n J_1^2(\qpi) \\
	&	\leq 2{K_0}\Big(\mu_n (1+J_1^2(\qpi) + J_2^2(\g (\qpi))) + \frac{p}{n \mu_n^{\alpha}} + \frac{1+\log(1/\delta)}{n}\Big)  + \lambda_n J_1^2(\qpi) \\
	&	\leq (2{K_0} +1) \Big(\mu_n (1+J_1^2(\qpi) + J_2^2(\g (\qpi)))  + \lambda_n J_1^2(\qpi)+ \frac{p}{n \mu_n^{\alpha}} +   \frac{1+\log(1/\delta)}{n} \Big) \\
	& \leq (2{K_0} +1) \times 2(C_1^2 + C_2^2)\Big[(\mu_n + \lambda_n) (1+J_1^2(\qpi)) + \frac{p}{n \mu_n^{\alpha}} + \frac{1+\log(1/\delta)}{n}  \Big] \\
	& \leq (2{K_0} +1) \times 2(C_1^2 + C_2^2) (1+\sup_{\pi \in \Pi} J_1^2(\qpi)) \Big[\mu_n + \lambda_n  + \frac{p}{n \mu_n^{\alpha}} + \frac{1+\log(1/\delta)}{n}  \Big]\\
	& = {K_1}\Big[\mu_n + \lambda_n + \frac{p}{n \mu_n^{\alpha}} + \frac{1+\log(1/\delta)}{n}  \Big]
	\end{align*}
	where in the third inequality we use Assumption (6-4) and ${K_1} = 2(C_1^2 + C_2^2)(2{K_0} +1) (1+\sup_{\pi \in \Pi} J_1^2(\qpi))$.

	\textbf{Step II: bounding $I_3(\pi)$}
	Using the optimizing property of $\gn( q)$ and Assumption (6-2) that $\g(\qpi) = e^\pi \in \G$,  the followings holds for all $H \in \F, \pi \in \Pi$, 
	\begin{align*}
	& \mu_n J_2( \gn(q), \g(\qpi))  \\
	& = \Pn  [(1/T) \sum_{t=1}^T \big(1-H(S_t, A_t) + \sum_{a'} \pi(a'|S_{t+1}) H(S_{t+1}, a') - \gn(S_t, A_t; H)\big)\g(S_t, A_t; \qpi)] \\
	& = \Pn  [(1/T) \sum_{t=1}^T (\Delta^\pi(S_t, A_t, S_{t+1}; H) - \gn(S_t, A_t; H))\g(S_t, A_t; \qpi)]
	\end{align*}
	Thus we have 
	\begin{align*}
	& (1/2) \Pn \f_2 (\gn(\qn)- \gn(\qpi) , \g(\qpi)) \\
	& =   \Pn (1/T) \sum_{t=1}^T  \g(S_t, A_t; \qpi) [ \gn(S_t, A_t; \qn)     - \gn(S_t, A_t; \qpi)] \\
	& =  \Pn (1/T) \sum_{t=1}^T  \g(S_t, A_t; \qpi) [ \gn(S_t, A_t; \qn)    - \Delta^\pi(S_t, A_t, S_{t+1}; \qn) + \Delta^\pi(S_t, A_t, S_{t+1}; \qn) \\
	& \qquad\qquad  - \Delta^\pi(S_t, A_t, S_{t+1}; \qpi) + \Delta^\pi(S_t, A_t, S_{t+1}; \qpi) - \gn(S_t, A_t; \qpi)] \\
	& =   \Pn (1/T) \sum_{t=1}^T \g(S_t; A_t; \qpi) [\Delta^\pi(S_t, A_t, S_{t+1}; \qn) - \Delta^\pi(S_t, A_t, S_{t+1}; \qpi)] \\
	& \qquad\qquad  + \mu_n J_2( \gn(\qpi), \g(\qpi)) - \mu_n J_2(\gn(\qn), \g(\qpi))
	\end{align*}
	In addition, under the event $E_n$, we have
	\begin{align*}
	& \abs{\mu_n J_2( \gn(\qpi), \g(\qpi)) - \mu_n J_2(\gn(\qn), \g(\qpi))}  \leq \mu_n J_2(e^\pi) \left(J_2(\gn(\qpi) ) + J_2(\gn(\qn))\right) \\
	& \leq {K_0}   \mu_nJ_2(e^\pi)\Big(2 + J_1(\tilde H^\pi) + J_2(\g(\qpi)) +  J_1(\qn) + J_2(\g(\qn))  \\
	& \qquad + 2\sqrt{\frac{p}{n \mu_n^{\alpha+1}} }+ 2\sqrt{\frac{1}{n \mu_n}} +  2\sqrt{\frac{\log(1/\delta)}{n\mu_n}}\Big)\\
	& \leq {K_0}   \mu_nJ_2(e^\pi)\Big(2 + J_1(\tilde H^\pi) + C_2 + C_3 J_1(\tilde H^\pi) +  J_1(\qn) + C_2 + C_3 J_1(\qn)   \\
	& \qquad + 2\sqrt{\frac{p}{n \mu_n^{\alpha+1}} }+ 2\sqrt{\frac{1}{n \mu_n}} +  2\sqrt{\frac{\log(1/\delta)}{n\mu_n}}\Big)\\
	& \leq 2(1+C_1+C_2)  {K_0}   \mu_nJ_2(e^\pi)\Big(1 + J_1(\tilde H^\pi) + J_1(\qn)  + 
	\sqrt{\frac{p}{n \mu_n^{\alpha+1}} }+ \sqrt{\frac{1}{n \mu_n}} +  \sqrt{\frac{\log(1/\delta)}{n\mu_n}}\Big)\\
	& \leq 2(1+C_1+C_2)  {K_0}  (C_2 + C_3 \sup_{\pi \in \Pi} J_1(\tilde H^{\pi})) \mu_n  \Big(1 + \sup_{\pi \in \Pi} J_1(\tilde H^{\pi}) +  J_1(\qn) + \\
	& \qquad  \sqrt{\frac{p}{n \mu_n^{\alpha+1}} }+ \sqrt{\frac{1}{n \mu_n}} +  \sqrt{\frac{\log(1/\delta)}{n\mu_n}}\Big) \\
	& \leq {K_3} \mu_n J_1(\qn) + {K_3} \mu_n \Big(1 + \sqrt{\frac{p}{n \mu_n^{\alpha+1}} }+ \sqrt{\frac{1}{n \mu_n}} +  \sqrt{\frac{\log(1/\delta)}{n\mu_n}}\Big)
	\end{align*}
	where ${K_3} = 2(1+C_1+C_2)  {K_0}  (C_2 + C_3 \sup_{\pi \in \Pi} J_1(\tilde H^{\pi})) (1+\sup_{\pi \in \Pi} J_1(\tilde H^\pi))$.
	Thus we have
	\begin{align*}
	I_{3}(\pi) & =  2\abs{ \Pn \f_2 (\gn(\qn)- \gn(\qpi) , \g(\qpi))} \\
	& \leq 4\bigabs{\Pn (1/T) \sum_{t=1}^T \g(S_t; A_t; \qpi) [\Delta^\pi(S_t, A_t, S_{t+1}; \qn) - \Delta^\pi(S_t, A_t, S_{t+1}; \qpi)] }  \\
	& \qquad +  {K_3}  \mu_n J_1(\qn) +{K_3} \mu_n \Big(1 + \sqrt{\frac{p}{n \mu_n^{\alpha+1}} }+ \sqrt{\frac{1}{n \mu_n}} +  \sqrt{\frac{\log(1/\delta)}{n\mu_n}}\Big) \\
	& = \Rem(\pi) +  {K_3}  \mu_n J_1(\qn) +{K_3} \mu_n \Big(1 + \sqrt{\frac{p}{n \mu_n^{\alpha+1}} }+ \sqrt{\frac{1}{n \mu_n}} +  \sqrt{\frac{\log(1/\delta)}{n\mu_n}}\Big)
	\end{align*}
	
	\textbf{Step III: bounding $I_2(\pi)$}

	For the second term, 
	\begin{align*}
	I_2(\pi) & = P\f_1(\gn(\qn)-\g(\qpi)) - \lambda_n J_1^2(\qn) - 2 \Pn  \f_1(\gn(\qn)-\g(\qpi)) \\
	& = 2(P - \Pn) \f_1(\gn(\qn)-\g(\qpi)) - \lambda_n J_1^2(\qn) - P \f_1(\gn(\qn)-\g(\qpi))
	\end{align*}
	For simplicity, let $\beta(n, \mu_n, \delta, p) = p^{1/2} n^{-1/2} \mu_n^{-(1+\alpha)/2}+ (1+\sqrt{\log(1/\delta)} )(n\mu_n)^{-1/2}$. Under $E_n$, we have 
	\begin{align*}
	& J_2(\gn(\qn)-\g(\qpi))  \\
	& \leq J_2(\gn(\qn))  + J_2(\g(\qpi)) \\ 
	& \leq {K_0}(1 + J_1(\qn) +  J_2(\g (\qn)) + \beta(n, \mu_n, \delta, p)) + C_2 + C_3 J_1(\tilde H^\pi) \\
	& \leq {K_0}(1 + J_1(\qn) +  C_2 + C_3 J_1(\qn) + \beta(n, \mu_n, \delta, p)) + C_2 + C_3 \sup_{\pi \in \Pi} J_1(\tilde H^\pi) \\
	%		& \leq \left({K_0} (1+C_2 + C_2) + C_2 + C_3 \sup_{\pi \in \Pi} J_1(\tilde H^\pi)\right) (1+J_1(\qn) +  \beta(n, \mu_n, \delta, p)) \\
	& \leq c_1 (1+J_1(\qn) +  \beta(n, \mu_n, \delta, p))
	\end{align*} 
	where  $c_1= {K_0} (1+C_2 + C_2) + C_2 + C_3 \sup_{\pi \in \Pi} J_1(\tilde H^\pi)$ and $C_1, C_2$ are constants specified in Assumption (6-4). 
	
	Now we have $\Pr(\exists \pi \in \Pi,  I_2(\pi) > t) \leq \Pr(\{\exists \pi \in \Pi,  I_2(\pi) > t\} \jiao E_n) + \delta$ and we bound the first term using peeling device on $\lambda_n J_1^2(\qn)$ in $I_2(\pi)$: 
	\begin{align*}
	& \Pr(\{\exists \pi \in \Pi,  I_2(\pi) > t\} \jiao E_n)  \\
	& = \sum_{l=0}^\infty \Pr\big(\{\exists \pi \in \Pi,  I_2(\pi) > t, ~\lambda_n J_1^2(\qn) \in [2^{l}t \indicator{t\neq 0}, 2^{l+1} t)\} \jiao E_n \big)\\
	& \leq \sum_{l=0}^\infty \Pr\big(\exists \pi \in \Pi, ~ 2(P-\Pn) \f_1(\gn(\qn)-\g(\qpi)) > P \f_1(\gn(\qn)-\g(\qpi)) + \lambda_n J_1^2(\qn) + t, \\
	& \hspace{12ex} \lambda_n J_1^2(\qn) \in [2^{l}t \indicator{t\neq 0}, 2^{l+1} t), J_2(\gn(\qn)-\g(\qpi)) \leq c_1 (1+J_1(\qn) + \beta(n, \mu_n, \delta, p))\big)\\
	& \leq  \sum_{l=0}^\infty \Pr\big(\exists \pi \in \Pi, ~2(P-\Pn) \f_1(\gn(\qn)-\g(\qpi)) > P \f_1(\gn(\qn)-\g(\qpi)) + 2^{l}t \indicator{t\neq 0} + t, \\
	& \hspace{12ex} \lambda_n J_1^2(\qn) \leq 2^{l+1} t, J_2(\gn(\qn)-\g(\qpi)) \leq c_1 (1+ \sqrt{(2^{l+1} t)/\lambda_n} + \beta(n, \mu_n, \delta, p))\big)\\
	& \leq  \sum_{l=0}^\infty \Pr\big(\exists \pi \in \Pi, ~2(P-\Pn) \f_1(\gn(\qn)-\g(\qpi)) > P \f_1(\gn(\qn)-\g(\qpi)) + 2^{l}t, \\
	& \hspace{12ex} J_2(\gn(\qn)-\g(\qpi)) \leq c_1 (1+\sqrt{ (2^{l+1} t)/\lambda_n} + \beta(n, \mu_n, \delta, p))\big)\\
	& \leq \sum_{l=0}^\infty \Pr\left( \sup_{f \in \F_l}  \frac{(P-\Pn) f(D)}{P f(D) + 2^l t} > \frac{1}{2} \right)
	\end{align*}
	where $\F_l = \{ \f_1(g): J_2(g) \leq c_1 (1+ \sqrt{(2^{l+1} t)/\lambda_n} + \beta(n, \mu_n, \delta, p)), g \in \G  \}$. In what follows we verify the conditions (A1-A4) in Theorem 19.3 in \citep{gyorfi2006distribution} with $\F = \F_l$, $\epsilon = 1/2$ and $\eta = 2^{l} t$ to get an exponential inequality for each term in the summation. 
	
	For (A1),  it is easy to see that $\abs{\f_1(g)(D)} = \abs{\frac{1}{T} \sum_{t=1}^{T} g(S_t, A_t)^2} \leq G_{\max}^2$. We set $K_1 = G_{\max}^2$.  
	
	For (A2), we have $P \f_1^2(g) \leq G_{\max}^2 P \f_1(g)$. We set $K_2 = G_{\max}^2$.  
	
	For (A3),  the condition $\sqrt{n} \epsilon \sqrt{1-\epsilon} \sqrt{\eta} \geq 288 \max\{2K_1, \sqrt{2K_2}\}$ becomes $\sqrt{n} (1/2)^{3/2} \sqrt{2^lt} \geq 288 \max\{2G_{\max}^2, \sqrt{2} G_{\max}\}$.	So this holds for all $l \geq 0$ as long as $t \geq c_2/n$ for $c_2 = (8 \cdot 288 \max\{2G_{\max}^2, \sqrt{2} G_{\max}\})^2$. 
	
	Now we verify the condition (A4).  
	First note that for any $g_1, g_2 \in \G$
	\[
	\frac{1}{n} \sum_{i=1}^{n} \left[\f_1(g_1)(D_i) - \f_1(g_2)(D_i)\right]^2 \leq 4 G_{\max}^2 \norm{g_1 - g_2}_{n}^2
	\]
	The Assumption (4-4) then implies that the metric entropy for each $l$ is bounded  by
	\begin{align*}
	& \log {N}(u, \F_l, \norm{\cdot}_\infty) \\
	& \leq \log {N}\left(\frac{u}{2G_{\max}}, \{g: J_2(g) \leq c_1 (1+ \sqrt{(2^{l+1} t)/\lambda_n }+ \beta(n, \mu_n, \delta, p)), g \in \G  \}, \norm{\cdot}_\infty \right)\\
	& \leq C_3\left(\frac{c_1 (1+ \sqrt{(2^{l+1} t)/\lambda_n} + \beta(n, \mu_n, \delta, p))}{u/(2G_{\max})}\right)^{2\alpha} \\
	& \leq c_3 \left(  1+ \left(\frac{2^{{l+1}} t}{\lambda_n}\right)^\alpha + \beta^{2\alpha}(n, \mu_n, \delta, p)  \right) u^{-2\alpha}
	\end{align*} 
	where $C_1$   in the last inequality is specified in Assumption (4-4) and the constant $c_3 = (2G_{\max} c_1)^{2\alpha} C_3$, . Now we just need to ensure for all $x \geq \eta/8 = 2^l t/8$ and $l\geq 0$:
	\begin{align*}
	\frac{\sqrt{n} (1/2)^2 x}{96 \sqrt{2} \cdot 2G_{\max}^2} \geq \int_{0}^{\sqrt{x}} \sqrt{c_3} \left(  1+ \left(\frac{2^{l+1} t}{\lambda_n}\right)^\alpha + \beta^{2\alpha}(n, \mu_n, \delta, p)   \right)^{1/2} u^{-\alpha} du
	\end{align*}
	Note that $\int_{0}^{\sqrt{x}}  u^{-\alpha} du = (1-\alpha)^{-1} x^{\frac{1-\alpha}{2}}$. 
	%(\ZL{Miss a constant}). 
	The above equality is equivalent with the following:
	\begin{align*}
	%		c_4\sqrt{n} x \geq  \left(  1+ \frac{2^l t}{\lambda_n} + \sqrt{\log (4/\delta)}  \right)^{1/2} x^{\frac{1-\alpha}{2}}\\
	%		&\iff 
	\frac{(1/2)^2(1-\alpha)}{96 \sqrt{2} \cdot 2G_{\max}^2} \sqrt{n} x^{\frac{1+\alpha}{2}} \geq  \sqrt{c_3}  \left(  1+ \left(\frac{2^{l+1} t}{\lambda_n}\right)^\alpha + \beta^{2\alpha}(n, \mu_n, \delta, p)  \right)^{1/2}
	\end{align*}
	Note that the LHS is a increasing function of $x$. It's then enough to ensure the followings hold for all $l \geq 0$:
	\begin{align*}
	&\frac{1-\alpha}{4\cdot 96 \sqrt{2} \cdot 2G_{\max}^2}\sqrt{n} (2^l t/8)^{\frac{1+\alpha}{2}} \geq  \sqrt{c_3}  \\
	& \frac{1-\alpha}{ 4\cdot 96 \sqrt{2} \cdot 2G_{\max}^2}\sqrt{n} (2^l t/8)^{\frac{1+\alpha}{2}} \geq  \sqrt{c_3}  \left(\frac{2^{l+1} t}{\lambda_n}\right)^{\alpha/2}\\
	& \frac{1-\alpha}{4\cdot 96 \sqrt{2} \cdot 2G_{\max}^2}\sqrt{n} (2^l t/8)^{\frac{1+\alpha}{2}} \geq \sqrt{c_3}  \beta^{\alpha}(n, \mu_n, \delta, p).
	\end{align*}
	The above is satisfied for all $l$ by choosing large enough $t$. For example,  the first one holds whenever $$t \geq  8(c_3  4\cdot 96 \sqrt{2} \cdot 2G_{\max}^2  (1-\alpha)^{-1})^{\frac{2}{1+\alpha}}  n^{-1/(1+\alpha)} := c_4 n^{-1/(1+\alpha)}$$
	Similarly, the second and third inequalities hold for all $l$ if $$t \geq  8^{1+\alpha} 2^{\alpha} {c_3}(4\cdot 96 \sqrt{2} \cdot 2G_{\max}^2 (1-\alpha)^{-1})^2  (n\lambda_n^{\alpha})^{-1} := c_5(n\lambda_n^{\alpha})^{-1}$$The third one holds when $t$ satisfies
	\begin{align*}
	t & \geq 8\big((4\cdot 96 \sqrt{2} \cdot 2G_{\max}^2)^2{c_3} (1-\alpha)^{-1}  \big)^{\frac{1}{1+\alpha}}  \cdot n^{-1/(1+\alpha)} \beta^{\frac{2\alpha}{1+\alpha}}(n, \mu_n, \delta, p) \\
	& = c_4n^{-1/(1+\alpha)} \beta^{\frac{2\alpha}{1+\alpha}}(n, \mu_n, \delta, p) 
	\end{align*}
	Note that 
	\begin{align*}
	n^{-\frac{1}{1+\alpha}}  \beta^{\frac{2\alpha}{1+\alpha}}(n, \mu_n, \delta, p)  & = n^{-\frac{1}{1+\alpha}}   \left[p^{1/2} n^{-1/2} \mu_n^{-(1+\alpha)/2}+ (1+\sqrt{\log(1/\delta)} )(n\mu_n)^{-1/2}\right]^{\frac{2\alpha}{1+\alpha}} \\
	& \leq \frac{p^{\frac{\alpha}{1+\alpha}}}{n \mu_n^{\alpha}} + \frac{1+\log^{\frac{\alpha}{1+\alpha}}(1/\delta)}{n\mu_n^{\alpha/(1+\alpha)}}
	\end{align*}
	Thus the third one can be reduced to require $t$ such that 
	\[
	t \geq c_4 \left[\frac{p^{\frac{\alpha}{1+\alpha}}}{n \mu_n^{\alpha}} + \frac{1+\log^{\frac{\alpha}{1+\alpha}}(1/\delta)}{n\mu_n^{\alpha/(1+\alpha)}}\right]
	\]
	In summary, all conditions (A1) to (A4) would be satisfied for all $l \geq 0$ when
	\begin{align*}
	t \geq c_2n^{-1} + c_4 n^{-1/(1+\alpha)} + c_5(n\lambda_n^{\alpha})^{-1}  + c_4 \left[\frac{p^{\frac{\alpha}{1+\alpha}}}{n \mu_n^{\alpha}} + \frac{1+\log^{\frac{\alpha}{1+\alpha}}(1/\delta)}{n\mu_n^{\alpha/(1+\alpha)}}\right]
	\end{align*}
	%\begin{align*}
	%t \geq c_4 (1+ (\log(4/\delta))^{\alpha}) n^{-1/(1+\alpha)} + c_4 (n\lambda_n^{\alpha})^{-1}
	%\end{align*}
	We can now apply Theorem 19.3 in \citep{gyorfi2006distribution} for each $l$-th term. Similar to the proof of Lemma \ref{lemma: uniform l2 bound}, we have
	\begin{align*}
	\Pr(\{\exists \pi \in \Pi, I_2(\pi) > t\} \jiao E_n) & \leq \sum_{l=0}^\infty \Pr\left( \sup_{f \in \F_l}  \frac{(P-\Pn) f(D)}{P f(D) + 2^l t} > \frac{1}{2} \right)  \leq  \frac{60\exp(-\frac{nt}{c_6}) }{1-\exp(-\frac{n t}{c_6}) }
	\end{align*}
	where $c_6 =  8 \cdot 128\cdot 2304\max(G_{\max}^4, G_{\max}^2)$. When $t \geq \log(120/\delta) c_6n^{-1}$, we have both $\exp(-\frac{n t}{c_6})  \leq 1/2$ and $120 \exp(-nt/c_6) \leq \delta$ and thus
	\begin{align*}
	\Pr\left(\exists \pi \in \Pi, I_2(\pi) > t\right)  \leq \delta + \frac{60\exp(-\frac{nt}{c_5}) }{1-\exp(-\frac{n t}{c_5}) }\leq \delta + 120 \exp(-nt/c_5) \leq 2\delta
	\end{align*}
	Collecting all condition on $t$,  w.p. $1-2\delta$, for all policy $\pi \in \Pi$ we have shown that for a constant ${K_2}$ that depends only on $F_{\max}, G_{\max}, L_\Theta, \text{diam}(\Theta), \{C_i\}_{i=1}^3$, $\sup_{\pi \in \Pi}J_1(\tilde H^{\pi})$ and $\alpha$,
	\begin{align*}
	I_2(\pi) 
	%	 & \leq c_2n^{-1} + c_5 n^{-\frac{1}{1+\alpha}}+ c_5(n\lambda_n^{\alpha})^{-1}  + c_5 \left[\frac{p^{\frac{\alpha}{1+\alpha}}}{n \mu_n^{\alpha}} + \frac{1+\log^{\frac{\alpha}{1+\alpha}}(1/\delta)}{n\mu_n^{\alpha/(1+\alpha)}}\right] + \log(120/\delta) c_6n^{-1}\\
	&  \leq {K_2}\Big[ \frac{1}{n^{1/(1+\alpha)}} + \frac{1}{n\lambda_n^\alpha} + \frac{p^{\frac{\alpha}{1+\alpha}}}{n \mu_n^{\alpha}} + \frac{1+\log^{\frac{\alpha}{1+\alpha}}(1/\delta)}{n\mu_n^{\alpha/(1+\alpha)}} + \frac{1+\log(1/\delta)}{n} \Big]
	\end{align*}

	\textbf{Summary} \quad
	Collecting the three bounds on $I_1(\pi), I_2(\pi), I_3(\pi)$, for ${K} = {K_1} + {K_2} + {K_3}$, we have
	\begin{align*}
	\norm{\gn(\qn) - \g(\qpi)}^2 + \lambda_n J_1^2(\qn) 
	\leq  \gamma_2(\delta, n, p, \mu_n, \lambda_n) + {K}\mu_n J_1(\qn) +  \Rem(\pi)
	\end{align*}
	where $\gamma_2(\delta, n, p, \mu_n, \lambda_n)$ is a constant independent of the policy
	\begin{align*}
	\gamma_2 & (\delta, n, p, \mu_n, \lambda_n) 
	%	 =  {K_1}\Big[\mu_n + \lambda_n + \frac{p}{n \mu_n^{\alpha}} + \frac{1+\log(1/\delta)}{n}  \Big]  \\ 
	%	& 
	%	+ {K_2}\Big[ \frac{1}{n^{1/(1+\alpha)}} + \frac{1}{n\lambda_n^\alpha} + \frac{p^{\frac{\alpha}{1+\alpha}}}{n \mu_n^{\alpha}} + \frac{1+\log^{\frac{\alpha}{1+\alpha}}(1/\delta)}{n\mu_n^{\alpha/(1+\alpha)}} + \frac{1+\log(1/\delta)}{n} \Big] \\
	%	& + {K_3} \mu_n \Big(1 + \sqrt{\frac{p}{n \mu_n^{\alpha+1}} }+ \sqrt{\frac{1}{n \mu_n}} +  \sqrt{\frac{\log(1/\delta)}{n\mu_n}}\Big) \\
	= {K} \Big[\mu_n + \lambda_n + \frac{p}{n \mu_n^{\alpha}} + \frac{1}{n^{1/(1+\alpha)}} + \frac{1}{n\lambda_n^\alpha} \\ 
	&   +  \frac{1+\log^{\frac{\alpha}{1+\alpha}}(1/\delta)}{n\mu_n^{\alpha/(1+\alpha)}} + \frac{1+\log(1/\delta)}{n} + \sqrt{\frac{p}{n \mu_n^{\alpha-1}} }+ \sqrt{\frac{\mu_n}{n}} (1+\sqrt{\log(1/\delta)}) \Big]
	\end{align*}
\end{proof}

\begin{lemma} \label{lemma: donsker class prob bound}
	For $M, \sigma > 0$, let	
	\begin{align*}
	\F^*= & \Bigdkh{ f: D \mapsto (1/T) \sum_{t=1}^T \big[H(S_t, A_t) - \sum_{a'} \pi(a'|S_{t+1})    H (S_{t+1}, a')\big] g(S_t, A_t): \\
		& \qquad \qquad \pi \in \Pi, g \in \G, H \in \F, J_1(H) \leq M, P f^2 \leq \sigma^2, J_2(g) \leq \sup_{\pi \in \Pi} J_2(e^\pi)}
	\end{align*}
	Under Assumption 4, the following holds  with probability at least $1-\delta$, 
	\begin{align*}
	\sup_{f \in \F^*}\abs{\Gn f} & \leq C(\delta) \gamma(n, p, M, \sigma) 
	\end{align*}
	where
	\begin{align*}
	C(\delta) = K_1  + \frac{4+ 32G_{\max} F_{\max} K_1}{ (8/3) G_{\max}F_{\max}} + (8/3) G_{\max}F_{\max} \log(1/\delta)
	\end{align*}
	and
	\begin{align*}
	\gamma(n, p, M, \sigma)  & = \sqrt{p}\sigma^{1-\alpha} + p\sigma^{-2\alpha} n^{-1/2} + M^{\alpha}  \sqrt{p}\sigma^{1-\alpha}   + M^{2\alpha} p\sigma^{-2\alpha} n^{-1/2} \\
	&\qquad \qquad + \sigma + n^{-1/2}  + p^{1/4} n^{-1/4} \sigma^{(1-\alpha)/2} + \sqrt{p} \sigma^{-\alpha}n^{-1/2} \\
	&\qquad \qquad + M^{\alpha/2} p^{1/4} n^{-1/4} \sigma^{(1-\alpha)/2} + M^{\alpha}\sqrt{p} \sigma^{-\alpha}n^{-1/2}.  
	\end{align*}
\end{lemma}
\begin{proof}[Proof of Lemma \ref{lemma: donsker class prob bound}]

	We need to show that w.p. $1-\delta$ 
	\begin{align*}
	Z_n \triangleq \sup_{f \in \F^\ast}\abs{\Gn f} \leq C(\delta) \gamma(n, p, M, \sigma).
	\end{align*}
	Let $B = 2G_{\max}F_{\max}$ and $\bar \delta = \frac{\sigma}{\norm{F}}$. Then by Lemma 2.2 in \citep{chernozhukov2014gaussian}, 
	$$
	\EE[Z_n] \leqconst J(\bar \delta, \F^\ast, F)  \norm{F}_2 + \frac{BJ^2(\bar \delta, \F^\ast, F)}{\bar \delta^2 \sqrt{n}},
	$$
	where
	\begin{align*}
	J(\bar \delta, \F^\ast, F) &= \int_{0}^{\bar \delta} \sqrt{\log N(\epsilon \norm{F}_2, \F^\ast, \norm{\cdot}_\infty)} d\epsilon \\
	& = \frac{1}{\norm{F}}\int_{0}^\sigma \sqrt{\log(\epsilon, \F^\ast, \norm{\cdot}_\infty)} d\epsilon
	\end{align*}
	Let $J_{\max} := \sup_{\pi \in \Pi} J_2(e^\pi)$, then we can show that
	\begin{align*}
	& N(2(G_{\max} + F_{\max})\epsilon, \F^\ast, \norm{\cdot}_\infty) \\
	& \leq N(\epsilon, \F_{M}, \norm{\cdot}_\infty) N(\epsilon, \G_{J_{\max}}, \norm{\cdot}_\infty) N(\epsilon, \V(\Pi, \F_M), \norm{\cdot}_\infty).
	\end{align*}
	Therefore
	\begin{align*}
	& \log N(\epsilon, \F^\ast, \norm{\cdot}_\infty) \\
	& \leq \log N(\frac{\epsilon}{2(G_{\max} + F_{\max})}, \F_{M}, \norm{\cdot}_\infty) \\
	& \qquad + \log N(\frac{\epsilon}{2(G_{\max} + F_{\max})}, \G_{J_{\max}}, \norm{\cdot}_\infty) \\
	& \qquad + \log N(\frac{\epsilon}{2(G_{\max} + F_{\max})}, \V(\Pi, \F_M), \norm{\cdot}_\infty) \\
	& \leq C_1  (2(G_{\max} + F_{\max}))^{2\alpha}\left(\left(\frac{M}{\epsilon}\right)^{2\alpha} + \left(\frac{J_{\max}}{\epsilon}\right)^{2\alpha}\right) + \tilde C_1  (2(G_{\max} + F_{\max}))^{2\alpha} \left(\frac{M}{\epsilon}\right)^{2\alpha}\\
	& \leq (2M^{2\alpha}+J_{\max}^{2\alpha})(C_1  + \tilde C_1) (2(G_{\max} + F_{\max}))^{2\alpha}\epsilon^{-2\alpha} \\
	& \leq {K (1+M^{2\alpha}) p\epsilon^{-2\alpha}},
	\end{align*}
	where
	\[
	\tilde C_1  = C_1  + p (\text{diam}(\Theta)L_\Theta F_{\max})^{2\alpha}/(2\alpha).
	\]
	Now we have
	\begin{align*}
	J(\delta, \F^\ast, F) & \leq \frac{\sqrt{(2M^{2\alpha}+J_{\max}^{2\alpha})(C_1  + \tilde C_1) (2(G_{\max} + F_{\max}))^{2\alpha}}}{2G_{\max}  F_{\max}(1-\alpha)}  \sigma^{1-\alpha} \\
	& \leq {{\frac{\sqrt{K}}{2G_{\max}  F_{\max}(1-\alpha)}} (1+M^{\alpha}) \sqrt{p}\sigma^{1-\alpha}} 
	\end{align*},
	where $K = (2M^{2\alpha}+J_{\max}^{2\alpha})(C_1  + \tilde C_1) (2(G_{\max} + F_{\max}))^{2\alpha}$.
	Therefore
	\begin{align*}
	\EE[Z_n] 
	& \leq {\frac{C\sqrt{K}}{1-\alpha}} (1+M^{\alpha}) \sqrt{p}\sigma^{1-\alpha} + 2CG_{\max}  F_{\max}  \frac{K(1+M^{2\alpha})}{(1-\alpha)^2}  p\sigma^{-2\alpha} n^{-1/2} \\
	%	& = K_1 (1+M^{\alpha}) \sqrt{p}\sigma^{1-\alpha} + K_2 (1+M^{2\alpha}) p\sigma^{-2\alpha} n^{-1/2}  \\
	& \leq  K_1\left( \sqrt{p}\sigma^{1-\alpha} + p\sigma^{-2\alpha} n^{-1/2} + M^{\alpha}  \sqrt{p}\sigma^{1-\alpha}   + M^{2\alpha} p\sigma^{-2\alpha} n^{-1/2}\right),
	\end{align*}
	where $K_1 = C\frac{\sqrt{K}}{1-\alpha} + 2CG_{\max}  F_{\max} K$.
	By Talagrand's inequality, with probability $1-e^{-t}$, we have
	\begin{align*}
	Z_n & \leq  \EE[Z_n] + \sqrt{2t(\sigma^2 + 4n^{-1/2}  \EE[Z_n] b) }+ \frac{2tb}{3\sqrt{n}} \\
	& \leq \EE[Z_n] + \sqrt{2t } \sigma + \sqrt{8t n^{-1/2}  \EE[Z_n] b }+ \frac{2tb}{3\sqrt{n}} \\
	& \leq K_1\left( \sqrt{p}\sigma^{1-\alpha} + p\sigma^{-2\alpha} n^{-1/2} + M^{\alpha}  \sqrt{p}\sigma^{1-\alpha}   + M^{2\alpha} p\sigma^{-2\alpha} n^{-1/2}\right)    + \sqrt{2t} \sigma +  \frac{4t G_{\max}F_{\max}}{3\sqrt{n}} \\
	& ~ + 4 \sqrt{G_{\max} F_{\max} K_1  t} \sqrt{n^{-1/2}\left( \sqrt{p}\sigma^{1-\alpha} + p\sigma^{-2\alpha} n^{-1/2} + M^{\alpha}  \sqrt{p}\sigma^{1-\alpha}   + M^{2\alpha} p\sigma^{-2\alpha} n^{-1/2}\right) }\\
	& \leq (K_1 + \sqrt{2t} + 4\sqrt{G_{\max} F_{\max} K_1 t} + (4/3) G_{\max}F_{\max} t) \\ & \times (\sqrt{p}\sigma^{1-\alpha} + p\sigma^{-2\alpha} n^{-1/2} + M^{\alpha}  \sqrt{p}\sigma^{1-\alpha}   + M^{2\alpha} p\sigma^{-2\alpha} n^{-1/2} + \sigma + n^{-1/2} \\
	&\qquad + p^{1/4} n^{-1/4} \sigma^{(1-\alpha)/2} + \sqrt{p} \sigma^{-\alpha}n^{-1/2} + M^{\alpha/2} p^{1/4} n^{-1/4} \sigma^{(1-\alpha)/2} + M^{\alpha}\sqrt{p} \sigma^{-\alpha}n^{-1/2} ). 
	\end{align*}
	Let $e^{-t} = \delta$, then we have
	\begin{align*}
	Z_n & \leq C(\delta) \gamma(n, p, M, \sigma) 
	\end{align*}
	where we can show
	\begin{align*}
	& K_1 + \sqrt{2\log(1/\delta)} + 4\sqrt{G_{\max} F_{\max} K_1 \log(1/\delta)}  + (4/3) G_{\max}F_{\max} \log(1/\delta)\\
	& \leq K_1 + \sqrt{(4+ 32G_{\max} F_{\max} K_1) \log(1/\delta)}  + (4/3) G_{\max}F_{\max} \log(1/\delta)\\
	& \leq K_1  + \frac{4+ 32G_{\max} F_{\max} K_1}{ (8/3) G_{\max}F_{\max}} + (8/3) G_{\max}F_{\max} \log(1/\delta) := C(\delta)
	\end{align*}
\end{proof}

\begin{lemma} \label{lemma: key ineq}
	Suppose Assumptions 2, (3-3) and (6-3) hold. For any $H \in \F$,  we have
	\begin{align*}
	& \EE\Big[\Big(\frac{1}{T}\sum_{t=1}^T [\Delta^\pi(S_t, A_t, S_{t+1}; H) - \Delta^\pi(S_t, A_t, S_{t+1}; \qpi)]\Big)^2\Big]  \\
	& \leq \Bigg[\frac{2 + 2 C_0 \beta/(1-\beta) }{\kappa'} \Bigg]^2 \Bigg[1+\frac{1 + (1/T) \norm{\frac{d_{T+1}}{d_D}}_\infty}{p_{\min}} \Bigg] \norm{\g(H) - \g(\qpi)}^2
	\end{align*}
	
\end{lemma}

\begin{proof}[Proof of Lemma \ref{lemma: key ineq}]
	Let $U^\pi(s, a, s'; H) = \sum_{a'} \pi(a'|s')H(s', a') - H(s, a)$. Let $\bar H(s, a) = H(s, a) - \sum_{s, a} d^\pi(s) \pi(a|s) H(s,a)$ be the shifted version of $H$ such that the expectation under the stationary distribution is zero. And similarly we define $\bar H^\pi = \qpi(s, a) - \sum_{s, a} d^\pi(s) \pi(a|s) \qpi (s,a)$.  By definition, we have
	\begin{align*}
	& \EE\Big[\Big(\frac{1}{T}\sum_{t=1}^T [\Delta^\pi(S_t, A_t, S_{t+1}; H) - \Delta^\pi(S_t, A_t, S_{t+1}; \qpi)]\Big)^2\Big]  \\
	& =  \EE\Big[\Big(\frac{1}{T}\sum_{t=1}^T U^\pi(S_t, A_t, S_{t+1}; \bar H) - U(S_t, A_t, S_{t+1}; \bar H^\pi) \Big)^2\Big]  \\
	& \leq \Bigdkh{1 + \frac{1}{p_{\min}}\big[1 + (1/T) \bignorm{\frac{d_{T+1}}{d_D}}_\infty\big]} \norm{\bar H - \bar H^\pi}^2
	\end{align*}
	where we use Lemma B.4 in \cite{liao2019off} in the last inequality.  Next we can apply Lemma B.5  in \cite{liao2019off} to get
	\begin{align*}
	\norm{\bar H - \bar H^\pi} & \leq 2 (1+C_0 \beta/(1-\beta)) \norm{(\I - \P^\pi)(\bar H - \bar H^\pi)} \\
	& =  2 (1+C_0 \beta/(1-\beta)) \norm{(\I - \P^\pi)(H - \qpi)}
	\end{align*}
	where in the last equality we use the fact that the operator $(\I - \P^\pi)$ is invariant to the constant shift. Now we can use Assumption (6-3) to bound the last term and get
	\begin{align*}
	\norm{ H -  \qpi} \leq  \frac{2 (1+C_0 \beta/(1-\beta))}{\kappa'}\norm{\g(H) - \g(\qpi)}
	\end{align*}
	Combining the two inequalities gives the desired result.  
\end{proof}

\section{Regret Bound}
\label{appendix: regret}

\begin{proof}[Proof of Theorem 5.1]
	Since $\Theta$ is compact and $\V^{\pi_{\theta}}$ is continuous according to Lemma \ref{lemma: lip}, there exists $\theta^* \in \Theta$, such that $\sup_{\pi \in \Pi} \V^{\pi} =\sup_{\theta \in \Theta} \V^{\pi_\theta}= \V^{\pi_{\theta^*}}$. Let $\pi^* = \pi_{\theta^*}$. 
	We bound the regret by
	\begin{align*}
	\Regret (\hat \pi_n) & = \sup_{\pi \in \Pi} \V^{\pi} -  \V^{\hat \pi_n} = \V^{\pi^*} - \V^{\hat \pi_n} \\
	& = \Vn^{\hat \pi_n} - \V^{\hat \pi_n} - (\Vn^{\pi^*} - \V^{\pi^*}) + \Vn (\pi^*) - \Vn(\hat \pi_n) \\
	& \leq \Vn^{\hat \pi_n} - \V^{\hat \pi_n} - (\Vn^{\pi^*} - \V^{\pi^*}) + \sup_{\pi \in \Pi} \Vn^{\pi} - \Vn(\hat \pi_n) \\
	& \leq \Vn^{\hat \pi_n} - \V^{\hat \pi_n} - (\Vn^{\pi^*} - \V^{\pi^*}) 
	\end{align*}
	Recall the efficient influence function is given by $\phi^\pi(D) = 	\frac{1}{T} \sum_{t=1}^T \omega^\pi(S_t, A_t) [R_{t+1} +  \sum_{a'} \pi(a'|S_{t+1}) Q^\pi(S_{t+1}, a') - Q^\pi(S_t, A_t) -  \V^{\pi} ]$. 
	Define the remainder term $\Rem_n(\pi) =(\Vn^{\pi} - \V^{\pi}) - \Pn \phi^{\pi}$.  We then have
	\begin{align*}
	\Regret (\hat \pi_n)  & \leq \Pn (\phi^{\hat \pi_n} - \phi^{\pi^*}) + (\Rem_n(\hat \pi_n) - \Rem_n(\pi^*))  \\
	& \leq \sup_{\pi \in \Pi} ~ \Pn (\phi^{\pi} - \phi^{\pi^*}) + 2 \sup_{\pi \in \Pi} \abs{\Rem_n(\pi)} 
	\end{align*}
	
	%		Below we show that (i) $\sup_{\pi \in \Pi} \Pn (\phi^{\pi} - \phi^{\pi^*})  = O_P(n^{-1/2})$ and (ii) $\sup_{\pi \in \Pi} \abs{\Rem_n(\pi)}  = o_P(n^{-1/2})$ and thus $\Regret(\hat \pi_n) = O_P(n^{-1/2})$.

	\textbf{(i) Leading Term} \quad
	For any $(s, a, s', r)$, we have
	\begin{align*}
	& \abs{\omega^{\pi_{\theta_1}}(s, a) (R + U^{\pi_{\theta_1}}(s, a, s') - \V^{\pi_{\theta_1}}) - \omega^{\pi_{\theta_2}}(s, a) (r + U^{\pi_{\theta_2}}(s, a, s') - \V^{\pi_{\theta_2}})}\\
	& \leq 	\abs{\omega^{\pi_{\theta_1}}(s, a) (R + U^{\pi_{\theta_1}}(s, a, s') - \V^{\pi_{\theta_1}}) - \omega^{\pi_{\theta_2}}(s, a) (r + U^{\pi_{\theta_1}}(s, a, s') - \V^{\pi_{\theta_1}})} \\
	& \quad +  \abs{ \omega^{\pi_{\theta_2}}(s, a) (r + U^{\pi_{\theta_1}}(s, a, s') - \V^{\pi_{\theta_1}}) - \omega^{\pi_{\theta_2}}(s, a) (r + U^{\pi_{\theta_2}}(s, a, s') - \V^{\pi_{\theta_2}})}\\
	& \leq 2(R_{\max} + F_{\max}) \abs{\omega^{\pi_{\theta_1}}(s, a) - \omega^{\pi_{\theta_2}}(s, a)}  \\
	& \quad + G_{\max}  \left(\sup_{\pi \in \Pi} \norm{\omega^\pi}^2\right) \left(\abs{U^{\pi_{\theta_1}}(s, a, s') - U^{\pi_{\theta_2}}(s, a, s')} + \abs{\V^{\pi_{\theta_1}} - \V^{\pi_{\theta_2}}}\right)
	\end{align*}
	Using Lemma \ref{lemma: lip} and Assumption {$\inf_{s} d_D(s) := d_{\min} > 0$}, (4-1), (6-1), (6-2) and (5-1), it can be seen that 
	\begin{align*}
	&\abs{\omega^{\pi_{\theta_1}}(s, a) - \omega^{\pi_{\theta_2}}(s, a)} 
	%		\leq \frac{1}{d_D(s, a)} (C_d \norm{\theta_1 - \theta_2}_2 + L_{\Theta} \norm{\theta_1 - \theta_2}) 
	\leq (p_{\min} d_{\min})^{-1}(C_d  + L_{\Theta})\norm{\theta_1 - \theta_2}_2\\
	&	\abs{\V^{\pi_{\theta_1}} - \V^{\pi_{\theta_2}}}  \leq 
	%		R_{\max} \abs{\A} L_{\Theta} \norm{\theta_1 - \theta_2}_{2} + R_{\max}  C_d\norm{\theta_1 - \theta_2}_2 =\\
	R_{\max}(L_{\Theta} \abs{\A} + C_d) \norm{\theta_1 - \theta_2}_2 
	\end{align*}
	On the other hand, for any constant $c$, we have 
	\begin{align*}
	&\abs{U^{\pi_{\theta_1}}(s, a, s') - U^{\pi_{\theta_2}}(s, a, s')} 
	\\&\leq \abs{Q^{\pi_{\theta_1}}(s, a)  - (Q^{\pi_{\theta_2}}(s, a) - c)} + \abs{\sum_{a'} \pi_{\theta_1}(a'|s') Q^{\pi_{\theta_1}}(s', a')  - \pi_{\theta_2}(a'|s') (Q^{\pi_{\theta_2}}-c)(s', a')} \\
	&	\leq \abs{\V^{\pi_{\theta_1}} - \V^{\pi_{\theta_2}}} + \norm{V^{\pi_{\theta_1}} + c - V^{\pi_{\theta_2}}}_{\infty} \left(1 + L_{\Theta} \abs{\A} + 1\right) \\ & 
	\leq (R_{\max}(L_{\Theta} \abs{\A} + C_d) \norm{\theta_1 - \theta_2}_2  + \left(2 + L_{\Theta} \abs{\A}\right) \norm{V^{\pi_{\theta_1}} - (V^{\pi_{\theta_2}}-c)}_{\infty}
	\end{align*}
	where we define the state-only relative value function by $V^\pi (s) = \sum_{a} \pi(a|s) Q^\pi(s, a)$. By choosing $c = \mu^{\pi_{\theta_1}}(V^{\pi_{\theta_2}})$, we can apply Lemma \ref{lemma: lip} to bound $\norm{V^{\pi_{\theta_1}} - (V^{\pi_{\theta_2}}-\mu^{\pi_{\theta_1}}(V^{\pi_{\theta_2}})}_{\infty}$ and get
	\begin{align*}
	\abs{U^{\pi_{\theta_1}}(s, a, s') - U^{\pi_{\theta_2}}(s, a, s')}  \leq   \Bigdkh{R_{\max}(L_{\Theta} \abs{\A} + C_d )+ \left(2 + L_{\Theta} \abs{\A}\right) C_V} \norm{\theta_1 - \theta_2}_2
	\end{align*}
	Here $C_d, C_V$ are the constants in Lemma \ref{lemma: lip}.
	Let ${K_1} = 2(R_{\max} + F_{\max}) (p_{\min} d_{\min})^{-1}(C_d  + L_{\Theta}) + G_{\max}  \left(\sup_{\pi \in \Pi} \norm{\omega^\pi}^2\right) [ R_{\max}(L_{\Theta} \abs{\A} + C_d) +  R_{\max}(L_{\Theta} \abs{\A} + C_d) + C_V\left(2 + L_{\Theta} \abs{\A}\right)]$, we have 
	\begin{align}
	\abs{\phi^{\pi_{\theta_1}}(D) - \phi^{\pi_{\theta_2}}(D)} \leq {K_1} \norm{\theta_1 - \theta_2}_2 \label{lip eif}
	\end{align}
	On the other hand, 
	\begin{align*}
	\abs{\phi^{\pi}(D)} \leq \phi_{\max} := 2 (R_{\max} + F_{\max}) G_{\max} \cdot \sup_{\pi \in \Pi} \norm{\omega^\pi}^2
	\end{align*}
	The maximal inequality with bracketing number then gives that 
	\begin{align*}
	\EE[\sup_{\pi \in \Pi} \abs{\Gn (\phi^{\pi} - \phi^{\pi^*})} ] = \EE[\sup_{f \in \F^*} \abs{\Gn f} ]  
	\leqconst J_{[]}(\phi_{\max}, \F^*, L_2)
	%		\leqconst J_{[]}(\delta, \F^*, L_2) \left(1 + \frac{J_{[]}(\delta, \F^*, L_2)M}{\delta^2 \sqrt{n}}\right)
	\end{align*}
	where $\F^* = \{ \phi^{\pi} - \phi^{\pi^*}: \pi \in \Pi \}$ and the bracketing entropy $J_{[]}(\phi_{\max}, \F^*, L_2) = \int_{0}^{\phi_{\max}} \sqrt{ \log N_{[]}(\epsilon, \F^*, L_2)} d\epsilon$. Using the Lipschitz property gives  
	\begin{align}
	J_{[]}(\phi_{\max}, \F^*, L_2) & \leq  \int_{0}^{\phi_{\max}} \sqrt{ \log N((2{K_1})^{-1} \epsilon, \Theta, \norm{\cdot}_2)} d\epsilon \notag\\
	& \leq \int_{0}^{\phi_{\max}} \sqrt{p \log \left(\frac{6K\text{diam}(\Theta)}{\epsilon}\right)} d\epsilon := {K_2} \sqrt{p} \label{entropy}
	\end{align}
	%		\[
	%		N(\epsilon, \Theta, \norm{\cdot}) \leq \left(\frac{3\text{diam}(\Theta)}{\epsilon}\right)^d
	%		\]
	We now apply the Talagrand's inequality. 
	\begin{align*}
	&  \sup_{\pi \in \Pi} ~ \Pn (\phi^{\pi} - \phi^{\pi^*})  \\
	& \leq \frac{1}{\sqrt{n}} \left[ {K_2} \sqrt{p}  + \sqrt{\frac{2\log(1/\delta)(4 \phi_{\max}^2 + 8 \phi_{\max}  {K_2}\sqrt{p})}{n}} + \frac{4\phi_{\max}\log(1/\delta)}{3n}   \right]\\
	& \leq {K_2} \sqrt{\frac{p}{n}} + \frac{ \phi_{\max}\sqrt{8\log(1/\delta) }}{n} +  \frac{4\sqrt{ \log(1/\delta) \phi_{\max}  {K_2}} p^{1/4}}{n} + \frac{4\phi_{\max}\log(1/\delta)}{3n^{3/2}} \\
	& \leq C_1(\delta) \left( p^{1/2} n^{-1/2}+ n^{-1} + p^{1/4}n^{-1} + n^{-3/2}  \right) 
	\end{align*}
	where $C_1(\delta) = {K_2} +  \phi_{\max}\sqrt{8\log(1/\delta) } + 4\sqrt{ \log(1/\delta) \phi_{\max}  {K_2}} + (4/3)\phi_{\max}\log(1/\delta)$.

	\textbf{(ii) Remainder Term} \quad
	For the ease of notation, define
	\[
	f(\omega, U, \pi):  D \mapsto  \frac{1}{T} \sum_{t=1}^T  \omega(S_t, A_t) (R_{t+1} + U(S_{t}, A_t, S_{t+1}) - \V^{\pi})
	\]
	Note that we have $\phi^{\pi} = f(\omega^\pi, U^\pi, \pi)$. Let $\hat \phi^\pi_n = f(\hat \omega^\pi_n, \Upin, \pi)$ be a ``plug-in'' estimator of $\phi^\pi$. Since the ratio estimator satisfies $\PP_n (1/T) \sum_{t=1}^T \hat \omega^\pi_n(S_t, A_t) = 1$ by construction, we have
	\begin{align*}
	& \Rem_n(\pi)  = \Vn^{\pi} - \V^{\pi} - \Pn \phi^{\pi} \\
	& = \frac{\Pn \dkh{(1/T)\sum_{t=1}^T \hat \omega_n^\pi(S_{t}, A_{t}) [R_{t+1} + \hat U^{\pi}_n(S_t, A_t, S_{t+1})]} }{\Pn \dkh{(1/T)\sum_{t=1}^T \hat \omega_n^\pi(S_{t}, A_{t})} }  -  \V^{\pi} - \Pn \phi^{\pi}  \\
	&= \Pn \dkh{(1/T)\sum_{t=1}^T \hat \omega_n^\pi(S_{t}, A_{t}) [R_{t+1} + \hat U^{\pi}_n(S_t, A_t, S_{t+1}) - \V^{\pi}]}  - \Pn \phi^{\pi}  \\
	& = \Pn (\hat \phi^\pi_n - \phi^\pi)    = (\Pn - P)(\hat \phi^\pi_n - \phi^\pi) + P (\hat \phi^\pi_n - \phi^\pi)
	\end{align*}
	This implies that $$\sup_{\pi \in \Pi} \abs{\Rem_n(\pi)} \leq \sup_{\pi \in \Pi} \abs{ P (\hat \phi^\pi_n - \phi^\pi)} + \sup_{\pi \in \Pi} \abs{(\Pn - P)(\hat \phi^\pi_n - \phi^\pi) } $$

	Consider the first term. The doubly-robustness structure of the efficient influence function, Lemma 3.2, implies that
	\begin{align*}
	& P(\hat \phi^\pi_n  - \phi^\pi)  = P ( f(\hat \omega^\pi_n, \Upin, \pi) -  f(\omega^\pi, U^\pi, \pi) ) \\
	& = P [f(\hat \omega^\pi_n, \Upin, \pi) - f(\hat \omega^\pi_n, \Upi, \pi)  +  f(\hat \omega^\pi_n, \Upi, \pi)   - f(\omega^\pi, U^\pi, \pi) ] \\
	& = P [f(\hat \omega^\pi_n, \Upi, \pi)   - f(\omega^\pi, U^\pi, \pi)]   + \Big(P [f(\hat \omega^\pi_n, \Upin, \pi) - f(\hat \omega^\pi_n, \Upi, \pi) ] \\
	& \qquad - P [f( \omega^\pi, \Upin, \pi) - f(\omega^\pi, \Upi, \pi)] \Big)+ P [f( \omega^\pi, \Upin, \pi) - f(\omega^\pi, \Upi, \pi)] \\
	& = \EE\Big[ (1/T) \sum_{t=1}^T (\hat \omega^\pi_n - \omega^\pi)(S_t, A_t) (R_{t+1} + \Upi(S_t, A_t, S_{t+1}) - \eta^\pi) \Big]  \\
	& \qquad + \EE\Big[ (1/T) \sum_{t=1}^T (\hat \omega^\pi_n - \omega^\pi)(S_t, A_t) \cdot (\hat U^\pi_n-U^\pi)(S_t, A_t, S_{t+1}) \Big]\\
	& \qquad + \EE\Big[ (1/T) \sum_{t=1}^T \omega^\pi(S_t, A_t) (\Upin - U^\pi)(S_{t}, A_{t}, S_{t+1}) \Big]\\
	& =\EE\Big[  (1/T) \sum_{t=1}^T   (\hat \omega^\pi_n - \omega^\pi)(S_t, A_t) \cdot (\Upin -U^\pi)(S_t, A_t, S_{t+1})\Big]
	\end{align*}
	where the last equality holds by noting  $\sum_{s, a} \EE[(\Upin - \Upi)(S_{t}, A_t, S_{t+1})|S_t=s, A_t=a] d^\pi(s, a) = 0$. 
	Furthermore, applying Cauchy inequality twice gives
	\begin{align*}
	& \abs{P( \hat \phi^\pi_n  - \phi^\pi)} =  \abs{(1/T) \sum_{t=1}^T	\EE[ (\hat \omega^\pi_n- \omega^\pi)(S_t, A_t) \cdot(\Upin-U^\pi)(S_t, A_t, S_{t+1}) ]} \\
	& \leq (1/T) \sum_{t=1}^T \sqrt{\EE\big[ (\hat \omega^\pi_n - \omega^\pi)^2(S_t, A_t)\big]} \cdot \sqrt{\EE[(\Upin-U^\pi)^2(S_t, A_t, S_{t+1})]}\\
	& \leq \sqrt{(1/T) \sum_{t=1}^T \EE[ (\hat \omega^\pi_n- \omega^\pi)^2(S_t, A_t)]} \cdot \sqrt{(1/T) \sum_{t=1}^T \EE[(\Upin-U^\pi)^2(S_t, A_t, S_{t+1})]} \\
	& = \norm{\hat \omega^\pi_n - \omega^\pi} \cdot \norm{\hat U^\pi_n - U^\pi} 
	\end{align*}	
	Using Theorem \ref{thm: ratio} and Theorem \ref{thm: value}, there exists constant $C_1(\delta)$ such that 
	\begin{align*}
	\sup_{\pi \in \Pi} \abs{P( \hat \phi^\pi_n  - \phi^\pi)} & \leq  \sup_{\pi \in \Pi}  \dkh{\norm{\hat \omega^\pi_n - \omega^\pi} \cdot \norm{\hat U^\pi_n - U^\pi}} \\
	& \leq (\sup_{\pi \in \Pi}  \norm{\hat \omega^\pi_n - \omega^\pi}) ( \sup_{\pi \in \Pi} \norm{\hat U^\pi_n - U^\pi} ) \\
	& \leq C_1(\delta) \iota^{\frac{1+\omega_k}{2}} p n^{-\frac{\beta_k}{2} - \frac{1}{2(1+\alpha)}}
	%O_P(n^{-\frac{1-\alpha_1\alpha_2}{(1+\alpha_1) (1+\alpha_2)}}) = o_P(1)
	\end{align*}
	Since $k > 2$, we have $\beta_k < 1/(1+\alpha)$ and $\omega_k < 1$. This implies that  
	$$	\sup_{\pi \in \Pi} \abs{P( \hat \phi^\pi_n  - \phi^\pi)}  \leq C_1(\delta) \iota p n^{-\beta_k} 
	$$

	Now we consider the second term. 	
	There exists $C_2(\delta)$, such that $\Pr(E_n) > 1- (6+k)\delta$ where $E_n = E_{n, 1} \jiao E_{n, 2}$ and 
	\begin{align*}
	& E_{n, 1} = \{ \norm{\hat \omega^\pi_n - w_\pi}^2 \leq  C_2(\delta) \iota^{\omega_k} n^{-\beta_k}, J_2(\hat e^\pi_n) \leq C_2(\delta) \iota^{\omega_k/2} n^{\frac{1}{2(1+\alpha)} - \frac{\beta_k}{2}}, \forall \pi \in \Pi\}\\
	& E_{n, 2} = \{ \norm{\Upin - \Upi}^2 \leq  C_2(\delta) \iota n^{-\frac{1}{1+\alpha}}, J_1(\hat Q_n^\pi) \leq C_2(\delta), \forall \pi \in \Pi\}
	\end{align*}
	Under this event $E_n$, we have
	\begin{align*}
	\sup_{\pi \in \Pi} \bigdkh{\abs{(\PP_{n} - P)( \hat \phi^\pi_n  - \phi^\pi) }} \leq \sup_{f \in \F^*, Pf^2 \leq \beta(n, \delta)}  \abs{ (\PP_{n} - P) f} 
	\end{align*}
	where $\beta(n, \delta) = 2(2(R_{\max}+2F_{\max}))^2 C_2(\delta) \iota^{\omega_k} n^{-\beta_k} + 2G_{\max} \left(\sup_{\pi} \norm{\omega^\pi}^2 \right)C_2(\delta) \iota n^{-\frac{1}{1+\alpha}} \leq C_3(\delta) \iota n^{-\beta_k}$ and $\F^* = \{ f: D \mapsto \frac{1}{T}\sum_{t=1}^T g(S_t, A_t, S_{t+1}) - \phi^\pi(D):   \pi \in \Pi, g \in \G^*\}$. Here $\G^*$ is given by
	\begin{align*}
	\G^* = \{& (s, a, s')   \mapsto  g(s, a) (\RR(s, a, s') + \sum_{a'} \pi(a'|s') Q(s', a') - Q(s, a) - \eta):\\
	&    g \in \G_{M_1}, Q \in \F_{M_2}, \pi \in \Pi, \eta \in [-R_{\max}, R_{\max}]\}
	\end{align*}
	where $M_1 = C_2(\delta) \iota^{\omega_k/2} n^{\frac{1}{2(1+\alpha)} - \frac{\beta_k}{2}}$ and $M_2 = C_2(\delta)$. 
	%	\ZL{Has some problem in the definition of $\G$}
	Applying a slightly modified version of Lemma \ref{lemma: donsker class prob bound} implies that for some constant $C_4(\delta)$, 
	\[
	\sup_{\pi \in \Pi} \bigdkh{\abs{(\PP_{n} - P)( \hat \phi^\pi_n  - \phi^\pi) }}  \leq  C_4(\delta)\iota^{\frac{1+\omega_k}{2}} p n^{-\frac{\beta_{k} + 1/(1+\alpha)}{2}} \leq  C_4(\delta)\iota p n^{-\beta_k} 
	\]
	Thus $\sup_{\pi \in \Pi} \abs{\Rem_n(\pi)}  \leq \left[C_1(\delta)+C_4(\delta)\right]\iota p n^{-\beta_k} $

\end{proof}

\begin{lemma} \label{lemma: lip}
	Define the state relative value function $V^\pi (s) = \sum_{a} \pi(a|s) Q^\pi(s, a)$. Under Assumption 1, (3-2) and (3-3), there exists constants $C_d, C_V$ that depend on only $\abs{\A}, L_{\Theta}, \beta, C_0, F_{\max}$ and $R_{\max}$, such that for any $\theta_1, \theta_2 \in \Theta$
	\begin{enumerate}
		\item $\norm{d^{\pi_{\theta_1}} - d^{\pi_{\theta_2}}}_{\tv} \leqconst  \sup_{s \in \S} \norm{P^{\pi_{\theta_1}}(\cdot|s) - P^{\pi_{\theta_2}}(\cdot|s)}_{\tv} \leq C_d\norm{\theta_1 - \theta_2}_2$
		\item $\sup_{s \in \S}\left|V^{\pi_{\theta_1}}(s) - \bar V^{\pi_{\theta_2}}(s)\right|\leq C_{V} \norm{\theta_1 - \theta_2}_2$  where $\bar V^{\pi_{\theta_2}} =  V^{\pi_{\theta_2}} - \mu^{\pi_{\theta_1}}(V^{\pi_{\theta_2}})$
	\end{enumerate}

\end{lemma}
\begin{proof}
	It can be seen that Assumption (3-3) implies geometrically ergodic. Then the first inequality in the first statement is given by Corollary 3.1 of \cite{mitrophanov2005sensitivity}, where the constant in this inequality can be chosen as  $\lceil \log_{\beta} C^{-1} \rceil + C_0 \frac{\beta^{\lceil \log_{\beta} C^{-1} \rceil}}{1-\beta}$. Furthermore, we can see that
	\begin{align*}
	&\sup_{s \in \S} \norm{P^{\pi_{\theta_1}}(\cdot|s) - P^{\pi_{\theta_2}}(\cdot|s)}_{\tv} \\[0.1in] =& \sup_{s \in \S} \frac{1}{2} \sum_{s' \in S} \left|\sum_{a \in \cal A}\pi_{\theta_1}(a | s) P(s' | s, a) - \sum_{a \in \cal A}\pi_{\theta_2}(a | s) P(s' | s, a) \right|\\[0.1in]
	\leq & \sup_{s \in \S} \frac{1}{2}\sum_{s' \in S} \sum_{a \in \cal A}L_{\Theta}\left\|\theta_1 - \theta_2\right\|_2 P(s' | s, a)  \leq  \frac{1}{2}\left|\cal A\right|L_{\Theta}\left\|\theta_1 - \theta_2\right\|_2, 
	\end{align*}
	where the first inequality uses Assumption (3-2) and $\left|\cal A\right|$ denotes the number of actions.
	
	Next we show the second statement of this lemma. From Bellman equation, we have
	\begin{align*}
	& (\I - \P^\pi) V^{\pi_{\theta_1}} (s) = r^{\pi_{\theta_1}}(s) - \eta^{\pi_{\theta_1}} \\
	& (\I - \P^\pi) \bar V^{\pi_{\theta_2}} (s) = r^{\pi_{\theta_2}}(s) - \eta^{\pi_{\theta_2}}
	\end{align*}
	where in the second equality we use the fact that the operator $I-P^{\pi}$ is invariant to a constant shift so that the $\bar V^{\pi_{\theta_2}}$ also solves Bellman equation.  As a result, we have
	\begin{align*}
	\Delta(s) & = {V^{\pi_{\theta_1}} (s) - V^{\pi_{\theta_2}} (s)  } \\
	& = {(\I - \P^{\pi_{\theta_1}}) V^{\pi_{\theta_1}} (s) - (\I - \P^{\pi_{\theta_2}}) V^{\pi_{\theta_2}} (s)} + {\P^{\pi_{\theta_1}} V^{\pi_{\theta_1}}(s) - \P^{\pi_{\theta_2}} V^{\pi_{\theta_2}}(s)} \\
	& = {r^{\pi_{\theta_1}}(s) - r^{\pi_{\theta_2}}(s)} - {\eta^{\pi_{\theta_1}} + \eta^{\pi_{\theta_2}}}  + (\P^{\pi_{\theta_1}} - \P^{\pi_{\theta_2}})  V^{\pi_{\theta_2}}(s) + \P^{\pi_{\theta_1}} \Delta(s) 
	\end{align*}
	On the other hand, it is straightforward to show that 
	\begin{align*}
	& \sup_{s} \abs{r^{\pi_{\theta_1}}(s) - r^{\pi_{\theta_2}}(s)} 
	\leq \abs{\A} L_{\Theta} R_{\max} \norm{\theta_1 - \theta_2}_2,
	\end{align*}
	\begin{align*}
	\left|\eta^{\pi_{\theta_1}} -\eta^{\pi_{\theta_2}}\right|
	\leq & ( 2C_d + L_{\Theta}) R_{\max}	\left|\cal A\right| \left\|\theta_1 - \theta_2 \right\|_2
	\end{align*}
	and $$\sup_{s} \abs{(\P^{\pi_{\theta_1}} - \P^{\pi_{\theta_2}})  V^{\pi_{\theta_2}}(s)} \leq \abs{\A} L_{\Theta} F_{\max} \norm{\theta_1 - \theta_2}_2. $$
	Thus for $C = \abs{\A} L_{\Theta} R_{\max} + ( 2C_d + L_{\Theta}) R_{\max}	\left|\cal A\right|+ \abs{\A} L_{\Theta} F_{\max}$, we have for any $s \in \S$, 
	\begin{align*}
	-C\norm{\theta_1-\theta_2}_2 + \P^{\pi_{\theta_1}} \Delta(s)  \leq \Delta(s) \leq C\norm{\theta_1-\theta_2}_2 + \P^{\pi_{\theta_1}} \Delta(s) 
	\end{align*}
	Now we can then bound $\sup_{s} \abs{V^{\pi_{\theta_1}} (s) - \bar V^{\pi_{\theta_2}} (s)}$ by 
	\begin{align*}
	\sup_{s} \abs{\Delta (s) } & \leq C\norm{\theta_1-\theta_2}_2 + \sup_{s} \abs{\P^{\pi_{\theta_1}} \Delta(s)} \\
	& = C\norm{\theta_1-\theta_2}_2 + \sup_{s} \Bigdkh{ \bigabs{\EE_{\pi_{\theta_1}}[\Delta(S_2)|S_1 = s]}} \\
	& \leq C\norm{\theta_1-\theta_2}_2 + \sup_{s} \Bigdkh{C\norm{\theta_1-\theta_2}_2 + \abs{\EE_{\pi_{\theta_1}}[\P^{\pi_{\theta_1}} \Delta(S_2)|S_1 = s]}}\\
	& = 2C\norm{\theta_1-\theta_2}_2 + \sup_{s} \Bigdkh{ \abs{\EE_{\pi_{\theta_1}}[\Delta(S_3)|S_1 = s]}}\\
	& \cdots \\
	& \leq k C\norm{\theta_1-\theta_2}_2 + \sup_{s} \Bigdkh{ \abs{\EE_{\pi_{\theta_1}}[\Delta(S_{k+1})|S_1 = s]}}
	\end{align*}
	where $k \geq 1$. Recall that $\mu^{\pi_{\theta_1}}(V^{\pi_{\theta_1}}) = 0$. By definition of $\bar V^{\pi_{\theta_2}}$, we have $\mu^{\pi_{\theta_1}} (\Delta) = 0$. Using Assumption (3-3), we have
	\begin{align*}
	& \sup_{s} \Bigdkh{ \abs{\EE_{\pi_{\theta_1}}[\Delta(S_{k+1})|S_1 = s]}}   = \sup_{s} \Bigdkh{ \abs{\EE_{\pi_{\theta_1}}[\Delta(S_{k+1})|S_1 = s] - \mu^{\pi_{\theta_1}} (\Delta)}} \\
	& \leq  \sup_{s} \bigdkh{\norm{\Delta}_{\infty}  \cdot 2 \cdot \norm{P^\pi\left(S_{k+1} = \cdot \given S_1 =s \right)  - d^{\pi_{\theta_1}}(\cdot)}_{\tv}}  \leq 2 {C}_0 {\beta}^{k+1} \norm{\Delta}_{\infty} 
	\end{align*}
	where $C_0$ and $\beta$ are the constants specified in Assumption (3-3).  Since $\beta < 1$, we can choose $k$ large enough so that $2 {C}_0 {\beta}^{k+1} < 1/2$.  We then have
	\begin{align*}
	\sup_{s} \abs{\Delta (s) } \leq \frac{kC}{1 - 2 {C}_0 {\beta}^{k+1}} \norm{\theta_1 - \theta_2} \leq 2kC \norm{\theta_1 - \theta_2}
	\end{align*}
	
\end{proof}

\section{Asymptotic result}
\label{appendix: asymptotic}
\begin{proof}[Proof of Theorem 5.2] 
	%1. Convergence to a GP. 
	
	Recall $\phi^\pi(D)$ is the efficient influence function. We denote the remainder term by $\Rem_n(\pi) = (\Vn^{\pi} - \V^{\pi}) - \Pn \phi^{\pi}$.  Thus we have
	\begin{align}
	\sqrt{n}(\Vn^{\pi} - \V^{\pi}) =  \Gn \phi^\pi + \sqrt{n} \Rem_n(\pi) \label{decomp}
	\end{align}
	
	In the proof of Theorem 5.1, we've shown that under the listed assumptions, there exists some constant $C(\delta)$, such that with probability at least $1-\delta$,  $\sup_{\pi \in \Pi} \abs{\Rem_n(\pi)}  \leq C(\delta) n^{-\beta_k}$. Recall that $\beta_{k} = \frac{1}{1+\alpha} \left\{1-(1-\alpha) 2^{-k+1}\right\}$ and $C(\delta)$ does not depend on $n$ (see the exact dependence in Theorem 5.1). Because $\beta_k$ is an increasing sequence and the limit $\lim_{k \goes \infty} \beta_k = 1/(1+\alpha) > 1/2$, we can choose $k$ such that $\beta_k > 1/2$ and thus 
	\begin{align*}
	\sqrt{n}\sup_{\pi \in \Pi} \abs{\Rem_n(\pi))} = O_P(n^{-\beta_k + 1/2}) = o_P(1) 
	\end{align*}
	We now show the first term in (\ref{decomp})  converges weakly to a Gaussian Process. In other words, the function class $\{\phi^\pi: \pi \in \Pi\}$ is a Donsker. As shown in (\ref{lip eif}) in the proof of Theorem 5.1, we have $\phi^\pi$ is Lipschitz, i.e.,  there exists constant $K$ such that for any trajectory $D$, $\abs{\phi^{\pi_{\theta_1}}(D) - \phi^{\pi_{\theta_2}}(D)} \leq K \norm{\theta_1 - \theta_2}$.  As a result, the bracketing entropy integral $J_{[]}(\delta, \{\phi_\pi, \pi \in \Pi\}, L_2(P))$ is finite (see (\ref{entropy}) for more details). The weakly convergence then follows by the classic Donsker Theorem (see for example Theorem 2.3 in \cite{kosorok2007introduction}).  Finally, applying the Slutsky’s theorem (Theorem 7.15 in \cite{kosorok2007introduction}) proves the first statement in the theorem. 
	
	Next we prove the second statement. Let $X_n(\theta) = \Vn(\pi_\theta)$ for $\theta \in \Theta$ and $X(\theta) = \V^{\pi_{\theta}}$.  From the first part of the theorem we have $
	\sqrt{n} (X_n - X) \wcvg \GG$ where $\GG$ is defined in the statement of the theorem.  For a continuous function on $\Pi$,  $f: \Theta \goes \R$, define $\psi(f) = \sup_{\theta \in \Theta} f(\theta)$. With these notations, $\sqrt{n} (\Vn(\hat \pi_n) - \sup_{\pi \in \Pi}  \V^{\pi}) = \sqrt{n} (\psi(X_n) - \psi(X))$. It can be shown the max-function $\psi$ satisfies directional differentiability in the Hadamard sense (see Theorem 5.7 in \cite{shapiro2014lectures} for a proof). 
	Applying the functional Delta method (see for example Theorem 2.8 in \cite{kosorok2007introduction}), we have $\sqrt{n} (\psi(X_n) - \psi(X)) \wcvg \psi(\GG) = \sup_{\pi \in \Pi_{\max}} \GG(\pi)$ as desired.

\end{proof}
\section{Further details of implementation in RKHS}
\label{appendix: implementation}

Below we provide the details of our computation in Section 6. To be complete, we start with our overall optimization problem.

\noindent \textit{Upper level optimization task}:	\begin{align}
& \max_{\pi \in \Pi} && \frac{\Pn \dkh{(1/T)\sum_{t=1}^T \hat \omega_n^\pi(S_{t}, A_{t}) [R_{t+1} + \hat U^{\pi}_n(S_t, A_t, S_{t+1})]} }{\Pn \dkh{(1/T)\sum_{t=1}^T \hat \omega_n^\pi(S_{t}, A_{t})} }\label{multi-level optimization app} 
\end{align}

\noindent \textit{Lower level optimization task 1}:\begin{align}
& && (\tilde \eta_n^\pi, \hat Q_n^\pi) = \argminb_{(\eta, Q) \in \R \times \F} \Pn \left[\frac{1}{T} \sum_{t=1}^{T} \left[\gn(S_t, A_t; \eta, Q)\right]^2\right] + \lambda_n J_1^2(Q)\label{lower level1 app}\\
& &&\text{s.t. } \; \; \gn(\cdot, \cdot; \eta, Q) = \argmin_{g \in \G} \Pn\Big[ \frac{1}{T}\sum_{t=1}^{T} \big( \delta^\pi(Z_t; \eta, Q) - g(S_t, A_t)\big)^2  \Big] + \mu_n J_2^2(g) \label{lower level2 app}
\end{align}

\noindent \textit{Lower level optimization task 2}:\begin{align}
& &&\qn(\cdot,\cdot) = \argminb_{H \in \F} \Pn\Big[\frac{1}{T}\sum_{t=1}^{T}  \left[\gn(S_t, A_t; H)\right]^2\Big] + \lambda_n J_1^2(H)\label{lower level3 app}\\[0.1in]
& && \text{s.t. } \, \, \gn(\cdot, \cdot; H) = \argminb_{g \in \G} \Pn\Big[\frac{1}{T} \sum_{t=1}^{T} \big(\Delta^\pi(Z_t; H)- g(S_t, A_t)\big)^2\Big]+  \mu_n J_2^2(g). \label{lower level4 app}
\end{align}

Following the main text, we rewrite the training data $D$ into tuples $Z_h = \{S_h, A_h, R_h, S_{h}'\}$ where $h = 1, \dots, N = nT$ indexes the tuple of transition sample in the training set $\D_n$,   $S_h$ and $S_h'$ are the current and next states and $R_h$ is the associated reward. Let $W_h = (S_h, A_h)$  be the state-action pair,  and $W_h' = (S_h, A_h, S_{h}')$. Here we do not consider baseline features. However, this can be readily  generalized. See \cite{liao2019off} for more details. Suppose the kernel function for the state is denoted by $k_0(s_1, s_2)$, where $s_1, s_2 \in \S$. In order to incorporate the action space, we can define $k((s_1, a_1), (s_2, a_2)) = \indicator{a_1 = a_2} k_0(s_1, s_2)$. Basically, we model each $Q(\cdot, a)$ separately for each arm in the RKHS with the same kernel $k_0$.  Recall that we have to restrict the function space $\F$ such that $Q(s^*, a^*) = 0$ for all $Q \in \F$ so as to avoid the identification issue.   Thus for any given kernel function $k$ defined on $\S \times \A$, we make the following transformation by defining $k(W_1, W_2) = k(W_1, W_2) - k(( s^*, a^*), W_2) - k(W_1, (s^*, a^*)) + k((s^*, a^*), (s^*, a^*))$. One can check that the induced RKHS by $k(\cdot, \cdot)$ satisfies the constraint in $\F$ automatically.    

We denote kernel functions for $\F$ and $\G$ by $k(\cdot, \cdot), l(\cdot, \cdot)$ respectively. The corresponding inner products are defined as $\innerprod{\cdot}{\cdot}_\F$ and $\innerprod{\cdot}{\cdot}_\G$. We first discuss the inner minimization problem (6.2)-(6.3). Note that this is indeed a standard kernel ridge regression problem. The closed form solution can be obtained as $\gn(\cdot, \cdot; \eta, Q) =  \sum_{h=1}^{N} l(W_h, \cdot) \hat{\gamma}(\eta, Q)$. In particular, $\hat{\gamma}(\eta, Q) = (L + \mu I_N)^{-1} \delta_N^{{\pi}}(\eta, Q)$, $L$ is the kernel matrix of $l$, $\mu = \mu_n N$, and $\delta^\pi_N(\eta, Q) = (\delta^\pi(Z_h; \eta, Q))_{h=1}^N$ is a vector of TD error.  Moreover, each TD error can be further written as $\delta^\pi(Z'; \eta, Q) = R - \eta- \innerprod{Q}{f_{W'}}_\G$ where 
\[
f_{W'_h}(\cdot) = k(W, \cdot) - \sum_{a'} \pi(a'|S') k((S', a'), \cdot) \in \F
\]
It can be checked that $\hat Q^\pi_n $ in (6.2) can be expressed by the linear span: $\dkh{\sum_{h=1}^N \alpha_h f_{W_h'}(\cdot):  \alpha_h \in \R, h = 1, \dots, N}$ according to the representer property.
Then by the representer theorem, we can solve the optimization problem (6.2)-(6.3) by equivalently computing
\begin{align}
\label{value estimation}
(\tilde \eta_n^\pi, \hat \alpha(\pi)) = \argminb_{\eta \in \R, \alpha \in \R^N} (R_N - \eta 1_N  - \tilde F(\pi) \alpha)^\transpose M (R_N - \eta 1_N  - \tilde F(\pi) \alpha) + \lambda \alpha^\transpose \tilde F(\pi) \alpha
\end{align} 
where $R_N = (R_h)_{h = 1}^N$, $\tilde F(\pi) = ( \innerprod{f_{W_h'}}{f_{W_{j}'}}_\F)_{j, h = 1}^N$, $M = (L + \mu I_N)^{-1}  L^2 (L + \mu I_N)^{-1}$, $1_N$ is a length-$N$ vector of all ones, $\lambda = \lambda_n N$ and $\alpha = (\alpha_h)_{h = 1}^N$ is a vector of length $N$. Note that $\tilde F(\pi)\left[h, k\right]$ can be further calculated as
\begin{align*}
\innerprod{f_{W_h'}}{f_{W_{j}'}}_\F  & = k(W_h, W_{j}) - \sum_{a'} \pi(a'|S_h') k((S_h', a'), W_j) - \sum_{a'} \pi(a'|S_j') k((S_j', a'), W_h)  \\
& \qquad + \sum_{a_h'} \sum_{a_{j}'} \pi(a_h'|S_h') \pi(a_j'|S_j') k((S_h', a_h'), (S_j', a_j')).
\end{align*}
We make $\tilde F(\pi)$ and $ \hat \alpha(\pi)$ as functions of $\pi$ to explicitly indicate their  dependency on the policy $\pi$. The first-order optimality implies that $(\tilde \eta_n^\pi, \hat \alpha(\pi))$ satisfies
\begin{align*}
1_N^\transpose M 1_N \tilde \eta_n^\pi  = 1_N^\transpose M(R_N - \tilde F(\pi) \hat \alpha(\pi)) \\
%(\tilde K M \tilde K  + \lambda \tilde K) \alpha = \tilde K M (R_n - F \beta) \iff 
( M \tilde F(\pi)  + \lambda I_N) \hat \alpha(\pi) =  M (R_N - 1_N \tilde \eta_n^\pi), 
\end{align*}
which gives 
\begin{align}
&\Big[M \tilde{F}(\pi)+\lambda I_N - M1_N(1_N^TM1_N)^{-1}1_N^TM\tilde{F}(\pi)\Big]\hat \alpha (\pi)\label{for gradient 0}\\
=& \Big(I_N - M1_N(1_N^TM1_N)^{-1}1_N^T\Big)MR_N
\label{for gradient 1}
\end{align}

and thus the corresponding $\{\hat{U}_n^\pi(W'_h)\}_{h = 1}^N =-\tilde{F}(\pi)  \hat \alpha(\pi)$. In order to apply gradient methods to obtain $\hat{\pi}_n$, or equivalently $\hat{\theta}$, we need to compute the Jacobian matrix of the vector $\{\hat{U}_n^\pi(W'_h)\}_{h = 1}^N$ with respect to $\theta$. Based on the above equations, we know
\begin{align*}
\frac{\partial \{\hat{U}_n^\pi(W'_h)\}_{h = 1}^N}{\partial \theta} = -\frac{\partial \tilde{F}(\pi) }{\partial \theta} \otimes \hat{\alpha}(\pi) - \tilde{F}(\pi) \frac{\partial \hat{\alpha}(\pi) }{\partial \theta},
\end{align*}
where $\otimes$ is denoted as a tensor product. Here $\frac{\partial \tilde{F}(\pi) }{\partial \theta}$ is a $\R^N \otimes \R^N \otimes \R^p$ tensor, where the $(i, j, k)$-th  element is the partial derivative $\frac{\partial \left[ \tilde{F}(\pi)\right]_{i, j}}{\partial \theta_k}$. In addition, $\frac{\partial \hat \alpha(\pi) }{\partial \theta}$ can be calculated via implicit theorem based on the equation \eqref{for gradient 0}-\eqref{for gradient 1}, i.e.,
\begin{align}
&\left(M \otimes \frac{\partial \tilde{F}(\pi) }{\partial \theta}+\lambda I_N - M1_N(1_N^TM1_N)^{-1}1_N^TM\otimes \frac{\partial \tilde{F}(\pi) }{\partial \theta}\right)\hat \alpha (\pi)\\
& = -(M \tilde{F}(\pi)+\lambda I_N - M1_N(1_N^TM1_N)^{-1}1_N^TM\tilde{F}(\pi))\frac{\partial \hat \alpha (\pi)}{\partial \theta},
\label{for gradient 2}
\end{align}
which gives the expression of $\frac{\partial \hat \alpha (\pi)}{\partial \theta}$, a $N$ by $p$ matrix.

Similarly, we can find the closed-form solution for the problem (6.3)-(6.4) and compute its gradient.  By some linear algebra, we can obtain $\{ \gn(W_h, \qn )\}_{h=1}^N = L\hat \mynu(\pi)$, where $\gn(W_h, \qn ) = \sum_{h=1}^N \hat\mynu_h(\pi) l(W_h, \cdot) $ and $\nu = (\hat\mynu_h(\pi))_{h=1}^N $ satisfying the following two equations:
\begin{align}
&(M \tilde F(\pi)  + \lambda I_N) \hat \varphi(\pi) &&=  M 1_N \label{for gradient 3}\\
&(L + \mu I_N) \hat \mynu (\pi)&& =  1_N - \tilde F(\pi) \hat \varphi(\pi) \label{for gradient 4},
\end{align}
given again by the representer theorem,
where $\hat{\varphi}(\pi)$ is an intermediate term.
We then can compute the Jacobian matrix of $\{ \gn(W_h, \qn )\}_{h=1}^N$ by again the implicit theorem using equations \eqref{for gradient 3} and \eqref{for gradient 4} and solving $\frac{\partial \hat \mynu(\pi)}{\partial \theta}$ based on the following two equations.
\begin{align}
&\left(M \otimes \frac{\partial \tilde{F}(\pi) }{\partial \theta}\right)  \hat \varphi(\pi) + \left(M \tilde F (\pi) + \lambda I_N\right) \frac{\partial \hat \varphi(\pi)}{\partial \theta}&&= 0 \label{for gradient 5}\\
&(L + \mu I_N) \frac{\partial \hat \mynu(\pi)}{\partial \theta} +  \tilde F (\pi) \frac{\partial \hat \varphi(\pi)}{\partial \theta} + \frac{\partial \tilde{F}(\pi) }{\partial \theta} \otimes \frac{\partial \hat \varphi(\pi)}{\partial \theta} && =  0 \label{for gradient 6}
\end{align}
Then we have
\begin{align*}
\frac{\partial \{ \gn(W_h, \qn )\}_{h=1}^N}{\partial \theta} = L \frac{\partial \hat \mynu(\pi)}{\partial \theta}
\end{align*}

Summarizing together by plugging all the intermediate results into the objective function of our upper optimization problem (6.1), we can simplify it as
\begin{align}
& \max_{\pi \in \Pi_{\Theta}} && \frac{\left(\hat \mynu (\pi)\right)^T L \left(R_N - \tilde F (\pi)\hat \alpha(\pi)\right)}{\hat \mynu (\pi)^T L 1_N}\label{final optimization app}.
\end{align}
The corresponding gradient with respect to $\theta$ can be computed directly as
\begin{align*}
&\frac{\left(\frac{\partial \hat \mynu(\pi)}{\partial \theta}\right)^T L \left(R_N - \tilde F (\pi)\hat \alpha(\pi)\right) - \left(\hat \mynu (\pi)\right)^T L \left(\frac{\partial \tilde{F}(\pi) }{\partial \theta} \otimes \hat \alpha(\pi) + \tilde{F}(\pi) \frac{\partial \hat \alpha(\pi) }{\partial \theta}\right) \left(\hat \mynu (\pi)^T L 1_N\right)}{\left(\hat \mynu (\pi)^T L 1_N\right)^2} \\[0.1in]
-&\frac{\left(\frac{\partial \hat \mynu(\pi)}{\partial \theta}\right)^T L \left(R_N - \tilde F (\pi)\hat \alpha(\pi)\right)\left(\hat \mynu (\pi)^T L 1_N\right)}{\left(\hat \mynu (\pi)^T L 1_N\right)^2}.
\end{align*}

\section{Additional Numerical Results}
	\textcolor{black}{
	In this section, we compare our proposed method with the three baseline methods via another simulation study. 
	The simulation setting is designed as the same as those in \cite{luckett2019estimating}. Specifically, we initialize two dimensional state vector $S_0 = (S_{0, 1}, S_{0, 2})$ by a standard multivariate Gaussian distribution. Given the current action $A_t \in \{0, 1\}$ and state $S_t$, the next state is generated by:
	\begin{align*}
	& S_{t+1, 1} = \frac{3}{4} (2A_t - 1) S_{t, 1} + \frac{1}{4}S_{t, 1}S_{t, 2} + \varepsilon_{t, 1}, \\
	&S_{t+1, 2} = \frac{3}{4} ( 1 - 2A_t) S_{t, 2} + \frac{1}{4}S_{t, 1}S_{t, 2} + \varepsilon_{t, 2},
	\end{align*}
	where each $\varepsilon_{t, j}$ follows independently $N(0, 1/4)$ for $j = 1, 2$. The reward function $R_{t+1}$ is given as
	\[
	R_{t+1} = 2S_{t+1, 1} + S_{t+1, 2} - \frac{1}{4}(2A_t - 1),
	\] 
	for $t = 1, \cdots, T$.
	We consider the behavior policy to be uniformly random, i.e., choosing each action with equal probability.}

	We consider different combinations of the number of trajectories $n$ and the length of each trajectory $T$ to evaluate the performance of our method. Specifically, we consider $n=25, 50$ and $T = 24, 48$, and replicate each setting for 128 times. To calculate the regret of our learned policy, basically we consider different policy parameters $\theta$, with the value of each dimension of $\theta$ ranging from $-10$ to $10$. For each of these policies, we generate one trajectory with length $10000$ following the corresponding policy, discard the first 5000 time points and take the average of the remaining rewards. Assuming achieving stationary distribution after $T=5000$, we use the largest average rewards among these policies as our optimal in-class average reward. Using a similar procedure, we can also compute the average reward of each of the learned learned policy under different settings. The regrets can be obtained by subtracting them from the optimal in-class average reward, which are provided in Table \ref{Tab: regret sim setting1}. In Table \ref{Tab: sim setting1}, we report mean and standard error of the average rewards of policies given by our method, BEAR, BCQ and FQI respectively. It can be seen that all reported regrets (or average rewards) of our method are the smallest (or the largest), indicating that our method can learn desirable in-class policies, compared with other three methods. 
	\begin{table}[H]
		\centering
		\caption{The regrets of our estimated policy and three offline RL algorithms.  Numbers in parentheses are the standard deviations of the regrets over 128 replications.} 
		\begin{tabular}{cc|cccc} % |cc|c
			\cline{1-6}
			$n$ & $T$ & Our method & BEAR & BCQ &	FQI \\\cline{1-6}
			25 & 24  & $0.027 \, (0.003)$ & 0.774 (0.060) & 0.625 (0.034) & 0.048 (0.032) \\
            50 & 24  & $0.017 \, (0.002)$ & 0.714 (0.056) & 0.625 (0.034) & 0.053 (0.032)\\
			25 & 48  & $0.012 \, (0.007)$ & 0.797 (0.053) & 0.658 (0.024) & 0.061 (0.029)\\
			50 & 48  & $0.009 \, (0.003)$ & 0.843 (0.060) & 0.652 (0.025) & 0.054 (0.027)\\
			\cline{1-6}
		\end{tabular}
		\label{Tab: regret sim setting1}
	\end{table}

	\begin{table}[h!]
		\centering
		\caption{Monte Carlo estimation of the average rewards of our learned policies and three offline RL algorithms. Numbers in parentheses are the standard deviations of the average rewards over 128 replications.}
		\begin{tabular}{cc|cccc}
			\hline
			$n$ & $T$ & Our method &  BEAR & BCQ &	FQI\\ 
			\hline
			25 & 24 & $0.898 \, (0.003)$  & 0.119 (0.044) & 0.269 (0.001) & 0.832 (0.007)  \\ 
			25 & 48 & $0.900 \, (0.007)$  & 0.153 (0.044) & 0.269 (0.001) & 0.849 (0.010)  \\ 
			\hline
			50 & 24 & $0.913 \, (0.002)$  & 0.109 (0.048) & 0.276 (0.001) & 0.866 (0.014)   \\ 
			50 & 48 & $0.914 \, (0.003)$  & 0.121 (0.046) & 0.276 (0.001) & 0.853 (0.023)  \\ 
			\hline
		\end{tabular}
		
		\label{Tab: sim setting1}
	\end{table}

		\bibliographystyle{imsart-nameyear}
	\bibliography{reference}
	
\end{document}